\documentclass[11pt]{amsart}

\usepackage{cgeometry}
\usepackage{cleveref}
\usepackage{mathtools} 

\bibliography{ComplexGeometry}
\newcommand\blfootnote[1]{%
  \begingroup
  \renewcommand\thefootnote{}\footnote{#1}%
  \addtocounter{footnote}{-1}%
  \endgroup
}

\begin{document}

\title{Mabuchi geometry of big cohomology classes with prescribed singularities}
\author{Mingchen Xia}


\begin{abstract}
Let $X$ be a compact Kähler unibranch  complex analytic space of pure dimension. 
Fix a big class $\alpha$ with smooth representative $\theta$ and a model potential $\phi$ with positive mass.
We define and study the non-pluripolar products of quasi-plurisubharmonic functions on $X$. 
We study the spaces $\mathcal{E}^p(X,\theta;[\phi])$ of finite energy Kähler potentials with prescribed singularities for each $p\geq 1$. We define a metric $d_p$ without solving Monge--Amp\`ere equations. This construction generalizes the usual $d_p$-metric defined for an ample class. 
\end{abstract}

\maketitle


\tableofcontents


\section{Introduction}
\subsection{Motivation}
Let $X$ be a compact Kähler manifold of pure dimension $n$ and $\omega$ be a Kähler form on $X$. It is well-known that the space $\mathcal{H}(X,\omega)$ of smooth strictly $\omega$-plurisubharmonic functions admits a natural Riemannian structure:
\[
\langle f,g \rangle_{\varphi}\coloneqq \int_X fg\,\omega_{\varphi}^n\,,\quad \varphi\in \mathcal{H}(X,\omega)\,,f,g\in C^{\infty}(X)=\mathrm{T}_{\varphi}\mathcal{H}(X,\omega)\,.
\]
It is shown by Chen (\cite{Chen00}) that the Riemannian structure endows $\mathcal{H}(X,\omega)$ with a \emph{bona fide} metric $d_2$.
Darvas (\cite{Dar17}, \cite{Dar15}) proved that the metric completion of $\mathcal{H}(X,\omega)$ with respect to the Riemannian metric can be naturally identified with the space $\mathcal{E}^2(X,\omega)$ of $\omega$-psh functions with finite energy, confirming a conjecture of Guedj. In fact, the results of Darvas show that the metric completion of $\mathcal{H}(X,\omega)$ with respect to the natural $p$-Finsler metric ($p\geq 1$) is given by $\mathcal{E}^p(X,\omega)$, namely, the space of $\omega$-psh functions with finite $p$-energy. The space $\mathcal{E}^p(X,\omega)$ is a geodesic metric space.
In the non-Archimedean world, one can similarly define a space $\mathcal{E}^1$, which enjoys similar properties as the Archimedean counterpart, see \cite{BJ18b}.

The spaces $\mathcal{E}^1$ and $\mathcal{E}^2$ have found numerous applications in the study of Kähler geometry and in non-Archimedean geometry, especially in the problem of canonical metrics. See \cite{CC17a}, \cite{BDL16},  \cite{LTW19}, \cite{BJ18} for  a few examples. 

The theory of $\mathcal{E}^p$ together with the metric $d_p$ is, however, not completely satisfactory for the following reasons:
\begin{enumerate}
    \item $\mathcal{E}^p$ accounts for only a small portion of the space of plurisubharmonic functions, which limits its application to general singularities.
    \item Although $\mathcal{E}^p$ can be defined for a general big cohomology class, it is not clear how to define a metric on $\mathcal{E}^p$. In particular, this makes it hard to apply them in the study of birational geometry.
    \item The spaces $\mathcal{E}^p$ ($p>1$) are only defined when $X$ is smooth.
    In particular, it is not powerful enough to deal with Kähler--Einstein metrics on normal varieties directly.
\end{enumerate}
The problem (1) is effectively solved by Darvas--Di Nezza--Lu in \cite{DDNL18mono} by introducing the so-called potentials with \emph{prescribed singularities}. Roughly speaking, this means that we prescribe a \emph{nice} singularity type $\phi$, then we can define a space $\mathcal{E}(X,\omega;[\phi])$ of potentials with singularities controlled by $\phi$. It can be shown that these spaces for various $\phi$ give a partition of $\PSH(X,\omega)$. Moreover, each space $\mathcal{E}(X,\omega;[\phi])$ behaves exactly like $\mathcal{E}(X,\omega)$. So we have turned the study of general potentials into the study of potentials with \emph{relatively} full mass.

As for problem (2), so far only the following partial generalizations are known:
\begin{enumerate}
    \item When $p=1$, the metric on $\mathcal{E}^1$ is defined for a general big class in \cite{DDNL18big}. A further generalization to potentials with prescribed singularities is established in \cite{Tru20}.
    \item When the class is big and nef, the metric on $\mathcal{E}^p$ is defined in \cite{DNL20} by perturbation.
    \item On a general normal compact complex analytic space, in \cite{BBEGZ16}, the space $\mathcal{E}^1$ with respect to a K\"ahler class is studied.
\end{enumerate}
In particular, there are no satisfactory theory of $\mathcal{E}^p$ for a general big cohomology class yet.

As for problem (3), this is in fact just technical. The pluripotential theory has not been fully developed on a general normal variety, except in the case of bounded potential (\cite{Dem85}). In particular, the theory of non-pluripolar products (\cite{BEGZ10}) has never been fully developed on singular spaces. 
It seems that this theory is well-known to experts. And in fact, it has been widely applied in various circumstances.

The goal of this paper is to solve all of these problems simultaneously.

\subsection{Main results}

\subsubsection{Non-pluripolar products}
The first part of this paper is devoted to develop the theory of non-pluripolar products on a singular space.  

Let $X$ be a unibranch (see \cref{def:uni}) complex analytic space of pure dimension $n$ and $X^{\Red}$ be the reduced space underlying $X$. Recall that there is a well-defined notion of psh functions on $X$ (\cite{FN80}).
Let $\pi:Y\rightarrow X^{\Red}$ be a proper resolution of singularities. As in the algebraic setting, the unibranchness assumption ensures that the fibres of $\pi$ are connected, a result classically known as \emph{Zariski's main theorem}, see \cref{thm:Zmt}. In particular, we prove that $\pi^*$ gives a bijection between the spaces $\PSH(X)=\PSH(X^{\Red})$ and $\PSH(Y)$. Therefore, a large part the study of $\PSH(X)$ can be essentially reduced to the known results in the smooth setting. 

We prove in \cref{thm:Lelongconj} the following generalization of Lelong's conjecture:
\begin{theorem}
Assume that $X$ is an open subspace of a compact unibranch Kähler space.
Let $E\subseteq X$ be a subset. Then $E$ is negligible if and only if  $E$ is pluripolar.
\end{theorem}
 This result is in fact equivalent to the so-called continuity of envelope property. To the best of the author's knowledge this is the first written proof in the singular setting, although this theorem has been applied implicitly in the literature for a long time. As a consequence, we show that the Bedford--Taylor products are local with respect to the pluri-fine topology in \cref{thm:pluloc}, which allows one to introduce the non-pluripolar products.

\subsubsection{Pluripotential theory on singular spaces}
In the second part, we study the space $\mathcal{E}^p$.
Let $X$ be a compact unibranch Kähler space of pure dimension $n$ and $\alpha$ be a big class on $X$ represented by a strongly smooth form $\theta$.
Let $\phi$ be a model potential (\cref{def:mod}) in $\PSH(X,\theta)$ with positive mass. Then we can define the space $\Ep$ exactly as in the smooth case:
\[
\Ep\coloneqq \left\{\,\varphi\in \PSH(X,\theta):\varphi \leq \phi+C, \int_X \theta_{\varphi}^n=\int_X \theta_{\phi}^n, \int_X |\phi-\varphi|^p\theta_{\varphi}^n<\infty \,\right\}\,,
\]
see \cref{def:ep} for the details. We  generalize various known results in the smooth setting to the singular setting. In particular, we establish a general comparison principle (\cref{lma:comp}) and several energy estimates. We show that the rooftop operator is well-defined and has the expected support:
\begin{theorem}[={\cref{thm:rooftopep}}+{\cref{cor:rooftop}}]
Let $\varphi,\psi\in \Ep$, then
\[
\varphi\land \psi\coloneqq \sups\{\eta\in \PSH(X,\theta):\eta\leq \varphi,\eta\leq \psi\}\in \Ep\,.
\]
Moreover, $\theta_{\varphi\land \psi}^n$ is supported on the set $\{\varphi\land\psi=\varphi\}\cup \{\varphi\land\psi=\psi\}$.
\end{theorem}

\subsubsection{Metrics on $\mathcal{E}^p$}
The main result in this part is the definition of a metric $d_p$. This part is the main innovation in this paper.
To demonstrate the idea, let us consider the case of a compact Kähler manifold $X$. Let $\omega$ be a Kähler form on $X$. Let us see how the usual $d_p$-metric can be defined without solving Monge--Amp\`ere equations. Let $\varphi_0,\varphi_1\in \mathcal{H}(X,\omega)$. Recall that solving a suitable Monge--Amp\`ere equation gives us a unique geodesic $(\varphi_t)_{t\in [0,1]}$, see \cite{Chen00}. The geodesic has $C^{1,1}$-regularity, see \cite{CTW17}. In this case, we have
\[
d_p(\varphi_0,\varphi_1)^p=\int_X |\dot{\varphi}_0|^p\,\omega_{\varphi_0}^n\,.
\] 
Recall that $d_p$ satisfies a so-called Pythagorean equality, which allows us to assume $\varphi_0\leq \varphi_1$.
It is shown in \cite{Dar15} that $d_p(\varphi_0,\varphi_1)$ is of the order $\left(\int_X (\varphi_1-\varphi_0)^p\,\theta_{\varphi_0}^n\right)^{1/p}$. We could reverse this machinery by \emph{defining}
\[
d_p(\varphi_0,\varphi_1)\approx \left(\int_X (\varphi_1-\varphi_0)^p\,\theta_{\varphi_0}^n\right)^{1/p}\,,
\]
when $\varphi_1\geq \varphi_0$ is sufficiently close to $\varphi_0$. This turns out to work and the resulting metric is equal to the original $d_p$, see \cref{subsubsec:relamp}.
Note that this definition is purely pluripotential-theoretic, hence does not require to solve any PDE. The same definition works in general. This gives our definition of $d_p$.

The second main result in this part is that the space $\Ep$ is locally complete with respect to $d_p$. Here locally completeness means that subspaces of the form $\{\psi\in \Ep: \psi\geq \varphi\}$ are complete where $\varphi\in \Ep$.
In our proof, we found that an algebraic structure, which we call the \emph{rooftop structure}, plays a key role.
So we develop this algebraic structure in detail. Roughly speaking, a rooftop structure on a metric space is an associative, commutative, idempotent binary structure $\land$ such that $\bullet\land\bullet$ is distance-decreasing in both variables. We prove a general criterion for the completeness of a rooftop metric space (\cref{prop:rooftopcomp}).
Back to $\Ep$, the rooftop structure $\land$ is defined as follows: given $\varphi,\psi\in \Ep$, we let $\varphi\land \psi$ be the maximal $\theta$-psh function lying below both $\varphi$ and $\psi$. We show that this rooftop structure verifies the conditions in \cref{prop:rooftopcomp} locally. Thus, $d_p$ is a locally complete metric on $\Ep$. In other words,
\begin{theorem}[={\cref{thm:main}}]
\label{thm:main1}
The space $(\Ep,d_p,\land)$ is a $p$-strict locally complete rooftop metric space.
\end{theorem}
See \cref{def:roof} for the definition of $p$-strictness. We will prove in \cref{subsec:lit} that our definition generalizes all known definitions in the literature.

We have developed our theory in an axiomatic manner. The reason is that these results can be formally generalized to the non-Archimedean setting with essentially the same proofs as long as the conjecture of continuity of envelopes holds (\cite[Conjecture~4.51]{BJ18b}). 
The details will appear elsewhere.

\subsubsection{The space of geodesic rays}
In the last part, based on our previous results, we study the space $\mathcal{R}^1(X,\theta)$ of finite energy geodesic rays. The study of such spaces as metric spaces is initiated by Chen--Cheng (\cite{CC18c}) and Darvas--Lu (\cite{DL18}). In the case of a big class, $\mathcal{R}^1$ is studied in \cite{DDNL19}. We generalize these results to the singular setting as well. Similar to the case of potentials, we construct a natural metric $d_1$ and a rooftop structure $\land$ on $\mathcal{R}^1$. We prove the following theorem
\begin{theorem}[={\cref{thm:R1comproof}}]\label{thm:main2}
The space $(\mathcal{R}^1(X,\theta),d_1,\land)$ is a $1$-strict complete rooftop metric space. 
\end{theorem}

The importance of this theorem lies in its relation to the conjecture of continuity of envelopes (\cite[Conjecture~4.51]{BJ18b}) in non-Archimedean geometry. 
Suppose that $X$ be a unibranch projective complex variety and $L$ is an ample line bundle on $X$. 
Recall that the continuity of envelopes is equivalent to the completeness of $\mathcal{E}^{1,\NA}(L)$ (\cite[Theorem~9.8]{BJ18b}).
At least when $X$ is smooth, it is known that $\mathcal{E}^{1,\NA}(L)$ can be identified with a closed subspace of $\mathcal{R}^1(X,\theta)$ as in the proof of \cref{ex:E1NA} and the completeness follows from \cref{thm:main2}. We expect that these results extend to unibranch spaces as well.

\subsection{Structure of the paper}
In \cref{sec:cas}, we recall some basic properties of complex analytic spaces.

In \cref{sec:npp}, we develop the theory of non-pluripolar products on unibranch spaces.

In \cref{sec:pps}, we develop the theory of potentials with prescribed singularities in detail, generalizing a number of results in the literature. In particular, we carry out a detailed study of some energy functionals.

In \cref{sec:rooftop}, we study abstract rooftop structures. 

In \cref{sec:metep}, we define the $d_p$ metric on $\mathcal{E}^p$.

In \cref{sec:geod}, we study the space $\mathcal{R}^1$ of geodesic rays.

\subsection{Conventions}
All Monge--Amp\`ere operators are taken in the non-pluripolar sense.
We use the terms \emph{increasing}, \emph{decreasing} in the non-strict sense. We write $\land$ for the rooftop operator instead of the more common $P(\cdot,\cdot)$. For $\varphi,\psi\in\Ep$, we write $[\varphi]\land\psi$ for $P[\varphi](\psi)$. Given two qpsh functions $\varphi,\psi$, we write $\varphi\lor\psi=\max\{\varphi,\psi\}$.
We follow the convention
\[
\ddc=\frac{\mathrm{i}}{2\pi}\partial \bp\,.
\]
We use $C$ for a positive constant, whose value may change from line to line.

\textbf{Acknowledgements}
I benefited from discussions with Antonio Trusiani and Jie Liu. I would like to thank the referee for pointing out a mistake and for numerous helpful  suggestions.

The author is supported by Knut och Alice Wallenbergs Stiftelse grant KAW 2021.0231.

\section{Preliminaries on complex analytic spaces}\label{sec:cas}
For each $N\geq 0$, we endow an open subset $\Omega\subseteq \mathbb{C}^N$ with the sheaf of holomorphic functions, so that $\Omega$ can be regarded as a locally $\mathbb{C}$-ringed space.

\subsection{Complex analytic spaces}
In this section, we recall some basic facts about complex analytic spaces. This section is by no means intended to be complete. 
For the details, we refer to \cite{Fis06}, \cite{CAS}.

\begin{definition}
	A \emph{local model} of a complex analytic space is a locally $\mathbb{C}$-ringed space $(V,\mathcal{O}_V)$ such that there is a closed immersion $V\hookrightarrow \Omega$ into a bounded pseudo-convex domain $\Omega$ in some $\mathbb{C}^N$. To be more precise, this means that there are $f_1,\ldots,f_M\in H^0(\Omega,\mathcal{O}_{\Omega})$ such that $V$ is closed subset of $\Omega$ defined as the common zero locus of all $f_i$'s, $\mathcal{O}_V$ is the quotient of $\mathcal{O}_{\Omega}$ by the ideal $(f_1,\ldots,f_M)$, regarded as a sheaf on $V$.
	
	The category of local models is a full subcategory of the category of locally $\mathbb{C}$-ringed spaces.
\end{definition}
Here requiring pseudo-convexity of $\Omega$ is just for convenience. The notion of local models does not change if we remove this requirement.

\begin{definition}\label{def:cas}
	A \emph{complex analytic space} is a locally $\mathbb{C}$-ringed space $(X,\mathcal{O}_X)$ such that
	\begin{enumerate}
		\item $X$ is a paracompact Hausdorff space.
		\item For any $x\in X$, there is an open neighbourhood $U\subseteq X$ of $x$ such that $(U,\mathcal{O}_U)$ (with $\mathcal{O}_U$ being the restriction of $\mathcal{O}_X$ to $U$) is isomorphic to a local model as locally $\mathbb{C}$-ringed spaces.
	\end{enumerate}
	The category is complex analytic spaces is a full subcategory of the category of locally $\mathbb{C}$-ringed spaces.
\end{definition}
By abuse of language, we say that $X$ is a complex analytic space as well. 

In the literature, some authors prefer to weaken Condition~(1). We choose to include both paracompactness and separateness in our definition so that \cref{thm:Steinred} holds without further assumptions.

\begin{definition}\label{def:uni}
	A complex analytic space $(X,\mathcal{O}_X)$ is 
	\begin{enumerate}
		\item \emph{reduced} (resp. \emph{normal}) at $x\in X$ if $\mathcal{O}_{X,x}$ is reduced (resp. normal). We say $X$ is \emph{reduced} (resp. \emph{normal}) if it is reduced (resp. normal) at all points;
		\item \emph{unibranch} at $x\in X$ if  $\mathcal{O}_{X,x}$ is unibranch. We say $X$ is \emph{unibranch} if it is unibranch at all points.
	\end{enumerate}
\end{definition}
Recall that the notion of a unibranch local ring is defined in \cite[Section~0.23.2.1]{EGA4-1}: A local ring $A$ is unibranch is $A^{\Red}$ is integral and the integral closure of $A^{\Red}$ in its fraction field is local.

\begin{remark}
	For us, a ring is normal if it is an integral domain and integrally closed. So in particular, a normal analytic space is reduced. Also, recall that a normal ring is always unibranch.
\end{remark}

\begin{proposition}\label{prop-unibranchchar}
    Let $X$ be a complex analytic space and $x\in X$. Then the following are equivalent:
    \begin{enumerate}
        \item $X$ is unibranch at $x$;
        \item $X^{\Red}$ is unibranch at $x$;
        \item $\mathcal{O}_{X,x}$ is geometrically unibranch (see \cite[Section~0.23.2.1]{EGA4-1});
        \item $\mathcal{O}_{X,x}^{\Red}$ is geometrically unibranch;
        \item $\mathcal{O}_{X,x}$ has a unique minimal prime ideal.
    \end{enumerate}
\end{proposition}
\begin{proof}
    (1) $\Leftrightarrow$ (3): As $\mathcal{O}_{X,x}$ is excellent \cite{DG67}, the integral closure $\overline{\mathcal{O}_{X,x}^{\Red}}$ is a finite $\mathcal{O}_{X,x}^{\Red}$-algebra, so the residue field extension is finite. But the residue field of $\mathcal{O}_{X,x}$ is $\mathbb{C}$, so the residue field extension is the trivial extension.

    (1) $\Leftrightarrow$ (5): This follows from \cite[\href{https://stacks.math.columbia.edu/tag/0BQ0}{Tag 0BQ0}]{stacks-project} and the fact that $\mathcal{O}_{X,x}$ is Henselian (\cite[Page~45]{CAS}).

    (1) $\Leftrightarrow$ (2): This follows from the observation that (5) holds for $\mathcal{O}_{X,x}$ if and only if (5) holds for $\mathcal{O}_{X,x}^{\Red}$. 

    (3) $\Leftrightarrow$ (4): This follows from the same argument as (1) $\Leftrightarrow$ (2).
\end{proof}

\begin{remark}
	Note that our definition of unibranch space is different from the notion of locally irreducible space in \cite[Page~8]{CAS}. More precisely, a complex analytic space $X$ is unibranch in our sense if and only if $X^{\Red}$ is locally irreducible in the sense of \cite{CAS}.
\end{remark}

\begin{remark}
	For complex varieties, one could also define unibranchness using the Zariski topology. This definition is equivalent to ours.
	
	To be more precise, let $X$ be a reduced scheme of finite type over $\mathbb{C}$. Let $X^{\An}$ be the complex analytification of $X$. Let $x\in X(\mathbb{C})$. Then we claim that $X$ is unibranch at $x$ if and only if  $X^{\An}$ is unibranch at $x$. By GAGA, there is a natural morphism of ringed spaces $(X^{\An},\mathcal{O}_{X^{\An}})\rightarrow (X,\mathcal{O}_X)$.
	 Then since $\mathcal{O}_{X,x}$ is excellent, $X$ is unibranch at $x$ if and only if  $\widehat{\mathcal{O}_{X,x}}=\widehat{\mathcal{O}_{X^{\An},x}}$ is integral (\cite[Théorème~18.9.1]{EGAIV-4}). The latter implies that $\mathcal{O}_{X^{\An},x}$ is unibranch by \cite[Chapitre~5, Exercise~2.8]{AC}. Running the same argument the other way round, we conclude that $\mathcal{O}_{X,x}$ is unibranch if and only if  $\mathcal{O}_{X^{\An},x}$ is.
\end{remark}

\begin{definition}
	A complex analytic space $(X,\mathcal{O}_X)$ is 
	\begin{enumerate}
		\item \emph{Holomorphically separable} if for any $x,y\in X$, $x\neq y$, there is $f\in H^0(X,\mathcal{O}_X)$ such that $f(x)\neq f(y)$.
		\item \emph{Holomorphically convex} if for any compact set $K\subseteq X$, the set
		\[
		\hat{K}\coloneqq \left\{\,x\in X: |f(x)|\leq \sup_{k\in K} |f(k)|\,,\quad \forall f\in H^0(X,\mathcal{O}_X)\,\right\}
		\]
		is compact.
		\item \emph{Stein} if it is both holomorphically separable and holomorphically convex.
	\end{enumerate}
\end{definition}
Note that all three conditions are stable under passing to closed subspaces.

\begin{theorem}\label{thm:Steinred}
	Let $(X,\mathcal{O}_X)$ be a complex analytic space. Then $X$ is Stein if and only if  $X^{\Red}$ is.
\end{theorem}
This highly non-trivial result is proved by Grauert in \cite{Gra60}. It is based on a cohomological characterization of Stein spaces. See \cite{GR77} for a simplified proof. We remark that it is important to assume that $X$ is paracompact and Hausdorff for this theorem.
As a particular case of this theorem, consider a first-order thickening of complex analytic spaces $T\hookrightarrow T'$, namely, $T$ is the closed subspace of $T'$ defined by a coherent ideal sheaf with square $0$, then $T$ is Stein if and only if $T'$ is. If these equivalent conditions hold, we say that $T\hookrightarrow T'$ is a first-order thickening of Stein spaces.

\begin{definition}\label{def:forsm}
	Let $f:X\rightarrow S$ be a morphism of complex analytic spaces.
	We say the morphism $f$ is \emph{formally smooth} if given any solid commutative diagram
	\[
	\begin{tikzcd}
		T \arrow[r, "a"] \arrow[d, "i"]
		& X \arrow[d, "f" ] \\
		T' \arrow[ru, dashrightarrow]\arrow[r] & S
	\end{tikzcd}\,,
	\]
	where $i:T\rightarrow T'$ is a first-order thickening of Stein spaces, there is a dotted morphism $T'\rightarrow X$ making the whole diagram commutative.
	
	A complex analytic space $X$ is \emph{formally smooth} if the morphism to the final object $\mathbb{C}^0$ is formally smooth.
\end{definition}

\begin{theorem}\label{thm:smoothimpform}
	Let $f:X\rightarrow S$ be a morphism of complex analytic spaces. Assume that $f$ is smooth, then $f$ is formally smooth.
\end{theorem}
The proof is a simple extension of its algebraic analog \cite[\href{https://stacks.math.columbia.edu/tag/02H6}{Tag 02H6}]{stacks-project}. We refer to Grothendieck's exposé \cite{SHC~7} for the notion of $\Omega^1_{X/S}$. 
\begin{proof}
	Suppose that we are given a solid diagram as in \cref{def:forsm}. 
	
	Consider the sheaf $\mathcal{F}$ of sets over $T'$:
	\[
	H^0(U',\mathcal{F})\coloneqq \left\{\,a':U'\rightarrow X: a'|_U=a|_U\,\right\}\,,\quad U=U'\cap T
	\]
	for any open set $U'\subseteq T'$. 
 We want to show that $\mathcal{F}$ admits a global section on $T'$. Let
	\[
	\mathcal{H}\coloneqq \HHom_{\mathcal{O}_T}(a^*\Omega^1_{X/S},\mathcal{C}_{T/T'})\,,
	\]
	where $\mathcal{C}_{T/T'}$ is the conormal sheaf of $T$ in $T'$, namely, if $T$ is defined by a coherent ideal sheaf $\mathcal{I}\subseteq \mathcal{O}_{T'}$, then $\mathcal{C}_{T/T'}$ is $\mathcal{I}/\mathcal{I}^2$ regarded as a sheaf of $\mathcal{O}_T$-modules. There is an obvious action of $\mathcal{H}$ on $\mathcal{F}$, making $\mathcal{F}$ a pseudo-$\mathcal{H}$-torsor. We will show that $\mathcal{F}$ is a trivial $\mathcal{H}$-torsor.

 We first show that $\mathcal{F}$ is an $\mathcal{H}$-torsor. Namely, all fibers of $\mathcal{F}$ are non-empty. Let $t\in T$. Let $x=a(t)$, $s=f(x)$, $t'=i(t)$. We know that $f^{\sharp}:\mathcal{O}_{S,s}\rightarrow \mathcal{O}_{X,x}$ is smooth, hence formally smooth. Thus, we can find a local $\mathbb{C}$-homomorphism $g^{\sharp}:\mathcal{O}_{X,x}\rightarrow \mathcal{O}_{T',t'}$ such that the following diagram commutes:
	\[
	\begin{tikzcd}
		\mathcal{O}_{S,s} \arrow[r, "a^{\sharp}"] \arrow[d, "f^{\sharp}"]
		& \mathcal{O}_{T',t'} \arrow[d, "i^{\sharp}" ] \\
		\mathcal{O}_{X,x} \arrow[ru, "g^{\sharp}"] \arrow[r] & \mathcal{O}_{T,t}
	\end{tikzcd}\,.
	\]
	Now by \cite[Théorème~1.3]{SHC~6} or \cite[Section~0.21]{Fis06}, the homomorphism $g^{\sharp}$ induces a morphism of germs $g:(T',t')\rightarrow (X,x)$. This shows that $\mathcal{F}_t\neq \emptyset$.
	
	Now in order to show that the $\mathcal{H}$-torsor $\mathcal{F}$ is trivial, it suffices to show that $H^1(T,\mathcal{H})=0$. See \cite[\href{https://stacks.math.columbia.edu/tag/02FQ}{Tag 02FQ}]{stacks-project}.
    As shown in \cite{SHC~6}, both $\Omega^1_{X/S}$ and $\mathcal{C}_{T/T'}$ are coherent, hence so is $\mathcal{H}$. The vanishing then follows from Cartan's Theorem~B (\cite[Page~33]{Fis06}).
\end{proof}

We say an analytic space $X$ is irreducible if $X^{\Red}$ is irreducible (\cite[Section~9.1]{CAS}).

\begin{theorem}[Zariski's main theorem]\label{thm:Zmt}
	Let $f:Y\rightarrow X$ be a proper bimeromorphic morphism of complex analytic spaces. Assume that $X$ is unibranch at $x\in X$. Then the fibre $f^{-1}(x)$ is connected.
\end{theorem}
\begin{proof}
	Note that $f$ induces a morphism $f^{\Red}:Y^{\Red}\rightarrow X^{\Red}$ satisfying all assumptions of the theorem, so we may assume that $X$ and $Y$ are reduced. Then the result is proved in \cite[Proof of Théorème~1.7]{Dem85}.
\end{proof}

\subsection{Kähler spaces}
Let $X$ be a complex analytic space. See \cite{HL71} for the notion of differential forms and currents on a reduced complex analytic space. A differential form/current on $X$ is defined as a differential form/current on $X^{\Red}$.
\begin{definition}\label{def:Kah}
	A \emph{Kähler form} on $X$ is a smooth $(1,1)$-form $\omega$ on $X$ such that at any point $x\in X$, there is a neighbourhood $V\subseteq X$ of $x$, a closed immersion $V^{\Red}\hookrightarrow \Omega$ into some bounded pseudo-convex domain $\Omega\subseteq \mathbb{C}^N$, a Kähler form $\omega_{\Omega}$ on $\Omega$ such that $\omega=\omega_{\Omega}|_{V}$.
	
	A \emph{Kähler space} is a complex analytic space that admits a Kähler form.
\end{definition}

\begin{example}
	Any Kähler manifold is a Kähler space. In fact, when $X$ is smooth, the notion of Kähler form in \cref{def:Kah} coincides with the usual one.
\end{example}
\begin{example}\label{ex:sub}
	Let $X$ be a Kähler space. Let $Y$ be a closed analytic subspace, then $Y$ is a Kähler space.
\end{example}
As a consequence, any projective analytic space is Kähler.
More generally,
\begin{lemma}\label{lma:proj}
	Let $X$ be a compact Kähler space. Let $f:Y\rightarrow X$ be a projective morphism. Then $Y$ is a Kähler space.
\end{lemma}
For the definition of projective morphism, see \cite[Section~V.4]{GPR94}.
\begin{proof}
	We can embed $Y$ as a closed subspace of $X\times \mathbb{P}^N$ for some $N\geq 0$ preserving the morphism to $X$. Let $p_1$, $p_2$ be the projection from $X\times \mathbb{P}^N$ to two factors. Take a Kähler form $\omega$ on $X$.
	Let $\omega_{\FS}$ be the Fubini--Study metric on $\mathbb{P}^N$. Then we claim that $p_1^*\omega+p_2^*\omega_{\FS}$ defines a Kähler form on $Y$. By \cref{ex:sub}, it suffices to show that $p_1^*\omega+p_2^*\omega_{\FS}$ is a Kähler form on $X\times \mathbb{P}^N$. We may assume that $X$ is reduced.
	The problem is also local in $\mathbb{P}^N$, we could replace $\mathbb{P}^N$ by a polydisk $\Delta$ in it. 
	
	The problem is local on $X$, so we may assume that $X$ is a closed subspace of a pseudo-convex domain $\Omega$ in $\mathbb{C}^M$ and there is a Kähler form $\omega_{\Omega}$ such that $\omega=\omega_{\Omega}|_X$. Now $X\times \Delta\hookrightarrow  \Omega\times \Delta$ and the form
	\[
	p_1^*\omega+p_2^*\omega_{\FS}=\left(\pi_1^*\omega_{\Omega}+\pi_2^*\omega_{\FS}\right)|_{X\times \Delta}\,,
	\]
	where $\pi_1$, $\pi_2$ are the two projection from $\Omega\times \mathbb{P}^N$ to the two factors.
\end{proof}
\begin{corollary}\label{cor:Bl}
	Let $X$ be a compact K\"ahler space. Let $Y$ be a closed subspace. Then $\Bl_Y X$ is a Kähler space.
\end{corollary}
Here $\Bl_Y X$ is the blowing-up of $X$ with center $Y$. For its precise definition, we refer to \cite[Section~VII.2]{GPR94}.

Recall that we can always resolve the singularities of complex analytic spaces. 
\begin{theorem}\label{thm:ressing}
	Let $X$ be a reduced complex analytic space. Then there is a (proper) resolution of singularity $\pi:Y\rightarrow X$ of $X$. Moreover, we may assume that $\pi$ is an isomorphism over the non-singular part of $X$.
\end{theorem}
This theorem was first proved by Aroca--Hironaka--Vicente, see the book \cite{AHV18}. Later simplifications are due to Bierstone--Milman, Villamayor and W\l{}odarczyk. See \cite{Wlo09} for the details and further references.

\begin{corollary}
	Let $X$ be a reduced compact complex K\"ahler space. Then there is a resolution of singularity $\pi:Y\rightarrow X$ such that $Y$ is a compact Kähler manifold.
\end{corollary}
\begin{proof}
This follows from the fact that $X$ admits a projective resolution, see \cite{Wlo09}. Here we need the fact that $X$ is compact.
\end{proof}
It is of interest to know if this corollary still holds if $X$ is not compact, in which case we cannot resolve the singularities of $X$ by sequences of global blowing-ups.

\textcolor{red}{In the remaining of this paper, by a resolution of singularity $f:Y\rightarrow X$ of a reduced Kähler space $X$, we always assume that $Y$ is a Kähler manifold.}

\section{Non-pluripolar products on unibranch spaces}\label{sec:npp}

Let $X$ be a complex analytic space.

\subsection{Plurisubharmonic functions}
\begin{definition}
	Let $U\subseteq X$ be an open immersion. A function $\varphi:U\rightarrow [-\infty,\infty)$ is \emph{plurisubharmonic} if 
	\begin{enumerate}
		\item $\varphi$ is not identically $-\infty$ on any irreducible component of $U$.
		\item For any $x\in U$, there is an open neighbourhood $V$ of $x$ in $U$, a bounded pseudo-convex domain $\Omega\subseteq \mathbb{C}^N$, a closed immersion $V\hookrightarrow \Omega$, an open set $\tilde{V}\subseteq \Omega$ with $x\in \tilde{V}$ and a plurisubharmonic function $\tilde{\varphi}$ on $\tilde{V}$ such that $\varphi|_{\tilde{V}\cap V}=\tilde{\varphi}|_{\tilde{V}\cap V}$.
	\end{enumerate}
	The set of plurisubharmonic functions on $U$ is denoted by $\PSH(U)$.
\end{definition}

\begin{proposition}\label{prop:coiusu}
	Let $X$ be a complex manifold. Then with the canonical complex analytic space structure on $X$, the definition of $\PSH(X)$ coincides with the usual one.
\end{proposition}
\begin{proof}
	It suffices to recall that a psh function on a domain restricts to a psh function on a closed analytic submanifold.
\end{proof}

\begin{proposition}\label{prop:pshpull}
	Let $f:X\rightarrow Y$ be a morphism between complex analytic spaces. Let $\varphi\in\PSH(Y)$. Assume that $f^{*}\varphi$ is not identically equal to $-\infty$ on each irreducible component of $X$, then $f^*\varphi\in \PSH(X).$
\end{proposition}
\begin{proof}
	Let $x\in X$, $y=f(x)$. We need to verify Condition~(2). The problem is local, so we may assume that there is a closed immersion $X\hookrightarrow \Sigma$ and that there is an open neighbourhood $V\subseteq Y$ of $y$, a closed immersion  $V\hookrightarrow \Omega$, an open set $\tilde{V}\subseteq \Omega$ containing $y$ and a psh function $\tilde{\varphi}$ on $\tilde{V}$ such that $\tilde{\varphi}|_{\tilde{V}\cap V}=\varphi|_{\tilde{V}\cap V}$,  where $\Sigma$ and $\Omega$ are bounded pseudoconvex domains in $\mathbb{C}^N$ and $\mathbb{C}^M$ respectively. Shrinking $X$, we may assume that $f(X)\subseteq \tilde{V}\cap V$.
	We get a closed immersion:
	\[
	X\hookrightarrow X\times  V \hookrightarrow \Sigma\times \Omega\,,
	\]
	where the first morphism is the base change of $\Delta_V:V\rightarrow V\times V$:
	\[
	\begin{tikzcd}
		X \arrow[r, hookrightarrow] \arrow[d, "f"]
		& X\times V \arrow[d] \arrow[r, hookrightarrow] & \Sigma\times \Omega \\
		V \arrow[r, hookrightarrow, "\Delta_V"] \arrow[ru, phantom, "\square"] & V\times V &
	\end{tikzcd}\,.
	\]
	
	Define $\tilde{X}\coloneqq \Sigma \times \tilde{V}$. Define a psh function $\psi$ on $\tilde{X}$ as the pull-back of $\tilde{\varphi}$ from the projection onto the second variable. Then $f^*\varphi|_{\tilde{X}\cap X}=\psi|_{\tilde{X}\cap X}$.  
\end{proof}

\begin{proposition}\label{prop:pshbij}
	There is a canonical bijection
	\begin{equation}\label{eq:pshbij}
		\PSH(X)\cn \PSH(X^{\Red})\,.
	\end{equation}
\end{proposition}
\begin{proof}
	Let $\varphi\in \PSH(X)$. We claim that $\varphi\in \PSH(X^{\Red})$ as well. 
	Condition~(1) is trivially satisfied. Let us prove Condition~(2).
	The problem is local. Fix $x\in X$. We may assume that there is a closed immersion $X\hookrightarrow \Omega$, where $\Omega$ is a bounded pseudo-convex domain in $\mathbb{C}^N$, an open set $\tilde{X}\subseteq \Omega$ with $x\in  \tilde{X}$ and a psh function $\tilde{\varphi}$ on $\tilde{X}$ such that $\tilde{\varphi}|_{\tilde{X}\cap X}=\varphi|_{\tilde{X}\cap X}$. 
	Note that $X^{\Red}\hookrightarrow X$ induces a closed immersion $X^{\Red}\hookrightarrow \Omega$. Now we can take the same $\tilde{\varphi}$ to conclude.
	
	Now let $\varphi\in \PSH(X^{\Red})$. We want to show $\varphi\in \PSH(X)$. Again, it suffices to prove Condition~(2). The problem is local. We may assume that $X$ is Stein.
	Take $x\in X$. We may assume that there is a closed immersion $X^{\Red}\hookrightarrow \Omega$, where $\Omega$ is a bounded pseudo-convex domain in $\mathbb{C}^N$, an open set $\tilde{X}\subseteq \Omega$ with $x\in \tilde{X}$ and a psh function $\tilde{\varphi}$ on $\tilde{X}$ such that $\tilde{\varphi}|_{\tilde{X}\cap X}=\varphi|_{\tilde{X}\cap X}$. 
	
	By \cref{thm:smoothimpform}, after possibly shrinking $X$, we can lift the closed immersion $X^{\Red}\hookrightarrow \Omega$ to a morphism $j:X\rightarrow \Omega$:
	\[
	\begin{tikzcd}
		X^{\Red} \arrow[r, hookrightarrow] \arrow[d,hookrightarrow]
		& \Omega \\
		X \arrow[ru, "j"]& 
	\end{tikzcd}\,.
	\]
	By \cref{prop:pshpull}, $j^*\varphi$ is psh. Now $\varphi$ is the image of $j^*\varphi$ under \eqref{eq:pshbij}.
\end{proof}

By this proposition, we could always restrict our attention to reduced analytic spaces.

\begin{theorem}[Fornaess--Narasimhan]\label{thm:FN}
	Let $\varphi:X\rightarrow [-\infty,\infty)$ be a function. Assume that $\varphi$ is not identically $-\infty$ on any irreducible component of $X$, then the following are equivalent:
	\begin{enumerate}
		\item $\varphi$ is psh.
		\item $\varphi$ is usc and for any morphism $f:\Delta\rightarrow X$ from the open unit disk $\Delta$ in $\mathbb{C}$ to $X$ such that $f^*\varphi$ is not identically $-\infty$, the pull-back $f^*\varphi$ is psh. 
	\end{enumerate}
	Moreover, assume that $X$ is unibranch, then the conditions are equivalent to
	\begin{enumerate}[resume]
		\item $\varphi|_{X\setminus \Sing X^{\Red}}$ is psh, $\varphi$ is locally bounded from above near $\Sing X^{\Red}$ and $\varphi$ is strongly usc in the following sense:
		\begin{equation}\label{eq:susc}
			\varphi(x)=\varlimsup_{y\to x,y\in X\setminus \Sing X^{\Red}} \varphi(y)\,,\quad x\in X\,.
		\end{equation}
		\item $\varphi$ is locally integrable, locally bounded from above, strongly usc and $\ddc\varphi\geq 0$.
	\end{enumerate}
\end{theorem}
\begin{proof}Note that we can assume that $X$ is reduced.
	The equivalence between (1) and (2) is the classical Fornaess--Narasimhan theorem. See \cite{FN80}. For equivalence with the other conditions, see \cite[Section~1.8]{Dem85}. Here we need \cref{thm:Zmt} in an essential way.
\end{proof}

\begin{corollary}\label{cor:pshextension}
	Assume that $X$ is unibranch. Let $\varphi\in \PSH(X\setminus \Sing X^{\Red})$. Assume that $\varphi$ is locally bounded from above near $x\in \Sing X^{\Red}$, then there is a unique extension $\varphi\in \PSH(X)$.

  In particular, if $\pi:Y\rightarrow X$ is a resolution, then $\pi^*:\PSH(X)\rightarrow \PSH(Y)$ is bijective.
\end{corollary}
See \cite[Théorème~1.7]{Dem85} for the details. 

\begin{proposition}\label{prop:incdec}
	\leavevmode
	\begin{enumerate}
		\item Assume that $X$ is unibranch. Let $\varphi_{\theta}$ be a family in $\PSH(X)$, locally uniformly bounded from above. Then $\sups_{\theta} \varphi_{\theta}$ is also psh.
		\item Let $\varphi_{\theta}$ be a decreasing net in $\PSH(X)$ such that $\inf_{\theta} \varphi_{\theta}$ is not identically $-\infty$ on each irreducible component of $X$, then $\inf_{\theta} \varphi_{\theta}$ is psh.
	\end{enumerate}
\end{proposition}
Here
\[
\sups f_{\theta}\coloneqq (\sup_{\theta} f_{\theta})^*
\]
and
\[
f^*(x)\coloneqq \varlimsup_{y\to x,y\in X\setminus \Sing X^{\Red}} f(y)\,.
\]
\begin{proof}
	(1) Observe that 
	\[
	(\sups_{\theta} \varphi_{\theta})|_{X\setminus \Sing X^{\Red}}=\sups_{\theta} \varphi_{\theta}|_{X\setminus \Sing X^{\Red}}\,,
	\]
	where the right-hand side is psh by the classical theory. Now $\sups_{\theta} \varphi_{\theta}$ is clearly locally bounded from above, \eqref{eq:susc} also clearly holds.
	
	(2) Note that $\inf_{\theta}\varphi_{\theta}$ is usc. Condition~(2) of \cref{thm:FN} clearly holds. 
\end{proof}

\subsection{Local regularization}

We will need the following extension theorem.
\begin{theorem}\label{thm:ext}
	Let $M$ be a Stein manifold and $N\subseteq M$ be a closed reduced complex subspace. Let $\varphi$ be a psh function on $N$. Assume that $\psi$ is a continuous psh exhaustion function on $M$ such that $\varphi< \psi|_N$. Let $c>1$. Then there is a psh extension of $\varphi$ to $M$ such that $\varphi< c\max\{\psi,0\}$.
\end{theorem}
For the proof, we refer to \cite{CGZ13}.

For each $N\geq 0$, we fix a Friedrichs kernel $\rho=\rho_N:[0,\infty)\rightarrow [0,\infty)$ such that $\rho$ is smooth, $\rho(r)=0$ for $r\geq 1$ and 
\[
\int_{\mathbb{C}^N} \rho(|z|)\,\mathrm{d}\lambda(z)=1\,.
\]
Let $U\subseteq \mathbb{C}^N$ be an open subset.
For any locally integrable function $u:U\rightarrow [-\infty,\infty)$ and any $\delta>0$, define
\[
\varphi_{\delta}(x)\coloneqq \int_{\mathbb{C}^N} u(x-\delta y) \rho(|y|)\,\mathrm{d}\lambda(y)\,,\quad x\in U_{\delta}\,,
\]
where
\[
U_{\delta}\coloneqq \left\{\,x\in U: B(x,\delta)\subseteq U\,\right\}\,.
\]

\begin{lemma}\label{lma:BK1}
	Let $\Omega\subseteq \mathbb{C}^N$ be a bounded pseudo-convex domain, $V\hookrightarrow \Omega$ be a closed analytic subspace and $W\Subset V$ be an open subset.
	Let $\varphi\in \PSH(V)\cap L^{\infty}(V)$. There exists a decreasing sequence $\psi_i$ of smooth psh functions on $W$, converging pointwisely to $\varphi$ on $W$.
\end{lemma}
\begin{proof}
	By \cref{thm:ext} and \cref{prop:pshbij}, $\varphi$ can be extended to a psh function on $\Omega$. Define $\psi_i=\varphi_{1/i}|_{W}$.
\end{proof}

\subsection{Bedford--Taylor products}
Proofs in this section are mostly taken from \cite{Dem85} and the book \cite{GZ17}. 

Fix a complex analytic space $X$.
\begin{definition}[Bedford--Taylor]
	Let $\varphi_i\in \PSH(X)\cap L^{\infty}_{\loc}(X)$ ($i=1,\ldots,k$). Let $T$ be a closed positive current of bidimension $(m,m)$ on $X$. We define
	\begin{equation}
		\ddc\varphi_1\wedge \cdots\wedge \ddc\varphi_k\wedge T\coloneqq \ddc\left(\varphi_1 \ddc\varphi_2\wedge \cdots\wedge \ddc\varphi_k\wedge T\right)\,.
	\end{equation}
\end{definition}
	Unless $X$ is equi-dimensional, the bidegree of a current is not well-defined. So we only talk about bidimensions.

\begin{proposition}
	Let $\varphi_i\in \PSH(X)\cap L^{\infty}_{\loc}(X)$ ($i=1,\ldots,k$). Let $T$ be a closed positive current of bidimension $(m,m)$ on $X$. Then $\ddc\varphi_1\wedge \cdots\wedge \ddc\varphi_k\wedge T$ is a closed positive current of bidimension $(m-k,m-k)$.
\end{proposition}
\begin{proof}
	We prove by induction. When $k=0$, there is nothing to prove. Assume that the result is known for $k-1$, namely, assume that $\ddc\varphi_2\wedge \cdots\wedge \ddc\varphi_k\wedge T$ is closed positive. Then for any smooth psh function $\varphi$,
	\[
	\ddc\varphi_1\wedge\cdots\ddc\varphi_k\wedge T
	\]
	is clearly closed positive. As our problem is local, we may assume that there is a decreasing sequence of smooth psh functions $\varphi^i$ converging pointwisely to $\varphi_1$, then 
	\[
	\varphi^i\ddc\varphi_2\wedge \cdots\wedge \ddc\varphi_k\wedge T\rightharpoonup \varphi_1\ddc\varphi_2\wedge \cdots\wedge \ddc\varphi_k\wedge T\,,
	\]
	so
	\[
	\ddc\varphi^i\wedge \ddc\varphi_2\wedge \cdots\wedge \ddc\varphi_k\wedge T\rightharpoonup \ddc\varphi_1\wedge\ddc\varphi_2\wedge \cdots\wedge \ddc\varphi_k\wedge T\,,
	\]
	this proves our result.
\end{proof}
Now we prove the functoriality of this product. 
\begin{proposition}[Projection formula]\label{prop:projf}
	Let $\pi:Y\rightarrow X$ be a proper morphism of complex analytic spaces. Let $\varphi_0,\ldots,\varphi_k\in \PSH(X)\cap L^{\infty}_{\loc}(X)$ and $T$ be a closed positive current of bidimension $(m,m)$ on $Y$. Then
	\[
	\pi_*(\pi^*\varphi_0\,\ddc\pi^*\varphi_1\wedge \cdots\wedge \ddc\pi^*\varphi_k\wedge T)=\varphi_0\,\ddc\varphi_1\wedge \cdots \wedge \ddc \varphi_k\wedge \pi_*T\,.
	\]
\end{proposition}
\begin{proof}
The case $k=0$ is the classical projection formula. In general, by induction on $k$, we may assume that 
\[
	\pi_*(\pi^*\varphi_1\,\ddc\pi^*\varphi_2\wedge \cdots\wedge \ddc\pi^*\varphi_k\wedge T)=\varphi_1\,\ddc\varphi_2\wedge \cdots \wedge \ddc \varphi_k\wedge \pi_*T\,.
	\]
 Hence,
 \[
 \pi_*(\ddc \pi^*\varphi_1\wedge\ddc\pi^*\varphi_2\wedge \cdots\wedge \ddc\pi^*\varphi_k\wedge T)=\ddc\varphi_1\wedge \ddc\varphi_2\wedge \cdots \wedge \ddc \varphi_k\wedge \pi_*T\,.
 \]
 It suffices to apply the case with $k=0$.
\end{proof}

In particular, let $Y$ be an irreducible component of $X$ of dimension $n$, let $i:Y\rightarrow X$ be the inclusion, $T=[Y]$ is a closed positive current of bidimension $(n,n)$ (see \cite[Section~0.2]{GH78} for example). Then we conclude
\begin{corollary}
	Let $Y$ be an irreducible component of $X$ of dimension $n$ and $i:Y\rightarrow X$ be the inclusion. Let $\varphi_1,\ldots,\varphi_k\in \PSH(X)\cap L^{\infty}_{\loc}(X)$. Then
	\[
	i_*(\ddc\varphi_1|_Y\wedge\cdots \wedge\ddc\varphi_k|_Y)=\ddc\varphi_1\wedge \cdots\wedge \ddc\varphi_k\wedge [Y]\,.
	\]
\end{corollary}
\begin{corollary}
	Let $\pi:Y\rightarrow X$ be a proper bimeromorphic morphism between complex analytic spaces. Let $\varphi_1,\ldots,\varphi_k\in \PSH(X)\cap L^{\infty}_{\loc}(X)$. Then
	\[
	\pi_*(\pi^*\varphi_0\,\ddc\pi^*\varphi_1\wedge \cdots\wedge \ddc\pi^*\varphi_k)=\varphi_0\,\ddc\varphi_1\wedge \cdots \wedge \ddc \varphi_k\,.
	\]
\end{corollary}
\begin{corollary}
	Let $\varphi_1,\ldots,\varphi_k\in \PSH(X)\cap L^{\infty}_{\loc}(X)$. Let $i:X^{\Red}\rightarrow X$ be the canonical inclusion. Then
	\[
	i_*(\ddc\varphi_1|_{X^{\Red}}\wedge \cdots \wedge \ddc\varphi_k|_{X^{\Red}})=\ddc\varphi_1\wedge \cdots\wedge \ddc\varphi_k\,.
	\]
\end{corollary}

\begin{theorem}[Chern--Levine--Nirenberg]
	Let $U\subseteq X$ be an open subset, $\varphi_i\in \PSH(U)\cap L^{\infty}_{\loc}(U)$ ($i=1,\ldots,k$). Let $T$ be a closed positive current of bidimension $(k,k)$ on $U$. Assume that $W\Subset V\Subset U$ are two open sets. Then there is a constant $C=C(W,V)>0$ such that for any compact set $K\subset W$,
	\begin{equation}\label{eq:CLN1}
		\int_K \ddc\varphi_1\wedge\cdots\wedge \ddc\varphi_k \wedge T \leq C\norm{\varphi_1}_{L^{\infty}(E)}\cdots\norm{\varphi_k}_{L^{\infty}(E)}\norm{T}_{E}, 
	\end{equation}
	where $E=\Supp T\cap (V\setminus W)$.
	
	Here $\|T\|_{E}$ is a semi-norm, defined as follows: take finitely many open sets $U_i\subseteq U$, open subsets $V_i\Subset U_i$ such that $\{V_i\}$ covers $U$ and such that there are Kähler forms $\omega_i$ on each $U_i$. Then 
	\[
	\|T\|_{E}\coloneqq \sum_i\int_{V_i} \omega_i^k\wedge T\,.
	\]
\end{theorem}

\begin{proof}
The proof is almost identical to \cite[Theorem~3.9]{GZ17}. We briefly recall the argument.

	We first prove \eqref{eq:CLN1}. By induction, it suffices to treat the case where $k=1$. We omit the index and write $\varphi=\varphi_1$. We may assume that $\varphi|_V\leq 0$. Let $\chi$ be a non-negative compactly supported smooth function on $V$, equal to $1$ on $W$. Then
	\[
	\int_W \ddc\varphi_1\wedge T\leq \int_V \chi\ddc\varphi_1\wedge T=\int_V \varphi_1\ddc\chi\wedge T=\int_{V\setminus W} \varphi_1\ddc\chi\wedge T\,,
	\]
	where we have applied \cref{lma:ibp}. Now \eqref{eq:CLN1} is obvious.
\end{proof}

\begin{lemma}\label{lma:ibp}
	Let $U\Subset X$ be an open subset. Let $T$ be a closed positive current of bidimension $(1,1)$ on $U$. Let $\varphi,\psi\in \PSH(U)\cap L^{\infty}_{\loc}(U)$, $\varphi,\psi\leq 0$. Assume $\lim_{x\to\partial U} \varphi(x)=0$ and $\int_U \ddc \psi\wedge T<\infty$,  then
	\[
	\int_U \psi\,\ddc \varphi\wedge T\leq \int_U \varphi\,\ddc \psi\wedge T\,.
	\]
\end{lemma}
The proof is almost identical to that of \cite[Proposition~3.7]{GZ17}. It suffices to replace the smoothing procedure by the one explained in \cref{lma:BK1}.

\begin{definition}
	A subset $E\subseteq X$ is \emph{pluripolar} if for any $x\in X$, there is an open neighbourhood $V\subseteq X$ of $x$ and a psh function $\varphi$ on $V$ such that $E\cap V\subseteq \{x\in V:\varphi(x)=-\infty\}$.
\end{definition}
\begin{definition}
	The $\sigma$-algebra of \emph{quasi-Borel sets} is the $\sigma$-algebra generated by all Borel sets and all pluripolar sets. A set in this $\sigma$-algebra is called a \emph{quasi-Borel set}. 
\end{definition}

\begin{definition}Assume that $X$ is unibranch.
	A set $E\subseteq X$ is \emph{negligible} if for any $x\in X$, there is an open neighbourhood $V\subseteq X$ of $x$, a bounded from above family $\{\psi_{\theta}\}$ of psh functions on $V$ such that
	\[
	E\cap V\subseteq \left\{\,x\in V:\sups\{\psi_{\theta}\}(x)> \sup \{\psi_{\theta}\}(x)\,\right\}\,.
	\]
\end{definition}
\begin{remark}
	There are different definitions in the literature of both pluripolar sets and negligible sets. In some literature, our notion of pluripolar sets is called locally pluripolar sets instead. However, we want to emphasize that what is actually proved in \cite{BT82} is that pluri-polarity is equivalent to negligibility \emph{in our sense}.
\end{remark}

Now we study the pluri-fine topology on $X$. 
Classical analogs are proved in \cite{BT82} and \cite{BT87}.

Assume that $X$ has pure dimension $n$.
Let $\varphi_1,\ldots,\varphi_n\in \PSH(X)\cap L^{\infty}_{\loc}(X)$. Then $\ddc\varphi_1\wedge\cdots\wedge\ddc\varphi_n$ is a Borel measure. We extend this measure by taking its completion. Then all quasi-Borel sets are measurable.

\begin{theorem}\label{thm:Lelongconj}
	Assume that $X$ is an open subspace of a compact unibranch Kähler space.
	Let $E\subseteq X$ be a subset. Then $E$ is negligible if and only if  $E$ is pluripolar.
\end{theorem}
\begin{proof}
	We may assume that $X$ is reduced.
	
	Assume that $E$ is pluripolar. We show that $E$ is negligible. 
	The problem is local, so we may assume that there is a psh function $\varphi\leq 0$ on $X$ such that $E\subseteq \{\varphi=-\infty\}$. We may assume that equality  holds.
	For each $i\geq 1$, let $\varphi_i\coloneqq i^{-1}\varphi$. Observe that $E$ has empty interior by the classical theory on $E\cap (X\setminus \Sing X)$.
	Now $\sup_i\varphi_i=0$ on $X\setminus E$, while $\sup_i\varphi_i|_E=-\infty$, so $\sups\varphi_i=0$. Thus,
	\[
	E=\left\{\,\sup_i \varphi_i<\sups_i \varphi_i \,\right\}\,.
	\]
	
	Now assume that $E$ is negligible. We show that $E$ is pluripolar. We may assume that $X$ is a reduced, unibranch Kähler space with a Kähler form $\omega$.
	
	Let $\pi:Y\rightarrow X$ be a resolution of singularity. 
	Observe that $\pi^{-1}(E)$ is negligible, as the union of two negligible sets.
	Let 
	\[
	u_{E,\omega}^*\coloneqq \sups\left\{\, \varphi\in \PSH(X,\omega):\varphi\leq 0, \varphi|_E\leq -1 \, \right\}\in \PSH(X,\omega)\,.
	\]
	Observe that 
	\[
	\pi^*u_{E,\omega}^*=u_{\pi^{-1}E,\pi^*\omega}^*=0\,,
	\]
	where the second equality follows from the classical theory  (see \cite[Lemma~2.3]{Lu20} for example). In particular, $u_{E,\omega}^*=0$. So by Choquet's lemma (which holds on unibranch spaces thanks to \cref{cor:pshextension}), we may take an increasing sequence of $\omega$-psh functions $\psi_i$, $\psi_i\leq 0$, $\psi_i|_E\leq -1$ such that the $L^1$-norm of $\psi_i$ is bounded from above by $2^{-i}$. Then take $\psi=\sum_i\psi_i$. We find $E\subseteq\{\psi=-\infty\}$.
\end{proof}

\textcolor{red}{In the remainder of this section, we assume that  $X$ is an open subspace of a compact unibranch Kähler space.}
\begin{theorem}[Quasi-Lindel\"of property]\label{thm:qualin}
	An arbitrary union of pluri-fine open subsets differs from a countable subunion by at most a pluripolar set.
\end{theorem}
\begin{proof}
	Let $U_{\theta}\subseteq X$ ($\theta\in I$) be a family of pluri-fine open subsets in $X$. Let $U=\cup_{\theta} U_{\theta}$.
	Take a countable basis $B_i$ of the topology of $X$.
	We may assume that each $U_{\theta}$ is of the form $B_i\cap \{\varphi_{\theta}>-1\}$ for some $B_i$, where $\varphi_{\theta}\leq 0$ is a bounded psh function on $B_i$. It suffices to prove that for each fixed $i$, 
	\[
	\bigcup_{\theta} B_i\cap \{\varphi_{\theta}>-1\}
	\]
	with $\theta$ running through the subset set $J_i$ of $I$ consisting of all $\theta$ such that $U_{\theta}$ is of the form $B_i\cap \{\varphi_{\theta}>-1\}$. So we may assume that there are bounded psh functions $\varphi_{\theta}$ defined on $X$ such that $-1\leq \varphi_{\theta}\leq -1$, $U_{\theta}=\{\varphi_{\theta}>0\}$. Now by Choquet's lemma, there is a countable subset $J\subseteq I$ such that
	\[
	\sups_{\theta\in I}\varphi_{\theta}=\sups_{\theta\in J}\varphi_{\theta}\,.
	\]
	So by \cref{thm:Lelongconj},
	\[
	\sup_{\theta\in I}\left\{\varphi_{\theta}>0\right\}=\bigcup_{\theta\in I} U_{\theta}
	\]
	differs from
	\[
	\sup_{\theta\in J}\left\{\varphi_{\theta}>0\right\}=\bigcup_{\theta\in J} U_{\theta}
	\]
	by at most a pluripolar set. This proves our theorem.
\end{proof}

As a corollary, we obtain as in the classical setting:
\begin{corollary}\label{cor:pfqB}
	All pluri-fine Borel sets on $X$ are quasi-Borel.
\end{corollary}

\begin{theorem}\label{thm:conv}
	Let $\varphi_0^j,\ldots,\varphi_k^j$ be $(n+1)$-sequences of psh functions on $X$. Assume that the sequences are all uniformly bounded and converging a.e. monotonically (either increasing or decreasing) to psh functions $\varphi_0,\ldots,\varphi_k$ on $X$.
	Then
	\begin{enumerate}
		\item 
		\[
		\ddc\varphi_1^j\wedge\cdots \wedge \ddc\varphi_k^j\plright \ddc\varphi_1\wedge\cdots \wedge \ddc\varphi_k\,.
		\]
		\item 
		\[
		\varphi_0^j\ddc\varphi_1^j\wedge\cdots \wedge \ddc\varphi_k^j\plright \varphi_0\ddc\varphi_1\wedge\cdots \wedge \ddc\varphi_k\,.
		\]
	\end{enumerate}
	Here $\mathrm{p.f.}$ means that the weak convergence is with respect to the pluri-fine topology.
\end{theorem}
\begin{proof}
	It suffices to prove (2).
	
	We may assume that $X$ is reduced.
	Let $\pi:Y\rightarrow X$ be a resolution of singularity. 
	From the classical theory, we know that
	\[
	\pi^*\varphi_0^j\ddc\pi^*\varphi_1^j\wedge \cdots \wedge \ddc\pi^*\varphi_k^j\plright \pi^*\varphi_0\ddc\pi^*\varphi_1\wedge \cdots \wedge \ddc\pi^*\varphi_k\,.
	\]
	We conclude by \cref{prop:projf}.
\end{proof}

\begin{corollary}\label{cor:symm}
	Let $\varphi_i\in \PSH(X)\cap L^{\infty}_{\loc}(X)$ ($i=1,\ldots,k$).  Then for any $\sigma\in \mathcal{S}_k$,
	\[
	\ddc\varphi_1\wedge \cdots \wedge \ddc\varphi_k=\ddc\varphi_{\sigma(1)}\wedge \cdots \wedge \ddc\varphi_{\sigma(k)}\,.
	\]
\end{corollary}
This is a standard consequence of \cref{thm:conv}, we omit the details.

In particular, we may use the following notation for any of these products: $\bigwedge_{j=1}^k \ddc \varphi_j$.

\begin{lemma}\label{lma:btdoesnotchargepp}
	Let $\varphi_1,\ldots,\varphi_k\in \PSH(X)$. Let $E$ be a pluripolar set. Then $\ddc\varphi_1\wedge \cdots\wedge \ddc\varphi_k$ does not charge $E$ in the following sense: For each irreducible component $Y$ of $X$, $\ddc\varphi_1\wedge \cdots\wedge \ddc\varphi_k\wedge [Y]$ does not charge $E\cap Y$.
\end{lemma}
\begin{proof}
	We may assume that $X$ is irreducible of dimension $n$. Take a resolution, then we reduce immediately to the case where $X$ is smooth, in which case the result is well-known.
\end{proof}

\begin{lemma}\label{lma:maxloc}
	Assume that $X$ is of pure dimension $n$.
	Let $\varphi,\psi\in \PSH(X)\cap L^{\infty}_{\loc}(X)$. Let $O\coloneqq \{\varphi>\psi\}$. Then
	\begin{equation}\label{eq:ddcmax1}
		(\ddc (\varphi\lor \psi))^n|_O=(\ddc \varphi)^n|_O\,.    
	\end{equation}
\end{lemma}
\begin{proof}
	The problem is local. We may assume that $X$ is a closed analytic subspace of a bounded pseudo-convex domain $\Omega$ in some $\mathbb{C}^N$.
	By \cref{lma:BK1}, we may assume that there is a decreasing sequence of smooth psh functions $\varphi_k$ on $X$ converging pointwisely to $\varphi$. It is obvious that \eqref{eq:ddcmax1} holds for $\varphi_k$ in place of $\varphi$. We conclude using the quasi-Lindel\"of property \cref{thm:qualin} as in \cite[Lemma~4.1]{BT87}. 
\end{proof}

\begin{theorem}[Plurilocality]\label{thm:pluloc}
	Let $\varphi^i,\psi^i\in \PSH(X)\cap L^{\infty}_{\loc}(X)$ ($i=1,\ldots,n$). Let $W\subseteq X$ be an open set with respect to the pluri-fine topology. Assume that $\varphi^i|_W=\psi^i|_W$ for all $i$, then
	\[
	\ddc\varphi^1\wedge \cdots \wedge \ddc \varphi^n|_{W}=\ddc\psi^1\wedge \cdots \wedge \ddc \psi^n|_{W}\,.
	\]
\end{theorem}
\begin{proof}
	By polarization, we may assume that all $\varphi^i$ are equal and all $\psi^i$ are equal. We omit the superindex. Then we want to show $(\ddc\varphi)^n=(\ddc\psi)^n$ on $W$.
	Note that $\varphi=\varphi\lor (\psi-\epsilon)$ for any $\epsilon>0$. It suffices to apply \cref{lma:maxloc}.
\end{proof}

Finally, we introduce the notion of quasi-psh functions.
\begin{definition}
	Let $\theta$ be a smooth strongly closed real $(1,1)$-form on $X$ (in the sense that locally it is the pull-back of a smooth closed real $(1,1)$-form $\tilde{\theta}$ on pseudo-convex domains). Then the set $\PSH(X,\theta)$ consists of all functions $\varphi:X\rightarrow [-\infty,\infty)$ such that on each open set $V\subseteq X$, embedded as a closed analytic subspace in a pseudo-convex domain $\Omega\subseteq \mathbb{C}^N$ and any smooth function $a$ on $\Omega$ with $\ddc a|_V=\theta$, $a|_V+\varphi$ is psh.
\end{definition}

\subsection{Non-pluripolar products}
Assume that $X$ is an open subset of a compact unibranch Kähler space of pure dimension $n$.

Based on \cref{thm:pluloc} and \cref{lma:btdoesnotchargepp}, we could introduce the non-pluripolar products exactly as in \cite{BEGZ10}.

\begin{definition}Let $\varphi_1,\ldots,\varphi_p\in \PSH(X)$. Let 
	\[
	O_k=\bigcap_{j=1}^p\left\{\varphi_j>-k\right\}\,.
	\]
	We say that $\ddc \varphi_1\wedge\cdots\wedge \ddc \varphi_p$ is \emph{well-defined} if for each open subset $U\subseteq X$ such that there is a Kähler form $\omega$ on $U$, each compact subset $K\subseteq U$, we have
	\begin{equation}\label{eq:welldefinepluri}
		\sup_{k\geq 0} \int_{K\cap O_k}\left.\left(\bigwedge_{j=1}^p \ddc \max\{\varphi_j,-k\}\right)\right|_U\wedge \omega^{n-p}<\infty.
	\end{equation}
	
	In this case, we define $\ddc \varphi_1\wedge\cdots\wedge \ddc \varphi_p$ by 
	\begin{equation}\label{eq:npp}
		\mathds{1}_{O_k} \ddc \varphi_1\wedge\cdots\wedge \ddc \varphi_p =\mathds{1}_{O_k}\bigwedge_{j=1}^p \ddc \max\{\varphi_j,-k\}
	\end{equation}
	on $\bigcup_{k\geq 0} O_k$ and make a zero-extension to $X$. 
\end{definition}
\begin{remark}
	The condition \eqref{eq:welldefinepluri} is clearly independent of the choice of $U$ and $\omega$.
\end{remark}
The following proposition follows immediately from the definition and \cref{cor:symm}.
\begin{proposition}
	Let $\varphi_1,\ldots,\varphi_p\in \PSH(X)$. Let $\sigma\in \mathcal{S}_p$.
	Then $\ddc \varphi_1\wedge\cdots\wedge \ddc \varphi_p$ is well-defined if and only if  $\ddc \varphi_{\sigma(1)}\wedge\cdots\wedge \ddc \varphi_{\sigma(p)}$ is. Moreover, in this case,
	\[
	\ddc \varphi_1\wedge\cdots\wedge \ddc \varphi_p= \ddc \varphi_{\sigma(1)}\wedge\cdots\wedge \ddc \varphi_{\sigma(p)}\,.
	\]
\end{proposition}
In particular, we may use the following notation for any of these products: $\bigwedge_{j=1}^p \ddc \varphi_j$.

\begin{proposition}\label{prop:projnp}
	Let $\varphi_1,\ldots,\varphi_m$ be psh functions on $X$. Let $\pi:Y\rightarrow X^{\Red}$ be a resolution of singularity. Then $\varphi_1\wedge \cdots\wedge \varphi_m$ is well-defined if and only if  $\pi^*\varphi_1\wedge \cdots\wedge \pi^*\varphi_m$ is. In this case,
	\[
	\pi_*\left(\ddc\pi^*\varphi_1\wedge \cdots\wedge \ddc\pi^*\varphi_m\right)=\ddc\varphi_1\wedge \cdots\wedge \ddc\varphi_m\,.
	\]
\end{proposition}
\begin{proof}
	This follows directly from \cref{prop:projf}.
\end{proof}

This proposition allows us to generalize directly known facts about non-pluripolar products in the smooth setting to the current setting.

\begin{proposition} \label{prop:npp1}
	Let $\varphi_1,\ldots,\varphi_p\in \PSH(X)$.
	\begin{enumerate}
		\item The product $  \ddc \varphi_1\wedge\cdots\wedge \ddc \varphi_p $ is local with respect to the pluri-fine topology. In the following sense: let $O\subseteq X$ be a pluri-fine open subset, let $v_1,\ldots,v_p\in \PSH(X)$, assume that $\varphi_j|_{O}=v_j|_{O}$, for $j=1,\ldots,p$.
		Assume that $\bigwedge_{j=1}^p \ddc \varphi_j$, $\bigwedge_{j=1}^p \ddc v_j$
		are both well-defined, then 
		\begin{equation}\label{eq:ppp1}
			\left.  \bigwedge_{j=1}^p \ddc \varphi_j \right|_{O}=\left.  \bigwedge_{j=1}^p \ddc v_j \right|_{O}\,.
		\end{equation}
		If $O$ is open in the usual topology, then the product $\bigwedge_{j=1}^p \ddc (\varphi_j|_{O})$ on $O$ is well-defined and
		\begin{equation}\label{eq:ppp2}
			\left.  \bigwedge_{j=1}^p \ddc \varphi_j \right|_{O}=  \bigwedge_{j=1}^p \ddc (\varphi_j|_{O}) \,.
		\end{equation}
		Let $\mathcal{U}$ be an open covering of $X$. Then $  \ddc \varphi_1\wedge\cdots\wedge \ddc \varphi_p $
		is well-defined if and only if  each of the following products is well-defined
		\[
		\bigwedge_{j=1}^p \ddc (\varphi_j|_{U}) \,,\quad U\in \mathcal{U}\,.
		\]
		\item The current $  \ddc \varphi_1\wedge\cdots\wedge \ddc \varphi_p $ and the fact that it is well-defined depend only on the currents $\ddc \varphi_j$, not on specific $\varphi_j$.
		\item When $\varphi_1,\ldots,\varphi_p\in L^{\infty}_{\loc}(X)$, $  \ddc \varphi_1\wedge\cdots\wedge \ddc \varphi_p $ is well-defined and is equal to the Bedford--Taylor product.
		\item Assume that $  \ddc \varphi_1\wedge\cdots\wedge \ddc \varphi_p $ is well-defined, then $  \ddc \varphi_1\wedge\cdots\wedge \ddc \varphi_p $ puts not mass on pluripolar sets.
		\item Assume that $  \ddc \varphi_1\wedge\cdots\wedge \ddc \varphi_p $ is well-defined, then
		$\bigwedge_{j=1}^p \ddc \varphi_j$ is a closed positive current of bidegree $(p,p)$ on $X$.
		\item The product is multilinear: let $v_1\in \PSH(X)$, then
		\begin{equation}\label{eq:ppp6}
			\ddc (\varphi_1+v_1)\wedge \bigwedge_{j=2}^p \ddc \varphi_j  =   \ddc \varphi_1\wedge \bigwedge_{j=2}^p \ddc \varphi_j  +   \ddc v_1\wedge \bigwedge_{j=2}^p \ddc \varphi_j  
		\end{equation}
		in the sense that the left-hand side is well-defined if and only if  both terms on the right-hand side are well-defined, and the equality holds in that case.
	\end{enumerate}
\end{proposition}

Let $\theta_1,\ldots,\theta_p$ be smooth strongly closed $(1,1)$-forms on $X$. Let $\varphi_i\in \PSH(X,\theta_i)$. Let $T_i=\theta_i+\ddc\varphi_i$. Then $T_1\wedge \cdots\wedge T_p$ can be defined in the obvious way.

\begin{definition}
	We say a closed positive $(1,1)$-current $T$ on $X$ is \emph{good} if for any $x\in X$, there is a neighbourhood $V\subseteq X$ of $x$ such that there exists a smooth strongly  closed $(1,1)$-form $\theta$ on $V$ and a function $\varphi\in \PSH(V,\theta)$  such that $T=\theta+\ddc\varphi$ on $V$.
\end{definition}
Let $T_1,\ldots,T_p$ be good closed positive $(1,1)$-currents on $X$, we can define $T_1\wedge \cdots\wedge T_p$ in the obvious way. 
With the same proofs as in \cite{BEGZ10}, we have
\begin{proposition} \label{prop:npp2}
	
	Let $T_1,\ldots,T_p$ be closed positive currents of bidegree $(1,1)$ on $X$. Assume that all $T_i$'s are good.
	\begin{enumerate}
		\item The product $  T_1\wedge\cdots\wedge T_p $ is local in pluri-fine topology in the following sense: let $O\subseteq X$ be a pluri-fine open subset, let $S_1,\ldots,S_p$ be closed positive currents of bidegree $(1,1)$ on $X$. Assume that all $S_i$'s are good. Assume that
		\[
		T_j|_{O}=S_j|_{O},\quad j=1,\ldots,p
		\]
		and that $T_1\wedge \cdots\wedge T_p$, $S_1\wedge \cdots\wedge S_p$
		are both well-defined, then 
		\begin{equation}
			\left.  T_1\wedge \cdots\wedge T_p \right|_{O}=\left.  S_1\wedge \cdots\wedge S_p \right|_{O}.
		\end{equation}
		If $O$ is open in the usual topology, then the product $T_1\wedge \cdots\wedge T_p|_{O}$ on $O$ is well-defined and
		\begin{equation}
			\left.  T_1\wedge \cdots\wedge T_p \right|_{O}=  T_1\wedge \cdots\wedge T_p|_O .
		\end{equation}
		Let $\mathcal{U}$ be an open covering of $X$. Then $  T_1\wedge \cdots\wedge T_p $
		is well-defined if and only if  each of the following products is well-defined
		\[
		T_1\wedge \cdots\wedge T_p|_{U} ,\quad U\in\mathcal{U}.
		\]
		\item Assume that $  T_1\wedge \cdots\wedge T_p $ is well-defined, then the product $  T_1\wedge \cdots\wedge T_p $ puts not mass on pluripolar sets.
		\item Assume that $  T_1\wedge \cdots\wedge T_p $ is well-defined, then $T_1\wedge \cdots\wedge T_p$ is a closed positive current of bidegree $(p,p)$.
		\item The product $ T_1\wedge \cdots\wedge T_p $ is symmetric.
		\item The product is multi-linear: let $T'_1$ be a good closed positive current of bidegree $(1,1)$, then
		\[
		(T_1+T_1')\wedge T_2\wedge  \cdots\wedge T_p =  T_1\wedge T_2\wedge  \cdots\wedge T_p +  T_1'\wedge T_2\wedge  \cdots\wedge T_p 
		\]
		in the sense that the left-hand side is well-defined if and only if  both terms on the right-hand side are well-defined, and the equality holds in that case.
	\end{enumerate}
\end{proposition}

\begin{proposition}Let $X$ be a compact unibranch Kähler space of pure dimension $n$.
	Let $T_1,\ldots,T_p$ be good closed positive currents of bidegree $(1,1)$ on $X$. 
	Then $T_1\wedge \cdots\wedge T_p$ is well-defined.
\end{proposition}
\begin{proof}
	We may assume that $X$ is reduced.
	
	Fix a Kähler form $\omega$ on $X$. In this case, write $T_j=(T_j+C\omega)-C\omega$ for $C>0$ large enough and apply \cref{prop:npp2} (5), we may assume that $T_j$ is in a Kähler class.  So we can write $T_j=\omega_j+\ddc \varphi_j$, 
	where $\omega_j$ is a Kähler form and $\varphi_j$ is $\omega_j$-psh. Let $U$ be an open subset on which we can write $\omega_j=\ddc\psi_j$ with psh functions $\psi_j\leq 0$ on $U$. Now on $U$, for each $k\geq 0$, $\{\psi_j+\varphi_j>-k\}\subseteq \{\varphi_j>-k\}$
	so for each compact subset $K\subseteq U$,
	\[
	\begin{split}
		&\int_K \mathds{1}_{\cap_{j=1}^p\{\psi_j+\varphi_j>-k\}}\bigwedge_{j=1}^p \ddc \max\{\psi_j+\varphi_j,-k\}\wedge \omega^{n-p}\\= & \int_K \mathds{1}_{\cap_{j=1}^p\{\psi_j+\varphi_j>-k\}}\bigwedge_{j=1}^p \left(\omega_j+\ddc \max\{\varphi_j,-k\}\right)\wedge \omega^{n-p}\\
		\leq &\int_X \bigwedge_{j=1}^p \left(\omega_j+\ddc \max\{\varphi_j,-k\}\right)\wedge \omega^{n-p}\\
		=&\int_X \bigwedge_{j=1}^p \omega_j\wedge \omega^{n-p}.
	\end{split}
	\]
	The last step follows from the corresponding result on a resolution.
\end{proof}

\subsection{Bimeromorphic behaviour}
Let $X$ be a compact unibranch Kähler complex analytic space of pure dimension $n$.
Fix a resolution of singularity $\pi:Y\rightarrow X^{\Red}$ that is an isomorphism over $X\setminus \Sing X^{\Red}$.    

Let $\theta$ be a strongly closed smooth $(1,1)$-form on $X$. Assume that $[\theta]$ is big: for all proper bimeromorphic morphism $f:Y\rightarrow X$ from a complex manifold $Y$, $f^*\theta$ is big. In this case, define $\vol\theta=\vol f^*\theta$.
We set
\begin{equation}\label{eq:Vtheta}
	V_{\theta}\coloneqq \sups\left\{\,\varphi\in \PSH(X,\theta): \varphi\leq 0 \,\right\}\,.
\end{equation}

\begin{definition}
	Let $\varphi,\psi\in \PSH(X,\theta)$. Define 
	\[
	\varphi\land \psi\coloneqq \sups\left\{\,\eta\in \PSH(X,\theta): \eta\leq \varphi\,,\eta\leq \psi \,\right\}\,.
	\]
	When the set is empty, we just define $\sups\emptyset=-\infty$.
\end{definition}

\begin{definition}
	Let $U\subseteq X$ be an open subset. Let $\varphi\in \PSH(\pi^{-1}U)$, define $\pi_*\varphi\in \PSH(U)$ as the unique psh extension of $\varphi|_{\pi^{-1}(U\setminus \Sing X)}$. See \cref{cor:pshextension}.
    We call $\pi_*\varphi$ the psh pushforward of $\varphi$.
\end{definition}

\begin{definition}
	Let 
	\[
	\mathcal{E}(X,\theta)\coloneqq \left\{\,\varphi\in \PSH(X,\theta): \int_X \theta_{\varphi}^n=\vol\theta \, \right\}\,.
	\]
	For any $p\geq 1$,
	\[
	\mathcal{E}^p(X,\theta)\coloneqq \left\{\,\varphi\in \mathcal{E}(X,\theta): \int_X |\varphi|^p \theta_{\varphi}^n<\infty \, \right\}\,.
	\]
	Let
	\[
	\mathcal{E}^{\infty}(X,\theta)\coloneqq \left\{\,\varphi\in \PSH(X,\theta): \varphi-V_{\theta}\in L^{\infty}(X) \, \right\}\,.
	\]
\end{definition}

\begin{definition}
	Let $\varphi_0,\varphi_1\in \mathcal{E}^1(X,\theta)$. A \emph{subgeodesic} from $\varphi_0$ to $\varphi_1$ is a map $\varphi:(a,b)\rightarrow \mathcal{E}^1(X,\theta)$ such that
	\begin{enumerate}
		\item The potential $\Phi$ on $X\times \{z\in \mathbb{C}: e^{-b}<|z|<e^{-a}\}$ defined by $\Phi(x,z)\coloneqq \varphi_{-\log|z|}(x)$ is $\pi_1^*\theta$-psh, where $\pi_1:X\times \{z\in \mathbb{C}: e^{-b}<|z|<e^{-a}\}\rightarrow X$ is the natural projection.
		\item 
		\[
		\lim_{t\to a+}\varphi_t=\varphi_a\,,\quad \lim_{t\to b-}\varphi_t=\varphi_b\,.
		\]
	\end{enumerate}
\end{definition}

\begin{definition}\label{def:mod}
	A potential $\varphi\in \PSH(X,\theta)$ is \emph{model} if
	\[
	\varphi= \sups_{C>0} (\varphi+C)\land V_{\theta}\,.
	\]
\end{definition}
\begin{definition}
	We define $\mathcal{R}^1(X,\theta)$ as the set of all geodesic rays in $\mathcal{E}^1$ emanating from $V_{\theta}$.
\end{definition}

\begin{definition}
	The \emph{geodesic} between $\varphi_0,\varphi_1\in \mathcal{E}^1$ is the maximal subgeodesic between them.
\end{definition}

\begin{proposition}\label{prop:corr}
	The functions $\pi^*$ and $\pi_*$ are inverse to each other. Under this correspondence, we get a bijection between
	\begin{enumerate}
		\item $\PSH(X,\theta)$ and $\PSH(Y,\pi^*\theta)$. 
		\item $\mathcal{E}(X,\theta)$ and $\mathcal{E}(Y,\pi^*\theta)$.
		\item $\mathcal{E}^p(X,\theta)$ and $\mathcal{E}^p(Y,\pi^*\theta)$.
		\item Subgeodesics in $\mathcal{E}^1(X,\theta)$ and subgeodesics in $\mathcal{E}^{1}(Y,\pi^*\theta)$.
		\item Geodesics in $\mathcal{E}^1(X,\theta)$ and geodesics in $\mathcal{E}^{1}(Y,\pi^*\theta)$.
		\item $\mathcal{R}^1(X,\theta)$ and $\mathcal{R}^1(Y,\pi^*\theta)$.
		\item Model potentials in $\PSH(X,\theta)$ and in $\PSH(Y,\pi^*\theta)$.
	\end{enumerate}
 Moreover, if $\varphi,\psi\in \PSH(X,\theta)$, then
  \[
  \pi^*(\varphi\land \psi)=\pi^*\varphi \land \pi^*\psi.
  \]
\end{proposition}
\begin{proof}
	Only (4) needs proof. Let $\Phi$ be a subgeodesic in $\mathcal{E}^{1}(Y,\pi^*\theta)$, regarded as a potential on $Y\times A$, where $A=\{z\in \mathbb{C}: e^{-1}<|z|<1\}$. It is easy to verify then that the psh pushforward of $\Phi$ is the same as the ensemble of all psh pushforwards for fixed $z\in A$.
\end{proof}
Let $\varphi,\phi\in \PSH(X,\theta)$. We define
\[
[\varphi]\land\psi\coloneqq \sups_{\!\!\!\! C>0} \left((\varphi+C)\land \psi\right)\,.
\]
\begin{lemma}\label{lma:roofpro}
	Let $\varphi,\phi\in \PSH(X,\theta)$. Then $[\pi^*\varphi]\land \pi^*\psi=\pi^*\left([\varphi]\land \psi\right)$.
\end{lemma}
\begin{proof}
	By \cref{prop:corr}, for each $C>0$,
	\[
	(\pi^*\varphi+C)\land \pi^*\psi=\pi^*\left((\varphi+C)\land \psi\right)\,.
	\]
	As $[\varphi]\land\psi$ is by definition the minimal $\theta$-psh function lying above all $(\varphi+C)\land \psi$, by \cref{prop:corr} again, $\pi^*([\varphi]\land\psi)$ is the minimal $\pi^*\theta$-psh function lying above all $\pi^*((\varphi+C)\land \psi)=(\pi^*\varphi+C)\land \pi^*\psi$. Hence, we conclude.
\end{proof}

\subsection{Basic properties}
Let $X$ be a compact Kähler unibranch complex analytic space of pure dimension $n$. 
We use $\theta_0,\ldots,\theta_n$ to denote given strongly closed smooth forms on $X$ representing big cohomology classes.
\begin{theorem}[Integration by parts]\label{thm:ibpA}\leavevmode
	Let $\gamma_j\in \PSH(X,\theta_j)$ ($j=2,\ldots,n$) and $\varphi_1,\varphi_2\in \PSH(X,\theta_0)$, $\psi_1,\psi_2\in \PSH(X,\theta_1)$. Set $u=\varphi_1-\varphi_2$, $v=\psi_1-\psi_2$. 
	Assume that $u,v\in L^{\infty}(X)$.
	Then
	\begin{equation}\label{eq:ibp4}
		\int_X u\,\ddc v \wedge \theta_{2,\gamma_2}\wedge \cdots\wedge \cdots \wedge   \theta_{n,\gamma_{n}}= \int_X v\,\ddc u \wedge \theta_{2,\gamma_2}\wedge \cdots\wedge \cdots \wedge\theta_{n,\gamma_{n}}.
	\end{equation}
\end{theorem}
\begin{proof}
	In the smooth setting, this is proved in \cite{Xia19b}
		 and \cite{Lu20}. The general case follows immediately.
\end{proof}

\begin{theorem}[Semi-continuity]\label{thm:semic}
	Let $\varphi_j, \varphi_j^k\in \PSH(X,\theta_j)$ ($k\in \mathbb{Z}_{>0}$, $j=1,\ldots,n$). Let $\chi\geq 0$ be a bounded quasi-continuous function on $X$. 
	Assume that for any $j=1,\ldots,n$, as $k\to\infty$, $\varphi_j^k$ converges to $\varphi_j$ monotonically a.e..
	Then we have
	\begin{equation}\label{eq:semicon1}
		\varliminf_{k\to\infty}\int_X\chi \,\theta_{1,\varphi_1^k}\wedge \cdots \wedge \theta_{n,\varphi_n^k} \geq \int_X \chi \,\theta_{1,\varphi_1}\wedge \cdots \wedge \theta_{n,\varphi_n}\,. 
	\end{equation}
\end{theorem}
\begin{proof}
	In the smooth setting, this is due to \cite[Theorem~2.3]{DDNL18mono}. The general case follows immediately.
\end{proof}

\begin{theorem}[Monotonicity]\label{thm:mono}
	Let $\varphi_j,\psi_j\in \PSH(X,\theta_j)$ ($j=1,\ldots,n$). Assume that $[\varphi_j]\succeq [\psi_j]$ for every $j$, then
	\[
	\int_{X}   \theta_{1,\varphi_1}\wedge \cdots \theta_{n,\varphi_n} \geq \int_{X}   \theta_{1,\psi_1}\wedge \cdots \theta_{n,\psi_n} \,.
	\]
\end{theorem}
Here $[\varphi_j]\succeq [\psi_j]$ means that $\varphi_j+C\geq \psi_j$ for some $C\in \mathbb{R}$.
\begin{proof}
	In the smooth setting, this is proved in \cite{WN19}
	and \cite{DDNL18mono}. The general case follows immediately.
\end{proof}

\section{Potentials with prescribed singularities}\label{sec:pps}

Let $X$ be a compact unibranch Kähler space of pure dimension $n$. Let $\alpha$ be a big $(1,1)$-cohomology class. Let $\theta$ be a strongly closed smooth $(1,1)$-form in the class $\alpha$. Write $V$ for the volume of $\alpha$. Fix a resolution of singularity $\pi:Y\rightarrow X^{\Red}$.

For $\varphi,\psi\in \PSH(X,\theta)$, define
\[
[\varphi]\land\psi\coloneqq \sups_{\!\!\!\! C\in \mathbb{R}} \left((\varphi+C)\land \psi\right)  .
\]
This is usually denoted by $P[\varphi](\psi)$.

\subsection{Potentials with prescribed singularities}
Let $\phi\in \PSH(X,\theta)$ be a model potential with mass $V_{\phi}>0$. Fix a resolution of singularities $\pi:Y\rightarrow X^{\Red}$. Recall that we may always assume that $Y$ is Kähler.

\begin{definition}
	Define the relative full mass class as
	\[
	\mathcal{E}(X,\theta;[\phi])\coloneqq \left\{\,\varphi\in \PSH(X,\theta):[\varphi]\preceq[\phi], \int_X \theta_{\varphi}^n= \int_X \theta_{\phi}^n\,\right\}\,.
	\]
\end{definition}
Recall that $[\varphi]\preceq [\phi]$ means that there is $C\in \mathbb{R}$ such that $\varphi\leq \phi+C$.
\begin{remark}
	It is easy to see that under \cref{prop:pshbij},
	\[
	\mathcal{E}(X,\theta;[\phi])=\mathcal{E}(X^{\Red},\theta;[\phi])\,.
	\]
	Pull-back induces a bijection
	\[
	\pi^*:\mathcal{E}(X,\theta;[\phi])\cn \mathcal{E}(Y,\pi^*\theta;[\pi^*\phi])\,.
	\]
    The same applies to $\mathcal{E}^p$ and $\mathcal{E}^{\infty}$ defined later.
	We will use these identifications without explicit mention.
\end{remark}

\begin{proposition}
	Let $\varphi\in \PSH(X,\theta)$. Then $\varphi\in \mathcal{E}(X,\theta;[\phi])$ if and only if  $[\varphi]\land V_{\theta}=\phi$.
\end{proposition}
\begin{proof}
	We may assume that $X$ is reduced. By \cref{prop:corr}, \cref{lma:roofpro}, it suffices to prove the corresponding result on $Y$, in which case, this is exactly \cite[Theorem~2.1]{DDNL19log}.
\end{proof}
\begin{lemma}\label{lma:comp}
	Let $\gamma_j\in \mathcal{E}(X,\theta;[\phi])$ ($j=1,\ldots,j_0\leq n$) and $\varphi,\psi\in \mathcal{E}(X,\theta;[\phi])$. Then
	\[
	\int_{\{\varphi<\psi\}}\,\theta_{\psi}^{n-j_0}\wedge \theta_{\gamma_1}\wedge\cdots \wedge \theta_{\gamma_{j_0}}\leq \int_{\{\varphi<\psi\}}\,\theta_{\varphi}^{n-j_0}\wedge \theta_{\gamma_1}\wedge\cdots \wedge \theta_{\gamma_{j_0}}.
	\]
\end{lemma}
\begin{proof}
	We may assume that $X$ is reduced. By \cref{prop:corr}, it suffices to prove the corresponding result on $Y$, in which case, this is \cite[Corollary~3.16]{DDNL18mono}.
\end{proof}

\begin{proposition}\label{prop:fundam}
	Let $\varphi,\psi,\gamma\in \mathcal{E}(X,\theta;[\phi])$. Assume that $\varphi\leq \psi\leq \gamma$. Then
	\[
	\int_X (\gamma-\psi)^p\,\theta_{\psi}^n\leq 2^{n+p}\int_X (\gamma-\varphi)^p\,\theta_{\varphi}^n\,.
	\]
\end{proposition}
\begin{remark}
    Note that our measure $\theta_{\psi}^n$ does not charge any pluripolar set. The function $\gamma-\psi$ is well-defined outside a pluripolar set, so the integral $\int_X (\gamma-\psi)^p\,\theta_{\psi}^n$ makes sense. We will omit this kind of argument from now on.
\end{remark}
\begin{proof}
	We may assume that $X$ is a Kähler manifold.
	
	Let $S=\varphi^{-1}(-\infty)$.
	Observe that for any $t\geq 0$,
	\begin{equation}\label{eq:ppgcomp}
		\{\gamma>\psi+2t\}\setminus S\subseteq \{(\gamma+\psi)/2>\psi+t\}\setminus S\subseteq \{(\gamma+\psi)/2>\varphi+t\}\setminus S\subseteq \{\gamma>\varphi+t\}\setminus S\,.
	\end{equation}
	So
	\[
	\begin{aligned}
		\int_X (\gamma-\psi)^p\,\theta_{\psi}^n
		=& 2^p\int_X pt^{p-1} \int_{\{\gamma-\psi>2t\}}\,\theta_{\psi}^n\,\mathrm{d}t\\
		\leq & 2^p\int_X pt^{p-1} \int_{\{(\gamma+\psi)/2>\psi+t\}}\,\theta_{\psi}^n\,\mathrm{d}t && \text{By \eqref{eq:ppgcomp}}\\
		\leq  & 2^{n+p}\int_X pt^{p-1} \int_{\{(\gamma+\psi)/2>\varphi+t\}}\,\theta_{(\gamma+\psi)/2}^n\,\mathrm{d}t\\
		\leq & 2^{n+p}\int_X pt^{p-1} \int_{\{(\gamma+\psi)/2>\varphi+t\}}\,\theta_{\varphi}^n\,\mathrm{d}t && \text{By \cref{lma:comp}}\\
		\leq &2^{n+p}\int_X pt^{p-1} \int_{\{\gamma>\varphi+t\}}\,\theta_{\varphi}^n\,\mathrm{d}t &&\text{By \eqref{eq:ppgcomp}}\\
		=& 2^{n+p} \int_X (\gamma-\varphi)^p\,\theta_{\varphi}^n&& \text{By \cref{lma:comp}}\,.
	\end{aligned}
	\]
\end{proof}
\begin{remark}\label{rmk:gener}
	\cref{prop:fundam} is a direct generalization of the fundamental inequality of Guedj--Zeriahi (\cite[Lemma~2.3]{GZ07}). See also \cite[Lemma~2.4]{DDNL19log}. The same proof applies to a general weight function in $\mathcal{W}^+_M$ (See~\cite{Dar15} for the precise definition).
\end{remark}

\begin{proposition}\label{prop:intdom1}
	Let $\varphi,\psi,\gamma\in \mathcal{E}(X,\theta;[\phi])$. Assume that $\gamma\geq \varphi\lor\psi$.
	Then
	\[
	\int_X (\gamma-\varphi)^p\,\theta_{\psi}^n\leq 2^p\int_X (\gamma-\varphi)^p\,\theta_{\varphi}^n+2^p\int_X (\gamma-\psi)^p\,\theta_{\psi}^n\,.
	\]
\end{proposition}
\begin{proof}
	We may assume that $X$ is a Kähler manifold. 
	
	Observe that
	\begin{equation}\label{eq:gvp2}
		\{\gamma>\varphi+2t\}\subseteq \{\gamma>\psi+t \}\cup \{\psi>\varphi+t\}\,.
	\end{equation}
	So
	\[
	\begin{aligned}
		\int_X (\gamma-\varphi)^p\,\theta_{\psi}^n
		=& 2^p\int_0^{\infty} pt^{p-1}\int_{\{\gamma>\varphi+2t\}}\,\theta_{\psi}^n\,\mathrm{d}t \\
		\leq& 2^p\int_0^{\infty} pt^{p-1}\int_{\{\gamma>\psi+t\}}\,\theta_{\psi}^n\,\mathrm{d}t+
		2^p\int_0^{\infty} pt^{p-1}\int_{\{\psi>\varphi+t\}}\,\theta_{\psi}^n\,\mathrm{d}t && \text{By \eqref{eq:gvp2}}\\
		\leq & 2^p\int_X (\gamma-\psi)^p \,\theta_{\psi}^n+2^p\int_0^{\infty} pt^{p-1}\int_{\{\psi>\varphi+t\}}\,\theta_{\varphi}^n\,\mathrm{d}t &&\text{By \cref{lma:comp}}\\
		\leq & 2^p\int_X (\gamma-\psi)^p \,\theta_{\psi}^n+2^p\int_0^{\infty} pt^{p-1}\int_{\{\gamma>\varphi+t\}}\,\theta_{\varphi}^n\,\mathrm{d}t\\
		=&2^p\int_X (\gamma-\psi)^p \,\theta_{\psi}^n+2^p\int_X (\gamma-\varphi)^p \,\theta_{\varphi}^n\,.
	\end{aligned}
	\]
\end{proof}

\subsection{Relative $\mathcal{E}^{\infty}$ spaces}

\begin{definition}
	\[
	\mathcal{E}^{\infty}(X,\theta;[\phi])=\left\{\,\varphi\in \PSH(X,\theta): \varphi-\phi\in L^{\infty}(X)\,\right\}\,.
	\]
	We say that a potential $\varphi\in \mathcal{E}^{\infty}(X,\theta;[\phi])$ has \emph{relative minimal singularity} with respect to $[\phi]$.
\end{definition}
	Note that $\mathcal{E}^{\infty}(X,\theta;[\phi])\subseteq \mathcal{E}(X,\theta;[\phi])$
	by \cref{thm:mono}.

We have the following easy observation.
\begin{lemma}
	Each $\varphi\in \mathcal{E}(X,\theta;[\phi])$ is a decreasing limit of $\varphi^j\in \mathcal{E}^{\infty}(X,\theta;[\phi])$.
\end{lemma}
\begin{proof}
	It suffices to take $\varphi^j=\varphi\lor (\phi-j)$.
\end{proof}
We call $\varphi^j$ constructed in this way the \emph{canonical approximations} of $\varphi$.

\subsection{Relative $\mathcal{E}^p$ spaces}
Fix $p\in [1,\infty)$.
\begin{definition}\label{def:ep}
	Define
	\[
	\mathcal{E}^p(X,\theta;[\phi])\coloneqq \left\{\varphi\in \mathcal{E}(X,\theta;[\phi]): \int_X |\phi-\varphi|^p\,\theta_{\varphi}^n<\infty \right\}\,.
	\]
\end{definition}

\begin{proposition}\label{prop:stab}
	Let $\varphi_j,\gamma\in \Ep$ ($j\in \mathbb{Z}_{>0}$). Assume that $\varphi_j\leq \gamma$ for each $j$ and that $\varphi_j\to \varphi\in \PSH(X,\theta)$ in $L^1$-topology.
	Assume that there is a constant $A>0$ such that
	\[
	\int_X (\gamma-\varphi_j)^p \,\theta_{\varphi_j}^n\leq A\,.
	\]
	Then $\varphi\in \Ep$ and
	\[
	\int_X (\gamma-\varphi)^p\,\theta_{\varphi}^n\leq 2^{n+2p+1}A\,.
	\]
\end{proposition}
\begin{proof}
	We may assume that $X$ is reduced. 
	
	\textbf{Step 1}. Assume that $\varphi_j$ is decreasing. In this case, we prove 
	\[
	\int_X (\gamma-\varphi)^p\,\theta_{\varphi}^n\leq 2^{p+1}A\,.
	\]
	By \cref{prop:intdom1}, for any $j, k$,
	\[
	\int_X (\gamma-\varphi_j)^p \,\theta_{\varphi_k}^n\leq 2^{p+1}A\,.
	\]
	For any $C>0$,
	\[
	\int_X \left(\gamma- \varphi_j \lor (\gamma-C)\right)^p\,\theta_{\varphi_k}^n\leq \int_X (\gamma-\varphi_j)^p \,\theta_{\varphi_k}^n\leq 2^{p+1}A\,.
	\]
	Let $k\to \infty$, by \cref{thm:semic}, we find
	\[
	\int_X \left(\gamma- \varphi_j \lor (\gamma-C)\right)^p\,\theta_{\varphi}^n\leq 2^{p+1}A\,.
	\]
	Let $j\to \infty$, by the monotone convergence theorem,
	\[
	\int_X \left(\gamma- \varphi \lor (\gamma-C)\right)^p\,\theta_{\varphi}^n\leq 2^{p+1}A\,.
	\]
	Then we let $C\to \infty$, again by the monotone convergence theorem,
	\[
	\int_X \left(\gamma- \varphi\right)^p\,\theta_{\varphi}^n\leq 2^{p+1} A\,.
	\]
	
	\textbf{Step 2}. In general, let $\psi^j=\sups_{\!\!\! k\geq j} \varphi_k$.
	For each $C>0$, let
	\[
	\psi^{j,C}=\psi^j \lor (\gamma-C),\quad \varphi^C=\varphi \lor (\gamma-C)\,.
	\]
	Observe that $\psi^{j,C}$ decreases to $\varphi^C$ as $j\to \infty$. Moreover, $\gamma\geq \psi^{j,C} \geq \psi^j\geq \varphi_j$.
	By \cref{prop:fundam},
	\[
	\int_X (\gamma-\psi^{j,C})^p\,\theta_{\psi^{j,C}}\leq 2^{n+p}A\,.
	\]
	
	By Step 1, 
	\begin{equation}\label{eq:inproof1}
		\int_X(\gamma-\varphi^C)^p\,\theta_{\varphi^C}^n\leq 2^{n+2p+1}A\,.
	\end{equation}
	In particular,
	\[
	\int_{\{\varphi>\gamma-C\}}(\gamma-\varphi^C)^p\,\theta_{\varphi}^n\leq 2^{n+2p+1}A\,.
	\]
	Let $C\to \infty$, by the monotone convergence theorem,
	\[
	\int_X (\gamma-\varphi)^p\,\theta_{\varphi}^n\leq 2^{n+2p+1}A\,.
	\]
	
	In order to conclude that $\varphi\in \Ep$, we still have to prove that $\varphi\in \mathcal{E}(X,\theta;[\phi])$. In fact, by \eqref{eq:inproof1},
	\[
	\int_{\{\varphi\leq \gamma-C\}}\,\theta_{\varphi^C}^n\leq \frac{1}{C^p}\int_X (\gamma-\varphi^C)^p\,\theta_{\varphi^C}^n\leq 2^{n+2p+1}C^{-p}A\,.
	\]
	Then $\pi^*\varphi\in  \mathcal{E}(Y,\pi^*\theta;[\pi^*\phi])$ by \cite[Lemma~3.4]{DDNL18mono}. Thus, $\varphi\in \mathcal{E}(X,\theta;[\phi])$.
\end{proof}

\begin{theorem}\label{thm:rooftopep}
	Let $\varphi,\psi\in \Ep$, then $\varphi\land\psi\in \Ep$.
\end{theorem}
The proof of the theorem is similar to that of  \cite[Theorem~2.13]{DDNL18fullmass}. We reproduce the proof for the convenience of the readers.
We prove some preliminary results at first.

\begin{lemma}\label{lma:rooftop}
	Assume that $X$ is a Kähler manifold.
	Let $\varphi,\psi\in \mathcal{E}^{\infty}(X,\theta;[\phi])$. Then there is $\gamma\in \mathcal{E}^{\infty}(X,\theta;[\phi])$ such that
	\begin{equation}\label{eq:temp8}
		\theta_{\gamma}^n=e^{\gamma-\varphi}\theta_{\varphi}^n+e^{\gamma-\psi}\theta_{\psi}^n\,.
	\end{equation}
\end{lemma}
The proof is similar to that of \cite[Lemma~2.14]{DDNL18fullmass}.
\begin{proof}
	For each $j\geq 1$, let $\varphi_j\coloneqq \varphi\lor(-j)$, $\psi_j\coloneqq \psi\lor(-j)$.
	Let
	\[
	\mu_j=e^{-\varphi_j}\theta_{\varphi}^n+e^{-\psi_j}\theta_{\psi}^n\,.
	\]
	By \cite[Theorem~5.3]{DDNL19log}, we can find $\gamma_j\in \mathcal{E}^{\infty}(X,\theta;[\phi])$ such that
	\begin{equation}\label{eq:temp4}
		\theta_{\gamma_j}^n=e^{\gamma_j}\mu_j\,.
	\end{equation}
	Take a constant $C>0$ so that $|\varphi-\psi|\leq 2C$.
	Let 
	\[
	\eta=\frac{\varphi+\psi}{2}-C-n\log 2\,.
	\]
	Then $\eta\in \mathcal{E}^{\infty}(X,\theta;[\phi])$ and $\theta_{\eta}^n\geq e^{\eta}\mu_j$.
	Hence, $\gamma_j\geq \eta$ by \cite[Lemma~5.4]{DDNL19log}. By the same lemma, $\gamma_j$ is decreasing in $j$, let $\gamma=\lim_{j\to\infty}\gamma_j$
	in the pointwise sense. 
	Then $\gamma\geq \eta$, hence $\gamma\in \mathcal{E}^{\infty}(X,\theta;[\phi])$. Now \eqref{eq:temp8} follows from \eqref{eq:temp4} by letting $j\to\infty$ using \cite[Theorem~2.3]{DDNL18mono}.
\end{proof}

\begin{proof}[Proof of \cref{thm:rooftopep}]
	We may assume that $\varphi,\psi\leq \phi$. For each $j\geq 1$, consider the canonical approximations:
	\[
	\varphi_j\coloneqq \varphi\lor(\phi-j)\,,\quad \psi_j\coloneqq \psi\lor(\phi-j)\,.
	\]
	By \cref{lma:rooftop}, we can take $\gamma_j\in \mathcal{E}^{\infty}(X,\theta;[\phi])$ solving the following equation:
	\[
	\theta_{\gamma_j}^n=e^{\gamma_j-\varphi_j}\theta_{\varphi_j}^n+e^{\gamma_j-\psi_j}\theta_{\psi_j}^n\,.
	\]
	It follows from \cite[Lemma~5.4]{DDNL19log} that $\gamma_j\leq \varphi_j\land \psi_j$.
	We claim that
	\begin{equation}\label{eq:claim1}
		\int_X (\phi-\gamma_j)^p\,\theta_{\gamma_j}^n\leq C\,.
	\end{equation}
	
	Assume the claim is true for now.
	We get immediately that $\sup_{X}(\gamma_j-\phi)\geq -C$.
	Hence, according to \cite[Lemma~2.2]{DDNL19log}, after possibly subtracting a subsequence, we may assume that $\gamma_j\to \gamma\in \PSH(X,\theta)$ in $L^1$-topology. Then $\gamma\in \Ep$ by \cref{prop:stab}. Moreover, since $\gamma_j\leq \varphi_j\land \psi_j$, we know that $\gamma\leq \varphi\land \psi$.
	In particular, $\varphi\land\psi\in \PSH(X,\theta)$. Now by \cref{prop:fundam}, $\varphi\land\psi\in \Ep$.

	Now we prove the claim.
	By symmetry, it suffices to prove
	\[
	\int_X (\phi-\gamma_j)^p e^{\gamma_j-\varphi_j}\,\theta_{\varphi_j}^n\leq C\,.
	\]
	But note that
	\[
	\int_X (\phi-\gamma_j)^p e^{\gamma_j-\varphi_j}\,\theta_{\varphi_j}^n\leq C\int_X (\phi-\varphi_j)^p e^{\gamma_j-\varphi_j}\,\theta_{\varphi_j}^n+C\int_X (\varphi_j-\gamma_j)^p e^{\gamma_j-\varphi_j}\,\theta_{\varphi_j}^n\,.
	\]
	But $x^pe^{-x}\leq C$ when $x\geq 0$, so it suffices to prove
	\[
	\int_X (\phi-\varphi_j)^p e^{\gamma_j-\varphi_j}\,\theta_{\varphi_j}^n\leq C\,.
	\]
	As $\gamma_j\leq \varphi_j$, it suffices to prove
	\begin{equation}\label{eq:inproof2}
		\int_X (\phi-\varphi_j)^p\,\theta_{\varphi_j}^n\leq C\,.
	\end{equation}
	It follows from the argument of \cite[Proposition~2.10]{BEGZ10} that
	\[
	\int_X (\phi-\varphi_j)^p\,\theta_{\varphi_j}^n\leq \int_X (\phi-\varphi)^p\,\theta_{\varphi}^n\,.
	\]
	Thus, \eqref{eq:inproof2} follows.
\end{proof}
\begin{corollary}\label{cor:convex}
	The space $\Ep$ is convex.
\end{corollary}
\begin{proof}
	Let $\varphi_0,\varphi_1\in \Ep$, for $t\in [0,1]$, let $\varphi_t=t\varphi_1+(1-t)\varphi_0$. Note that
	$\varphi_0\land \varphi_1\leq \varphi_t$.
	Since $\varphi_0\land\varphi_1\in \Ep$ by \cref{thm:rooftopep}. So $\varphi_t\in \Ep$ by \cref{prop:fundam}.
\end{proof}

\begin{corollary}\label{cor:rooftop}
	Let $\varphi,\psi\in \mathcal{E}^p(X,\theta;[\phi])$, then
	\[
	\theta_{\varphi\land \psi}^n\leq \mathds{1}_{\{\varphi\land\psi=\varphi\}}\theta_{\varphi}^n+\mathds{1}_{\{\varphi\land\psi=\psi\}}\theta_{\psi}^n.
	\]
	In particular, $\theta_{\varphi\land \psi}^n$ is supported on $\{\varphi\land\psi=\varphi\}\cup \{\varphi\land\psi=\psi\}$.
\end{corollary}
\begin{proof}
	We may assume that $X$ is a Kähler manifold. In this case, the assertion follows from \cref{thm:rooftopep} and \cite[Lemma~3.7]{DDNL18mono}.
\end{proof}

\subsection{Energy functionals}\label{subsec:energyfun}
\begin{definition}\label{def:enr}
	Let $\varphi,\psi\in \Ep$.
	\begin{enumerate} 
		\item Define
		$F_p(\varphi,\psi)\coloneqq \int_X |\varphi-\psi|^p\,\theta_{\varphi\land\psi}^n$.
		\item Assume $\varphi\leq \psi$, define $G_p(\varphi,\psi)\coloneqq \int_X (\psi-\varphi)^p\,\theta_{\psi}^n$.
		In general, define 
  \[
G_p(\varphi,\psi)=G_p(\varphi\land\psi,\varphi)+G_p(\varphi\land\psi,\psi).
\]
		\item Assume $\varphi\leq \psi$, $[\varphi]=[\psi]$, define
		\begin{equation}
			E_p(\varphi,\psi)\coloneqq \frac{1}{n+1}\sum_{j=0}^{n} \int_X (\psi-\varphi)^p \,\theta_{\psi}^j\wedge\theta_{\varphi}^{n-j}.
		\end{equation}
		Assume $\varphi\leq \psi$, define $E_p(\varphi,\psi)\coloneqq \sup_{\eta} E_p(\eta,\psi)\in (-\infty,\infty]$,
		where the $\sup$ is taken over $\eta\in \Ep$ such that $[\eta]=[\psi]$, $\varphi\leq\eta\leq\psi$.
		
		In general, define 
		\[
		E_p(\varphi,\psi)\coloneqq E_p(\varphi\land \psi,\varphi)+E_p(\varphi\land \psi,\psi).
		\]
	\end{enumerate}
\end{definition}

The functional $E_p$ is similar to the energy functional studied in \cite[Section~2.2]{BEGZ10}. As we will see in the next proposition, $E_p$ always takes finite values.

\begin{proposition}\label{prop:e}
	Let $\varphi,\psi,\gamma\in \Ep$, $\varphi\leq \psi\leq \gamma$.
	\begin{enumerate}
		\item $j\mapsto \int_X (\psi-\varphi)^p\,\theta_{\psi}^j\land \theta_{\varphi}^{n-j}$
		is decreasing in $j$.
		\item Let $\varphi',\psi'\in \Ep$ with $\varphi\leq \varphi'\leq \psi\leq \psi'$. Then
		\[
		E_p(\varphi',\psi)\leq E_p(\varphi,\psi)\leq E_p(\varphi,\psi').
		\]
		\item We have
		\begin{equation}\label{eq:temp7}
			G_p(\varphi,\psi)\leq E_p(\varphi,\psi)\leq F_p(\varphi,\psi)\leq (n+1)E_p(\varphi,\psi)<\infty.
		\end{equation}
		
		\item When $p>1$, for any $\varphi,\psi\in \Ep$, $\varphi\leq \psi$, then there is a constant $C=C(p,n,V_{\theta})>0$ such that, $E_p(\varphi,\psi)^{1/p}\geq C^{-1} E_1(\varphi,\psi)$.
		\item When $p=1$, we have
		\[
		E_1(\varphi,\gamma)=E_1(\varphi,\psi)+E_1(\psi,\gamma).
		\]
	\end{enumerate}
\end{proposition}
\begin{proof}
	(i) Write
	\[
	\int_X (\psi-\varphi)^p\,\theta_{\varphi}^j\land \theta_{\psi}^{n-j}=p\int_0^{\infty}t^{p-1}\int_{\{\psi>\varphi+t\}}\,\theta_{\varphi}^j\land \theta_{\psi}^{n-j}\,\mathrm{d}t.
	\]
	By \cref{lma:comp}, $j\mapsto \int_{\{\psi>\varphi+t\}}\,\theta_{\varphi}^j\land \theta_{\psi}^{n-j}$
	is decreasing in $j$.
	
	(ii) The proof follows from the same argument as that of \cite[Proposition~2.8(ii)]{BEGZ10}.
	
	(iii) When $[\varphi]=[\psi]$, this is a direct consequence of (i). In general, for each $j\geq 1$, let $\varphi_j=(\psi-j)\lor \varphi$. For any $j\geq 1$,
	\[
	\begin{aligned}
		\int_X(\psi-\varphi_j)^p\,\theta_{\varphi}^n
		=&\lim_{k\to\infty}\int_{\{\varphi>\psi-k\}}(\psi-\varphi_j)^p\,\theta_{\varphi_k}^n \\
		\leq &\lim_{k\to\infty}\int_{\{\varphi>\psi-k\}}(\psi-\varphi_k)^p\,\theta_{\varphi_k}^n\\
		\leq & (n+1)\lim_{k\to\infty} E_p(\varphi_k,\psi)\\
		\leq & (n+1)E_p(\varphi,\psi).
	\end{aligned}
	\]
	Let $j\to\infty$, by monotone convergence theorem, $F_p(\varphi,\psi)\leq (n+1)E_p(\varphi,\psi)$.
	Now observe that  $F_p(\varphi_k,\psi)\leq F_p(\varphi,\psi)$.
	In fact, for any $\epsilon>0$,
	\[
	\begin{aligned}
		\int_X (\psi-\varphi_k)^p\,\theta_{\varphi_k}^n
		\leq &\int_{\{\varphi<\psi-k\}} (\psi-\varphi_k)^p\,\theta_{\varphi_k}^n+\int_{\{\varphi>\psi-k-\epsilon\}} (\psi-\varphi_k)^p\,\theta_{\varphi}^n\\
		= & k^p\int_{\{\varphi<\varphi_k\}} \,\theta_{\varphi_k}^n+\int_{\{\varphi>\psi-k-\epsilon\}} (\psi-\varphi_k)^p\,\theta_{\varphi}^n\\
		\leq & k^p\int_{\{\varphi<\psi-k\}} \,\theta_{\varphi}^n+\int_{\{\varphi>\psi-k-\epsilon\}} (\psi-\varphi)^p\,\theta_{\varphi}^n && \text{By \cref{lma:comp}}\\
		\leq & \int_X (\psi-\varphi)^p\,\theta_{\varphi}^n+(k+\epsilon)^p\int_{\{\psi-k<\varphi<\psi-k+\epsilon\}}\theta_{\varphi}^n,
	\end{aligned}
	\]
 where the second term tends to $0$ as $\epsilon\to 0+$ by dominated convergence theorem.
	By the arguments of \cite[Proposition~2.10(ii)]{BEGZ10}, $E_p(\varphi,\psi)=\lim_{k\to\infty}E_p(\varphi_k,\psi)$.
	So
	\[
	E_p(\varphi,\psi)=\lim_{k\to\infty}E_p(\varphi_k,\psi)\leq \varlimsup_{k\to\infty} F_p(\varphi_k,\psi)\leq F_p(\varphi,\psi). 
	\]
	
	By Fatou's lemma,
	\[
	G_p(\varphi,\psi)\leq \varliminf_{k\to\infty}G_p(\varphi_k,\psi)\leq \lim_{k\to\infty}E_p(\varphi_k,\psi)=E_p(\varphi,\psi).
	\]
	
	Finally, let us prove that $F_p(\varphi,\psi)<\infty$. Take a constant $C_1>0$ such that $\psi<V_{\theta}+C_1$, then
	\[
	\int_X (\psi-\varphi)^p\,\theta_{\varphi}^n\leq \int_X (C_1+V_{\theta}-\varphi)^p\,\theta_{\varphi}^n\leq C+C\int_X (V_{\theta}-\varphi)^p\,\theta_{\varphi}^n <\infty.
	\]
	
	(iv) We may assume that $\varphi,\psi\in\Ei$. This is a consequence of H\"older's inequality.
	
	(v) Assume that in addition, $[\varphi]=[\psi]=[\gamma]$,  then this is a direct generalization of \cite[Theorem~4.10]{DDNL18mono}. One just needs the integration by parts formula \cref{thm:ibpA}. In general, one concludes by canonical approximations and a generalization of \cite[Theorem~10.37]{GZ17}.
\end{proof}

Recall that as in \cite[Section~4.2]{DDNL18mono}, one can define a functional $E^{\phi}:\PSH(X,\theta;[\phi])\rightarrow [-\infty,\infty)$ as follows:
\[
E^{\phi}(\varphi)\coloneqq \frac{1}{n+1}\sum_{j=0}^n \int_X (\varphi-\phi)\,\theta_{\varphi}^j\wedge \theta_{\phi}^{n-j}
\]
when $\varphi\in \Ei$ and $E^{\phi}(\varphi)\coloneqq \inf_{\psi} E^{\phi}(\psi)$
in general, where the $\inf$ is taken over $\psi\in \Ei$ such that $\varphi\leq \psi$.

Note that
\begin{equation}\label{eq:coc}
	E_1(\varphi,\psi)=E^{\phi}(\psi)-E^{\phi}(\varphi)
\end{equation}
when $\varphi\leq \psi$ and $\varphi,\psi\in \mathcal{E}^1(X,\theta;[\phi])$. \begin{remark}\label{rmk:ibp}
	In \cite{DDNL18mono}, the authors assumed in addition that $\phi$ has small unbounded locus to make sure that one can perform the integration by parts. Since the general integration by parts formula is proved in \cref{thm:ibpA}, we no longer need this assumption. One can easily check that all results in \cite{DDNL18mono} still hold in general by the same proof even in the unibranch setting. We will apply this remark without explicitly mention.
\end{remark}

\begin{proposition}\label{prop:f}
	Let $\varphi,\psi,\gamma,\varphi_j\in \Ep$ ($j\in \mathbb{Z}_{>0}$). 
	\begin{enumerate}
		\item Assume that $\gamma\geq \varphi\lor \psi$. Then 
		\begin{equation}\label{eq:fp}
			F_p(\gamma,\varphi)+F_p(\gamma,\psi)\geq F_p(\varphi,\psi)= F_p(\varphi,\varphi\land \psi)+ F_p(\psi,\varphi\land \psi)\,.
		\end{equation}
		
		\item  $F_p(\gamma\land\varphi,\gamma\land \psi)\leq F_p(\varphi,\psi)$.
		
		\item Assume that $\varphi\leq \psi$. Let $\varphi_t=(1-t)\varphi+t\psi$. Then for any $N>0$,
		\[
		\sum_{j=0}^{N-1}F_p(\varphi_{j/N},\varphi_{(j+1)/N})^{1/p}\leq F_p(\varphi,\psi)^{1/p}\,.
		\]
		\item Assume that $\varphi_j$ increases a.e. to $\varphi$, then $\lim_{j\to\infty} F_p(\varphi_j,\varphi)=0$,
	\end{enumerate}
\end{proposition}
\begin{proof}
	We may assume that $X$ is a Kähler manifold throughout the proof.
	
	(1) By \cref{cor:rooftop},
	\[
	\mathds{1}_{\{\varphi\neq \psi\}}\theta_{\varphi\land \psi}^n\leq \mathds{1}_A\theta_{\varphi}^n+\mathds{1}_B\theta_{\psi}^n\,,
	\]
	where $A=\{\varphi\land\psi=\varphi<\psi\}$, $B=\{\varphi\land\psi=\psi<\varphi\}$.
	Note that $A$ and $B$ are disjoint sets.
	
	In particular,
	\[
	F_p(\varphi,\psi)\leq \int_X |\varphi-\psi|^p\mathds{1}_A\,\theta_{\varphi}^n+\int_X |\varphi-\psi|^p\mathds{1}_B\,\theta_{\psi}^n\,.
	\]
	Observe that on the set $A$, we have $\varphi<\psi\leq \gamma$.
	So
	\[
	\int_X |\varphi-\psi|^p\mathds{1}_A\,\theta_{\varphi}^n\leq \int_X |\varphi-\gamma|^p\,\theta_{\varphi}^n\,.
	\]
	Similarly,
	\[
	\int_X |\varphi-\psi|^p\mathds{1}_B\,\theta_{\varphi}^n\leq \int_X |\psi-\gamma|^p\,\theta_{\psi}^n\,.
	\]
	The first inequality in \eqref{eq:fp} follows. For the second, we observe that
	\[
	\begin{aligned}
		F_p(\varphi,\varphi\land \psi)+F_p(\psi,\varphi\land \psi)
		=&\int_X \left((\varphi-\varphi\land\psi)^p+(\psi-\varphi\land\psi)^p\right)\,\theta_{\varphi\land\psi}^n\\
		= & \int_X \mathds{1}_{\{\varphi\land\psi=\varphi\}}\left((\varphi-\varphi\land\psi)^p+(\psi-\varphi\land\psi)^p\right)\,\theta_{\varphi\land\psi}^n\\
		&+\int_X \mathds{1}_{\{\varphi\land\psi=\psi<\varphi\}}\left((\varphi-\varphi\land\psi)^p+(\psi-\varphi\land\psi)^p\right)\,\theta_{\varphi\land\psi}^n && \text{By \cref{cor:rooftop}}\\
		= & \int_X \mathds{1}_{\{\varphi\land\psi=\varphi\}}|\psi-\varphi|^p\theta_{\varphi\land\psi}^n+\int_X \mathds{1}_{\{\varphi\land\psi=\psi<\varphi\}}|\psi-\varphi|^p\theta_{\varphi\land\psi}^n\\
		= & \int_X|\psi-\varphi|^p\theta_{\varphi\land\psi}^n && \text{By \cref{cor:rooftop}}\\
		= & F_p(\varphi,\psi)\,.
	\end{aligned}
	\]
	
	(2) By (1), we may assume that $\varphi\leq \psi$. Then
	\[
	\begin{aligned}
		F_p(\gamma\land\varphi,\gamma\land \psi)
		=&\int_X (\gamma\land \psi-\gamma\land\varphi)^p\,\theta_{\gamma\land\varphi}^n\\
		\leq & \int_{\{\gamma\land \varphi=\varphi\}} (\gamma\land \psi-\varphi)^p\,\theta_{\varphi}^n && \text{By \cref{cor:rooftop}}\\
		\leq & \int_{\{\gamma\land \varphi=\varphi\}} ( \psi-\varphi)^p\,\theta_{\varphi}^n\\
		\leq &\int_X ( \psi-\varphi)^p\,\theta_{\varphi}^n\\
		=& F_p(\varphi,\psi)\,.
	\end{aligned}
	\]
	
	(3) We observe that
	\[
	\sum_{j=0}^{N-1} \left(\int_X (\varphi_{(j+1)/N}-\varphi_{j/N})^p \,\theta_{\varphi_{j/N}}^n\right)^{1/p}
	= \frac{1}{N}\sum_{j=0}^{N-1} \left(\int_X (\psi-\varphi)^p \,\theta_{\varphi_{j/N}}^n\right)^{1/p}\,.
	\]
	So it suffices to find a uniform upper bound of the summand.

	In fact, 
	\[
	\begin{aligned}
		\int_X (\psi-\varphi)^p \,\theta_{\varphi_t}^n
		=&\int_X (\psi-\varphi)^p \,\left(t\theta_{\psi}+(1-t)\theta_{\varphi}\right)^n\\
		=&\sum_{j=0}^n \binom{n}{j}\int_X t^j(1-t)^{n-j}(\psi-\varphi)^p \,\theta_{\psi}^j\wedge \theta_{\varphi}^{n-j}\\
		\leq& \sum_{j=0}^n \binom{n}{j}\int_X t^j(1-t)^{n-j}(\psi-\varphi)^p \,\theta_{\varphi}^{n} && \text{By \cref{prop:e}}\\
		=&\int_X (\psi-\varphi)^p \,\theta_{\varphi}^{n}\,.
	\end{aligned}
	\]
	
	(4) We may assume that $\varphi\leq 0$.
	For each $C\geq 0$, let
	\[
	\varphi_j^C=\varphi_j \lor (\phi-C)\,,\quad \varphi^C=\varphi \lor (\phi-C)\,.
	\]
	By \cite[Theorem~2.3, Remark~2.5]{DDNL18mono}, $\lim_{j\to\infty} F_p(\varphi_j^C,\varphi^C)=0$. The remaining  proof is the same as \cite[Theorem~2.17]{BEGZ10}. We reproduce the proof for the convenience of the readers.
	
	Take a function $\chi:(-\infty,\infty]\rightarrow \mathbb{R}$ satisfying
	\begin{enumerate}
		\item $\chi$ is concave, continuous, decreasing, $\chi(0)=0$, $\chi(\infty)=\infty$.
		\item There is a constant $M>1$ such that $|t\chi'(t)|\leq M|\chi(t)|$ for all $t\geq 0$.
		\item  $\frac{t^p}{\chi(t)}$ decreases to $0$ as $t\to\infty$.
		\item 
		\[
		\int_X \chi\circ (\phi-\varphi_k) \theta_{\varphi_k}^n<\infty\,.
		\]
	\end{enumerate}
	The existence of such weight follows from the standard analysis, see \cite{GZ07} for example.
	
	Now we estimate
	\[
	\begin{aligned}
		\left|F_p(\varphi_j,\varphi)-F_p(\varphi_j^C,\varphi^C)\right|\leq & \left|\int_{\{\varphi_k\leq \phi-C\}} (\varphi-\varphi_k)^p \theta_{\varphi_k}^n-\int_{\{\varphi_k\leq \phi-C\}} (\varphi^C-\varphi_k^C)^p \theta_{\varphi_k^C}^n\right|\\
		\leq &\int_{\{\varphi_k\leq \phi-C\}} (\phi-\varphi_k)^p \theta_{\varphi_k}^n-\int_{\{\varphi_k\leq \phi-C\}} (\phi-\varphi_k^C)^p \theta_{\varphi_k^C}^n\\
		\leq &\frac{C^p}{\chi(C)}\left(\int_{\{\varphi_k\leq \phi-C\}} \chi\circ(\phi-\varphi_k) \theta_{\varphi_k}^n+\int_{\{\varphi_k\leq \phi-C\}} \chi\circ(\phi-\varphi_k^C) \theta_{\varphi_k^C}^n\right)\\
		\leq & C_0\frac{C^p}{\chi(C)} \quad \text{By \cref{prop:fundam} and \cref{rmk:gener}}\,.
	\end{aligned}
	\]
	Hence, we conclude.
\end{proof}

\begin{proposition}\label{prop:g}
	Let $\varphi,\psi,\gamma\in\Ep$, $\varphi\leq \psi$. Then
	\begin{enumerate}
		\item $G_p(\gamma\lor\varphi,\gamma\lor\psi)\leq G_p(\varphi,\psi)$.
		\item Let $\varphi_t=(1-t)\psi+t\varphi$. Then for any $N>0$,
		\[
		\sum_{j=0}^{N-1}G_p(\varphi_{j/N},\varphi_{(j+1)/N})^{1/p}\geq G_p(\varphi,\psi)^{1/p}\,.
		\]
	\end{enumerate}
\end{proposition}
\begin{proof}
	(1) We calculate
	\[
	\begin{aligned}
		G_p(\gamma\lor\varphi,\gamma\lor\psi)
		=& \int_X (\gamma\lor\psi-\gamma\lor\varphi)^p\,\theta_{\gamma\lor\psi}^n\\
		= & \int_{\{\gamma<\psi\}}(\gamma\lor\psi-\gamma\lor\varphi)^p\,\theta_{\psi}^n\\
		\leq & \int_X(\psi-\varphi)^p\,\theta_{\psi}^n\\
		=& G_p(\varphi,\psi)\,.
	\end{aligned}
	\]
	
	(2) The proof is similar to that of \cref{prop:f}, we omit it.
\end{proof}

\begin{lemma}\label{lma:refined}
	Let $\varphi,\psi\in\Ep$, assume that $\varphi\leq \psi$. Let $\varphi_t=t\psi+(1-t)\varphi$.
	Then for $N\geq 1$,
	\[
	\sum_{j=0}^{N-1} \left(F_p(\varphi_{j/N},\varphi_{(j+1)/N})^{1/p}-G_p(\varphi_{j/N},\varphi_{(j+1)/N})^{1/p}\right)\leq CN^{-1/p}\,,
	\]
	where $C>0$ depends on $\varphi$ and $\psi$.
\end{lemma}
\begin{proof}
	In fact, it suffices to estimate
	\[
	\sum_{j=0}^{N-1} \left(F_p(\varphi_{j/N},\varphi_{(j+1)/N})-G_p(\varphi_{j/N},\varphi_{(j+1)/N})\right)^{1/p}\,.
	\]
	We estimate each term
	\[
	\begin{split}
		F_p(\varphi_{j/N},\varphi_{(j+1)/N})-G_p(\varphi_{j/N},\varphi_{(j+1)/N})
		=&\frac{1}{N^p}\int_X (\varphi_1-\varphi_0)^p(\theta_{\varphi_{j/N}}^n-\theta_{\varphi_{(j+1)/N}}^n)\\
		=&\frac{1}{N^{p+1}}\sum_{a=0}^{n-1}\int_X (\psi-\varphi)^p\ddc(\varphi-\psi) \wedge \theta_{\varphi_{j/N}}^a\wedge \theta_{\varphi_{(j+1)/N}}^{n-1-a}\,.
	\end{split}
	\]
	Let $C_0$ be a common upper bound for terms of the form:
	\[
	\left|\int_X (\psi-\varphi)^p\ddc(\varphi-\psi) \wedge \theta_{\varphi}^a\wedge \theta_{\psi}^{n-1-a}\right|\,.
	\]
	Expand $\theta_{\varphi_{j/N}}$ as a linear combination of $\theta_{\varphi}$ and $\theta_{\psi}$, we find immediately
	\[
	F_p(\varphi_{j/N},\varphi_{(j+1)/N})-G_p(\varphi_{j/N},\varphi_{(j+1)/N})\leq C C_0 N^{-1-p}\,.
	\]
\end{proof}

\section{Rooftop structures}\label{sec:rooftop}
\begin{definition}
	Let $E$ be a set. A \emph{pre-rooftop structure} on $E$ is a binary operator $\land:E\times E\rightarrow E$, satisfying the following axioms: for $x,y,z\in E$,
	\begin{enumerate}
		\item $x\land y=y\land x$.
		\item $(x\land y)\land z=x\land (y\land z)$.
		\item $x\land x=x$.
	\end{enumerate}
	We call $(E,\land)$ a pre-rooftop space. 
	
	A morphism between rooftop spaces $(E,\land)\rightarrow (E',\land')$ is a map $f:E\rightarrow E'$ such that
	\[
	f(x\land y)=f(x)\land' f(y)\,,\quad x,y\in E\,.
	\]
\end{definition}

A pre-rooftop structure $\land$ defines a partial order $\leq$ on $E$ as follows:
\[
x\leq y\quad \text{if and only if } \quad  x\land y=x\,.
\]
Here by abuse of notation, we use $\leq$ to denote the partial order.

In particular, it makes sense to talk about an increasing and decreasing sequences in $E$.

\begin{definition}\label{def:roof}
	Let $(E,d)$ be a metric space. A \emph{pre-rooftop structure} on $(E,d)$ is a pre-rooftop structure $\land$ on $E$. We say $(E,d,\land)$ is a pre-rooftop metric space. A morphism between pre-rooftop metric spaces $(E,d,\land)\rightarrow (E',d',\land')$ is a morphism $f:(E,\land)\rightarrow (E',\land')$, which is also distance decreasing. Let $\PRTCat$ be the category of pre-rooftop metric spaces.
	
	A \emph{rooftop structure} on $(E,d)$ is a pre-rooftop structure $\land$ on $E$ such that
	\begin{equation}\label{eq:rtdef}
		d(x\land z,y\land z)\leq d(x,y),\quad \forall x,y,z\in E.
	\end{equation}
	We call $(E,d,\land)$ a \emph{rooftop metric space}. 
	Morphisms between rooftop metric spaces is the same as morphisms between underlying pre-rooftop metric spaces.
	Let $\RTCat$ be the category of rooftop metric spaces.
	
	We say the rooftop structure $\land$ is \emph{$p$-strict} ($p\in [1,\infty)$) if the following \emph{Pythagorean formula} holds:
	\begin{equation}\label{eq:generalPyt}
		d(x,y)^p=d(x,x\land y)^p+d(y,x\land y)^p,\quad \forall x,y\in E.
	\end{equation}
	In this case, we also say $(E,d,\land)$ is a $p$-strict rooftop metric space.
\end{definition}
The name rooftop comes from the K\"ahler setting, where $\varphi\land \psi$ is known as the rooftop envelop of the quasi-psh functions $\varphi,\psi$ in the literature.

\begin{lemma}\label{lma:landdis}
	Let $(E,d,\land)\in \RTCat$. Let $x,y,x',y'\in E$, then
	\begin{equation}
		d(x\land y,x'\land y')\leq d(x,x')+d(y,y').
	\end{equation}
\end{lemma}
\begin{proof}
	We compute
	\[
	d(x\land y,x'\land y')\leq d(x\land y,x\land y')+d(x\land y',x'\land y')\leq  d(x,x')+d(y,y').
	\]
\end{proof}

\begin{proposition}\label{prop:rooftopcomp}
	Let $(E,d,\land)\in \RTCat$. Then $(E,d)$ is complete if and only if  both of the followings hold:
	\begin{enumerate}
		\item Each increasing Cauchy sequence converges.
		\item Each decreasing Cauchy sequence converges.
	\end{enumerate}
\end{proposition}
This is essentially an abstract version of \cite[Theorem~9.2]{Dar17}.
\begin{proof}
	The direct implication is trivial.
	
	Conversely, assume that both conditions are true.
	Let $x_j\in E$ ($j\geq 1$) be a Cauchy sequence. We want to prove that $x_j$ converges. By passing to a subsequence, we may assume that 
	\[
	d(x_j,x_{j+1})\leq 2^{-j}.
	\]
	For $k,j\geq 1$, let
	\[
	y_j^k\coloneqq x_k\land \cdots \land x_{k+j}.
	\]
	Then $(y_j^k)_j$ is decreasing, and
	\[
	d(y_j^k,y_{j+1}^k)\leq d(x_{k+j},x_{k+j+1})\leq 2^{-k-j}.
	\]
	So $(y^j_k)_j$ is a decreasing Cauchy sequence. Define
	\[
	y^k\coloneqq \lim_{j\to\infty}y_j^k.
	\]
	Then
	\[
	d(y^k,y^{k+1})=\lim_{j\to\infty}(y_{j+1}^k,y_j^{k+1})\leq d(x_k,x_{k+1})\leq 2^{-k}.
	\]
	So $y^k$ is an increasing Cauchy sequence. Let
	\[
	y\coloneqq \lim_{k\to\infty}y^k.
	\]
	Then
	\[
	d(y^k,x_k)=\lim_{j\to\infty}d(y^k_j,x_k)\leq \lim_{j\to\infty} d(y_{j-1}^{k+1},x_k).
	\]
	Note that
	\[
	d(y_{j-1}^{k+1},x_k)\leq d(y_{j-1}^{k+1},x_{k+1})+d(x_{k+1},x_k)\leq d(y_{j-1}^{k+2},x_{k+1})+2^{-k}.
	\]
	Hence,
	\[
	d(y^k,x_k) \leq 2^{-k}+\lim_{j\to\infty}d(y_{j-1}^{k+2},x_{k+1})\leq \lim_{j\to\infty}\sum_{r=k}^{j+k}d(x_r,x_{r+1})\leq 2^{1-k}.
	\]
	So $x_k$ converges to $y$.
\end{proof}

\begin{definition}
    A rooftop metric space $(E,d,\land)$ is \emph{locally complete} if for each $y\in E$, the subspace $\{y\in E:x\geq y\}$ is complete.
\end{definition}

Observe that $\{y\in E:x\geq y\}$ is a rooftop metric space with respect to the metric $d$ and the pre-rooftop structure $\land$. In particular, we find that 
\begin{corollary}\label{cor:loccompl}
    Let $(E,d,\land)\in \RTCat$. Then $(E,d)$ is locally complete if and only if  both of the followings hold:
	\begin{enumerate}
		\item Each increasing Cauchy sequence converges.
		\item Each decreasing Cauchy sequence that admits a lower bound in $E$ converges.
	\end{enumerate}
\end{corollary}

\begin{proposition}\label{prop:comp}
	Let $(E,d,\land)\in \RTCat$, let $i:(E,d)\rightarrow (\bar{E},\bar{d})$ be the metric completion of $(E,d)$.
	Then there is a unique rooftop structure $\bar{\land}$ on $(\bar{E},\bar{d})$, so that $i:(E,d,\land)\rightarrow (\bar{E},\bar{d},\bar{\land})$ is a morphism in $\RTCat$. 
\end{proposition}
\begin{proof}
	We first argue the existence. Let $(x_j), (y_j)$ be Cauchy sequences in $E$. For $j,k>0$, by \cref{lma:landdis},
	\[
	d(x_j\land y_j,x_k\land y_k)\leq d(y_j,y_k)+d(x_j,x_k)\,.
	\]
	Hence, $x_j\land y_j$ is also Cauchy. Now let $x,y\in \bar{E}$, represented by Cauchy sequences $(x_j), (y_j)$ in $E$, then we define
	\[
	x\bar{\land} y\coloneqq \lim_{j\to\infty} x_j\land y_j\,.
	\]
	We have to show that this is well-defined. Let $(x_j'), (y_j')$ be two other Cauchy sequences in $E$ representing $x,y$. By \cref{lma:landdis},
	\[
	d(x_j\land y_j,x_j'\land y_j')\leq d(y_j,y_j')+d(x_j,x_j')\,.
	\]
	Hence,
	\[
	\lim_{j\to\infty}d(x_j\land y_j,x_j'\land y_j')=0\,.
	\]
	Thus, $\bar{\land}$ is well-defined. We claim that $\bar{\land}$ is a rooftop structure. We only have to verify \eqref{eq:rtdef}. Let $x,y,z\in \bar{E}$, represented by Cauchy sequences $(x_j), (y_j), (z_j)$ in $E$. Then
	\[
	\bar{d}(x\bar{\land} z,y\bar{\land} z)=\lim_{j\to\infty} d(x_j\land z_j,y_j\land z_j)\leq \varliminf_{j\to\infty} d(x_j,y_j)=\bar{d}(x,y)\,.
	\]
	It is clear that $i$ becomes a morphism in $\RTCat$.
	
	Now we prove the uniqueness. Assume that we have a rooftop operator $\bar{\land}$ on $(\bar{E},\bar{d})$ such that $i$ becomes a morphism in $\RTCat$.
	Let $x,y\in \bar{E}$, represented by Cauchy sequences $(x_j), (y_j)$ in $E$. By \cref{lma:landdis},
	\[
	\bar{d}(x_j\bar{\land} y_j,x\bar{\land} y)\leq \bar{d}(x_j,x)+\bar{d}(y_j,y)\,.
	\]
	Thus,
	\[
	\lim_{j\to\infty} x_j\bar{\land} y_j=x\bar{\land} y\,. 
	\]
	As $i$ is a morphism in $\RTCat$, we find
	\[
	\lim_{j\to\infty} x_j\land y_j=x\bar{\land} y\,. 
	\]
\end{proof}
\begin{definition}
	Let $(E,d,\land)\in \RTCat$, we call $(\bar{E},\bar{d},\bar{\land})$ constructed in \cref{prop:comp} the \emph{completion} of $(E,d,\land)$.
\end{definition}

\begin{example}
	$\Ep$ is a $p$-strict rooftop metric space. 
\end{example}
This is \cref{thm:main1}.

\begin{example}
	Let $X$ be a compact Kähler manifold. Let $\omega$ be a Kähler form on $X$. Then $\mathcal{R}^p(X,\omega)$ is a complete $p$-strict rooftop metric space.
\end{example}
\begin{proof}
	The metric $d_p$ is constructed in \cite{DL18}:
	\[
	d_p(\ell^1,\ell^2)\coloneqq \lim_{t\to\infty}\frac{1}{t}d_p(\ell^1_t,\ell^2_t),\quad \ell^1,\ell^2\in \mathcal{R}^p\,.
	\]
	It is shown there that $(\mathcal{R}^p,d_p)$ is a complete metric space.
	
	Now we construct the rooftop structure: let $\ell^1,\ell^2\in \mathcal{R}^p$, define $\ell^1\land \ell^2$ as the maximal geodesic ray in $\mathcal{R}^p$ that lies below both $\ell^1$ and $\ell^2$. The proof of the  existence of $\ell$ is the same as \cref{lma:rooftopr1}, so we omit the details.
	In particular, $\ell^1\land\ell^2$ admits the following concrete description: for each $t>0$, let $(L^t_s)_{s\in [0,t]}$ be the geodesic from $0$ to $\ell^1_t\land\ell^2_t$. Then $(\ell^1\land\ell^2)_t$ is the limit in $\mathcal{E}^1$ of $L^t_s$ as $s\to\infty$.
	
	It is easy to verify that $\land$ is indeed a rooftop structure. That it is $p$-strict follows from the fact that $\mathcal{E}^p$ is $p$-strict.
\end{proof}

\begin{example}
	Let $X$ be a compact unibranch Kähler space of pure dimension. Let $\theta$ be a strongly closed smooth real $(1,1)$-form on $X$, representing a big class. Then $\mathcal{R}^1(X,\theta)$ is a complete $1$-strict rooftop metric space.
\end{example}
This is \cref{thm:main2}.

For the next example, we recall some related results from non-Archimedean geometry.
Let $L$ be an ample line bundle on $X$ and $\omega$ be a Kähler form in $c_1(L)$.
Let $X^{\NA}$ be the Berkovich analytification of $X$  with respect to the trivial norm on $\mathbb{C}$. There is a natural morphism of ringed spaces $X^{\NA}\rightarrow X$. Let $L^{\NA}$ be the pull-back of $L$ along this morphism.
Let $\mathcal{E}^{1,\NA}(L)$ be the space of non-Archimedean metrics of finite energy on $L^{\NA}$, see \cite[Section~5.2]{BJ18b}. Let $E:\mathcal{E}^{1,\NA}(L)\rightarrow \mathbb{R}$ be the Monge--Amp\`ere energy functional defined in \cite[Section~5.2]{BJ18b}. For $\varphi,\psi\in \mathcal{E}^{1,\NA}(L)$, $\varphi\leq \psi$, define
\[
d_1(\varphi,\psi)=E(\psi)-E(\varphi)\,.
\]
Recall that from \cite{BBJ15}, there is a canonical distance preserving embedding
\[
\iota:\mathcal{E}^{1,\NA}(L)\hookrightarrow \mathcal{R}^1(X,\omega),
\] 
It is not surjective in general (\cite[Example~6.10]{BBJ15}).
As in \cite{BBJ15}, there is a contraction $\Pi:\mathcal{R}^1(X,\omega)\rightarrow \mathcal{E}^{1,\NA}(L)$, given by $\Pi(\ell)=\ell^{\NA}$ such that
\[
\Pi\circ \iota(\phi)=\phi\,,\quad \phi\in \mathcal{E}^{1,\NA}\,.
\]
We may identify $\mathcal{E}^{1,\NA}(L)$ with a subset of $\mathcal{R}^1(X,\omega)$ through $\iota$, known as the set of \emph{maximal geodesic rays}.
Finally, recall that $\iota$ and $\Pi$ are both order preserving by definition.
\begin{example}\label{ex:E1NA}
	Let $X$ be a projective smooth scheme of finite type over $\mathbb{C}$. Let $L$ be an ample line bundle on $X$. Then $\mathcal{E}^{1,\NA}(L)$ is a complete $1$-strict rooftop metric space.
\end{example}
\begin{proof}
	\textbf{Step 1}.
	We first show that given $\phi,\psi\in\mathcal{E}^{1,\NA}(L)$, $\ell\coloneqq \iota(\phi)\land\iota(\psi)$ is maximal.
	Let $\ell'=\iota\circ \Pi(\ell)$.
	Then $\ell'\geq \ell$ by definition (\cite[Definition~6.5]{BBJ15}). Since both $\iota$ and $\Pi$ are order preserving, we have 
	\[
	\ell'\leq \iota(\phi)\,,\quad \ell'\leq \iota(\psi)\,.
	\]
	Thus, $\ell=\ell'$ and $\ell$ is maximal. 
	
	Note that the result also follows from the characterization of maximal geodesic rays in \cite{DX20}.
	
	\textbf{Step 2}. We define the rooftop structure as follows: let $\phi,\psi\in\mathcal{E}^{1,\NA}(L)$, define
	\[
	\phi\land \psi\coloneqq \iota^{-1}\left(\iota(\phi)\land\iota(\psi) \right)\,.
	\]
	It is easy to check that $\land$ is indeed a $1$-strict rooftop structure. It follows from \cite[Theorem~1.2]{DX20} that $\Pi$ is continuous, hence $\mathcal{E}^{1,\NA}(L)$ is identified with a closed subspace of $\mathcal{R}^1(X,\omega)$, hence complete.
\end{proof}

\section{Metric on \texorpdfstring{$\mathcal{E}^p$}{Ep} spaces}\label{sec:metep}

Fix $p\in [1,\infty)$. Let $X$ be a compact unibranch Kähler space of pure dimension $n$, $\alpha$ be a big $(1,1)$-cohomology class and $\theta$ be a strongly closed smooth $(1,1)$-form in the class $\alpha$. We use $V$ to denote the volume of $\alpha$. Let $\phi\in \PSH(X,\theta)$ be a model potential with mass $V_{\phi}>0$. Fix a resolution of singularities $\pi:Y\rightarrow X^{\Red}$. Recall that we may always assume that $Y$ is Kähler.

\subsection{Length elements}
\begin{definition}
	A \emph{length element} is a symmetric function $f:\Ep\times \Ep\rightarrow [0,\infty)$.
	A length element $f$ is \emph{good} if it satisfies the following conditions:
	\begin{enumerate}[label=\textbf{A.\arabic*},ref=A.\arabic*]
		\item\label{ax:a1} For $\varphi\in \Ep$, $f(\varphi,\varphi)=0$.
		\item\label{ax:a2} For $\varphi,\psi\in \Ep$, $f(\varphi,\psi)^p=f(\varphi\land\psi,\varphi)^p+f(\varphi\land\psi,\psi)^p$.
		\item\label{ax:a3} For $\varphi,\psi,\gamma\in \Ep$, $\gamma\geq \varphi\lor\psi$,
		$f(\varphi,\gamma)^p+f(\psi,\gamma)^p\geq f(\varphi,\psi)^p$.
		\item\label{ax:a4} For $\varphi,\psi,\gamma\in \Ep$, $f(\varphi\land\gamma,\psi\land\gamma)\leq f(\varphi,\psi)$.
		\item\label{ax:a5} For $\varphi_0,\varphi_1\in \Ep$, assume that $\varphi_0\leq \varphi_1$, set $\varphi_t=t\varphi_1+(1-t)\varphi_0$, then
		\[
		\varlimsup_{N\to\infty}\sum_{j=0}^{N-1} f(\varphi_{j/N},\varphi_{(j+1)/N})\leq f(\varphi_0,\varphi_1).
		\]
		\item\label{ax:a6} There is a constant $C=C(p)>0$, so that for any $\varphi,\psi\in \Ep$, $\varphi\leq \psi$, we have $C f(\varphi,\psi)\geq E_1(\varphi,\psi)$.
	\end{enumerate}
\end{definition}

\begin{theorem}
	The function $F_{p}^{1/p}$ is a good length element. 
	
	$G_p^{1/p}$ is a length element that satisfies the following condition:
	\begin{enumerate}[label=\textbf{A.\arabic*},ref=A.\arabic*]
		\setcounter{enumi}{6}
		\item \label{ax:a7} For $\varphi,\psi,\gamma\in \Ep$, $f(\varphi\lor\gamma,\psi\lor\gamma)\leq f(\varphi,\psi)$.
	\end{enumerate}
\end{theorem}
\begin{proof}
	The first part follows from \cref{prop:f}. The second part follows from \cref{prop:g}.
\end{proof}

\begin{definition}\label{def:Gamma}
	Let $\varphi_0, \varphi_1\in \Ep$.
	\begin{enumerate}
		\item Let $\Gamma(\varphi_0,\varphi_1)$ be the set of $\Psi=(\psi_0,\ldots,\psi_N)\in \Ep^{N+1}$ for various $N\geq 0$ (we call $N$ the \emph{length} of $\Psi$) such that
		\begin{enumerate}
			\item $\psi_0=\varphi_0$, $\psi_N=\varphi_1$.
			\item For each $j=0,\ldots,N-1$, either $\varphi_j\leq \varphi_{j+1}$ or $\varphi_j\geq \varphi_{j+1}$.
		\end{enumerate}
		\item When $\varphi_0\leq \varphi_1$, let $\Gamma_+(\varphi_0,\varphi_1)$ be the set of $\Psi=(\psi_0,\ldots,\psi_N)\in \Gamma(\varphi_0,\varphi_1)$ such that 
  \[
  \psi_0\leq \psi_1\leq \cdots\leq \psi_N.
  \]
		\item For $\Psi=(\psi_0,\ldots,\psi_N)\in \Gamma(\varphi_0,\varphi_1)$, $\gamma\in \Ep$, define
		\[
		\gamma\land \Psi\coloneqq (\gamma\land\psi_0,\ldots,\gamma\land\psi_N),\quad \gamma\lor \Psi\coloneqq(\gamma\lor\psi_0,\ldots,\gamma\lor\psi_N).
		\]
	\end{enumerate}
\end{definition}
We observe that in (iii) above when $\varphi_0\leq \varphi_1$ and $\Psi\in \Gamma_+(\varphi_0,\varphi_1)$, we have 
\[
\gamma\land \Psi\in \Gamma_+(\gamma\land \varphi_0,\gamma\land\varphi_1),\quad \gamma\lor \Psi\in \Gamma_+(\gamma\lor \varphi_0,\gamma\lor\varphi_1).
\]

\begin{definition}Let $\varphi_0,\varphi_1\in \Ep$.
	\begin{enumerate}
		\item 
		We say that $\Psi=(\psi_0,\ldots,\psi_N)\in \Gamma(\varphi_0,\varphi_1)$ is \emph{reduced} if for every $j=0,\ldots,N-1$, $\psi_j\neq \psi_{j+1}$. 
		\item 
		For a non-reduced $\Psi$, we can always delete repeating consecutive elements to get its \emph{reduction} $\tilde{\Psi}$.
		
		\item Let $\Psi=(\psi_0,\ldots,\psi_N)\in \Gamma(\varphi_0,\varphi_1)$. Assume that $\Psi$ is reduced. We say that $j\in [1,N-1]$ is a \emph{turning point} of $\Psi$ is one of the following is true
		\begin{enumerate}
			\item $\psi_j\geq \psi_{j-1}$, $\psi_{j}\geq \psi_{j+1}$.
			\item $\psi_j\leq \psi_{j-1}$, $\psi_{j}\leq \psi_{j+1}$.
		\end{enumerate}
		We say $j$ is a turning point of \emph{type 1} or \emph{type 2} respectively.
	\end{enumerate}
\end{definition}
Observe that a turning point cannot be both type 1 and type 2 as $\Psi$ is reduced.

\begin{definition}Let $f$ be a length element satisfying Condition~\ref{ax:a1}.
	Let $\varphi_0,\varphi_1\in \Ep$ and $\Psi=(\psi_0,\ldots,\psi_N)\in \Gamma(\varphi_0,\varphi_1)$.
	\begin{enumerate}
		\item When $\varphi_0\leq \varphi_1$ and $\Psi\in \Gamma_+(\varphi_0,\varphi_1)$, define
		\begin{equation}\label{eq:ellpsi0}
			\ell(\Psi)\coloneqq \sum_{j=0}^{N-1}f(\psi_j,\psi_{j+1}).
		\end{equation}
		\item When $\Psi$ is reduced, write the turning points of $\Psi$ as $i_1,\ldots,i_{S}$ for some $S\geq 0$. Let $i_0=0, i_{S+1}=N$. Then we define
		\begin{equation}\label{eq:ellpsi1}
			\ell(\Psi)\coloneqq \left(\sum_{j=0}^{S} \ell(\psi_{i_j},\psi_{i_j+1},\ldots,\psi_{i_{j+1}})^p   \right)^{1/p}.
		\end{equation}
		\item Define
		\begin{equation}\label{eq:ellpsi2}
			\ell(\Psi)\coloneqq \ell(\tilde{\Psi}).
		\end{equation}
		\item Define
		\[
		|\Psi|_{f^p}\coloneqq \sup_{j=0,\ldots,N-1} f(\psi_j,\psi_{j+1})^p.
		\]
	\end{enumerate}
\end{definition}
We write $|\Psi|=|\Psi|_{F_p}$. Note that $|\tilde{\Psi}|= |\Psi|$.

The following lemma is clear.
\begin{lemma}
	Let $\varphi_0,\varphi_1\in \Ep$, $\varphi_0\leq \varphi_1$, let $\Psi\in \Gamma_+(\varphi_0,\varphi_1)$, then the definitions of $\ell(\Psi)$ in \eqref{eq:ellpsi0} and in \eqref{eq:ellpsi2} coincide.
\end{lemma}

Fix a length element $f$ satisfying Condition~\ref{ax:a1}.
\begin{definition}\label{def:dpmetrics}
	Let $\varphi_0,\varphi_1\in \Ep$.
	\begin{enumerate}
		\item When $\varphi_0\leq \varphi_1$, define
		\begin{equation}\label{eq:dp1}
			d_p(\varphi_0,\varphi_1)\coloneqq \lim_{\delta\to 0+}\inf_{\substack{\Psi\in \Gamma_+(\varphi_0,\varphi_1)\\ |\Psi|<\delta}}\ell(\Psi).
		\end{equation}
		\item In general, define
		\begin{equation}			d_p(\varphi_0,\varphi_1)\coloneqq \left( d_p(\varphi_0\land \varphi_1,\varphi_0)^p+d_p(\varphi_0\land \varphi_1,\varphi_1)^p \right)^{1/p}.
		\end{equation}
	\end{enumerate}
\end{definition}

\begin{lemma}\label{lma:tri1}
	Let $\varphi_0,\varphi_1,\varphi_2\in \Ep$ with $\varphi_0\leq\varphi_1\leq \varphi_2$. Then
	\[
	d_p(\varphi_0,\varphi_2)\leq d_p(\varphi_0,\varphi_1)+d_p(\varphi_1,\varphi_2).
	\]
\end{lemma}
\begin{proof}
	This follows immediately from definition.
\end{proof}

\begin{definition}\label{def:mref}
	Let $\varphi_0,\varphi_1\in \Ep$. Assume $\varphi_0\leq \varphi_1$.
	Let $\Phi=(\psi_0,\ldots,\psi_N)\in \Gamma_+(\varphi_0,\varphi_1)$. For each $M\geq 1$, define
	\begin{equation}
		\Phi^{(M)}=(\psi_0^0,\ldots,\psi_0^{M-1},\psi_1^0,\ldots,\psi_1^{M-1},\ldots,\psi_{N-1}^0,\ldots,\psi_{N-1}^{M-1})\in \Gamma_+(\varphi_0,\varphi_1),    
	\end{equation}
	where
	\[
	\psi_j^k=\frac{M-k}{M}\psi_{j}+\frac{k}{M}\psi_{j+1},\quad j=0,\ldots,N-1;\,k=0,\ldots,M-1.
	\]
\end{definition}

\subsection{$d_p$-metrics}\label{subsec:dp}
From now on, we focus on the length element $F_p^{1/p}$. We write $\ell^F$, $\ell^G$ the function $\ell$ defined with respect to $F_p^{1/p}$, $G_p^{1/p}$ respectively. When we write $\ell$, we refer to $\ell^F$. 

Our $d_p$ will be defined relative to $F_p^{1/p}$. However, the length element $G_p^{1/p}$ helps to establish the triangle inequality through the proof of \cref{prop:twoop} below. It is of interest to know if it is possible to prove \cref{prop:twoop} only using $F_p^{1/p}$.

\begin{lemma}\label{lma:eqdef}
	Let $\varphi_0,\varphi_1\in \Ep$, $\varphi_0\leq \varphi_1$, then
	\[
	d_p(\varphi_0,\varphi_1)=\inf_{\Phi\in \Gamma_{+}(\varphi_0,\varphi_1)}\ell(\Phi).
	\]
	In particular, $d_p(\varphi_0,\varphi_1)\leq F_p(\varphi_0,\varphi_1)^{1/p}<\infty$.
\end{lemma}
\begin{proof}
	This follows from Condition~\ref{ax:a5}.
\end{proof}
\begin{lemma}\label{lma:gf}
	Let $\varphi_0,\varphi_1\in \Ep$. Assume $\varphi_0\leq \varphi_1$.
	Let $\Phi\in \Gamma_+(\varphi_0,\varphi_1)$.
	\begin{enumerate}
		\item For each $M\geq 1$,
		\[
		\ell^G(\Phi)\leq \ell^G(\Phi^{(M)})\leq \ell^F(\Phi^{(M)})\leq \ell^F(\Phi).
		\]
		\item When $M\to\infty$,
		\[
		\ell^F(\Phi^{(M)})-\ell^G(\Phi^{(M)})=\mathcal{O}(M^{-1/p}).
		\]
	\end{enumerate}
\end{lemma}
\begin{proof}
	(i) This follows from \cref{prop:f}, \cref{prop:g}.
	
	(ii) This follows from \cref{lma:refined}.
\end{proof}

\begin{proposition}\label{prop:twoop}
	Let $\varphi_0,\varphi_1,\gamma\in \Ep$, then
	\begin{enumerate}
		\item $d_p(\gamma\land \varphi_0,\gamma\land \varphi_1) \leq d_p(\varphi_0,\varphi_1)$.
		\item Assume that $\varphi_0\leq \varphi_1$, then $d_p(\gamma\lor \varphi_0,\gamma\lor \varphi_1) \leq d_p(\varphi_0,\varphi_1)$.
	\end{enumerate}
\end{proposition}
As far as the author knows, the second part is new even for K\"ahler classes.
\begin{proof}
	(i) We may assume that $\varphi_0\leq \varphi_1$. Then the statement follows from Condition~\ref{ax:a4}.
	
	(ii)\textbf{Step 1}. We prove this under the additional assumption that $[\varphi_0]=[\varphi_1]$.
	
	In this case, we take $A>0$ so that
	\begin{equation}\label{eq:varphidiffbdd}
		\varphi_1-A\leq \varphi_0.
	\end{equation}
	
	Let $\epsilon>0$. Take $\Phi=(\psi_0,\ldots,\psi_N)\in \Gamma_+(\varphi_0,\varphi_1)$ so that
	\[
	\ell(\Phi)<d_p(\varphi_0,\varphi_1)+\epsilon.
	\]
	For $M\geq 1$, as in \cref{def:mref}, let $\Phi^{(M)}\in \Gamma_+(\varphi_0,\varphi_1)$ be the refinement of $\Phi$ obtained by inserting the following points between $\psi_j$ and $\psi_{j+1}$ ($j=0,\ldots,N-1$):
	\[
	\psi_{j}^k\coloneqq \frac{k}{M}\psi_{j+1}+\frac{M-k}{M}\psi_{j},\quad k=0,\ldots,M-1.
	\]
	We claim that for $M$ large enough,
	\begin{equation}\label{eq:lFnearlF}
		0\leq \ell^F(\gamma\lor \Phi^{(M)})-\ell^G(\gamma\lor \Phi^{(M)})<\epsilon.
	\end{equation}
	Assume this for the time being, then
	by Condition~\ref{ax:a7}, 
	\begin{equation}\label{eq:ellGlor}
		\ell^G(\gamma\lor\Phi^{(M)})\leq \ell^G(\Phi^{(M)}).
	\end{equation}
	By \eqref{eq:lFnearlF}, \eqref{eq:ellGlor} and \cref{lma:gf}, for $M$ large enough,
	\[
	\ell^F(\gamma\lor \Phi^{(M)})< \ell^F(\Phi^{(M)})+\epsilon\leq \ell^F(\Phi)+\epsilon<d_p(\varphi_0,\varphi_1)+2\epsilon.
	\]
	Let $\epsilon\to 0+$, $d_p(\gamma\lor \varphi_0,\gamma\lor \varphi_1)\leq d_p(\varphi_0,\varphi_1)$.
	
	It remains to prove \eqref{eq:lFnearlF}, which is refined version of \cref{lma:gf}(ii).
 Clearly, it suffices to treat the case $N=1$. 
	In this case, we write $\gamma\lor \psi_{0}^k$ as $\eta_k$ and write $\psi^k$ for $\psi_0^k$. Let $S$ be the polar locus of $\psi^0$.
	By perturbation, we may assume that
	$\delta\coloneqq \inf_{X\setminus S}(\psi^M-\psi^0)>0$. Note that for $k=0,1,\ldots,M-1$,
	\begin{equation}
		\psi^{k+1}\geq \psi^k+\frac{\delta}{M}.
	\end{equation}
	Write $\ell^G=\ell^G(\eta_0,\ldots,\eta_M)$, and $\ell^F=\ell^F(\eta_0,\ldots,\eta_M)$ for simplicity.
	
	Observe that on $\{\gamma\leq \psi^k\}$ ($k=0,1,\ldots,M-2$),
	\begin{equation}\label{eq:etadiff}
		\eta_{k+2}-\eta_{k+1}=\eta_{k+1}-\eta_{k}.
	\end{equation}
    This relation allows us to relate $\ell^G$ and $\ell^F$.
	So
	\[
	\begin{aligned}
		\ell^G-&\left(\int_X (\eta_{M}-\eta_{M-1})^p\,\theta_{\eta_{M}}^n\right)^{1/p}\\
		\geq& \sum_{k=0}^{M-2} \left(\int_{\{\gamma\leq \psi^k\}} (\eta_{k+1}-\eta_k)^p\,\theta_{\eta_{k+1}}^n\right)^{1/p}\\
		=&\sum_{k=0}^{M-2} \left(\int_{\{\gamma\leq \psi^k\}} (\eta_{k+2}-\eta_{k+1})^p\,\theta_{\eta_{k+1}}^n\right)^{1/p} \quad \text{By \eqref{eq:etadiff}}\\
		=&\sum_{k=0}^{M-2} \left(\int_{X} (\eta_{k+2}-\eta_{k+1})^p\,\theta_{\eta_{k+1}}^n-\int_{\{\psi^k<\gamma<\psi^{k+2}\}} (\psi^{k+2}-\eta_{k+1})^p\,\theta_{\eta_{k+1}}^n\right)^{1/p}\\
		\geq &\ell^F-\sum_{k=0}^{M-2}\left(\int_{\{\psi^k<\gamma<\psi^{k+2}\}} (\psi^{k+2}-\eta_{k+1})^p\,\theta_{\eta_{k+1}}^n\right)^{1/p}-\left(\int_X (\eta_{1}-\eta_{0})^p\,\theta_{\eta_{0}}^n\right)^{1/p} \quad \text{By \cref{lma:calc1}}\\
		\geq &\ell^F-(M-1)^{1-1/p}\left(\sum_{k=0}^{M-2} \int_{\{\psi^k<\gamma<\psi^{k+2}\}} (\psi^{k+2}-\eta_{k+1})^p\,\theta_{\eta_{k+1}}^n\right)^{1/p}-\left(\int_X (\eta_{1}-\eta_{0})^p\,\theta_{\eta_{0}}^n\right)^{1/p}.
	\end{aligned}
	\]
 The last step follows from the power mean inequality.
	We estimate the second term, we write
	\[
	\int_{\{\psi^k<\gamma<\psi^{k+2}\}} (\psi^{k+2}-\eta_{k+1})^p\,\theta_{\eta_{k+1}}^n=\int_{\{\psi^k<\gamma<\psi^{k+1}\}}+\int_{\{\psi^{k+1}\leq\gamma<\psi^{k+2}\}}=:J_{1,k}+J_{2,k}.
	\]
	Then
	\[
	\begin{aligned}
		J_{1,k} \leq& \int_{\{\psi^{k}<\gamma<\psi^{k+1}\}} (\psi^{k+2}-\psi^{k+1})^p\,\theta_{\psi^{k+1}}^n\\
		\leq & CM^{-p}\int_{\{\psi^{k}<\gamma<\psi^{k+1}\}} \,\theta_{\psi^{k+1}}^n && \text{By \eqref{eq:varphidiffbdd}}.
	\end{aligned}
	\]
	Now observe that for any $t\geq 0$,
	\begin{equation}\label{eq:psietagammaineq}
		\left(\{\psi^{k+1}\leq\gamma<\psi^{k+2}\}\cap \{\psi^{k+2}-\eta_{k+1}>t\}\right) \setminus S\subseteq \left\{ \frac{\psi^{k+2}+\eta_{k+1}+\gamma}{3}>\psi^{k+1}+\frac{t}{3} \right\}.
	\end{equation}
	Note that outside $S$,
	\[
	\psi^{k+2}-\eta_{k+1}\leq \psi^{k+2}-\psi^{k+1}\leq \frac{A}{M}.
	\]
	Also observe that
	\begin{equation}\label{eq:complicatedcontain}
		\left\{\frac{A}{3M}\geq \frac{\psi^{k+2}+\eta_{k+1}+\gamma}{3}-\psi^{k+1}> 0\right\}\setminus S\subseteq \left\{ -\frac{A}{M}< \gamma-\psi^{k+1}\leq 2\frac{A}{M}\right\}.
	\end{equation}
	In fact, when $\gamma\geq \psi^{k+1}$ and on the complement of $S$, $\frac{A}{3M}\geq \frac{\psi^{k+2}+\eta_{k+1}+\gamma}{3}-\psi^{k+1}> 0$ implies that
	$\frac{A}{M}\geq 2\gamma-2\psi^{k+1}$, so $0\leq \gamma-\psi^{k+1}\leq \frac{A}{2M}$. On the other hand, when $\gamma<\psi^{k+1}$, the same inequality implies $0<-\psi^{k+1}+\frac{A}{M}+\gamma$, so $-\frac{A}{M}<\gamma-\psi^{k+1}<0$. This proves \eqref{eq:complicatedcontain}.
	
	So
	\[
	\begin{aligned}
		J_{2,k}= & \int_0^{\infty}pt^{p-1} \int_{\{ \psi^{k+2}-\eta_{k+1}>t \}}\mathds{1}_{\{\psi^{k+1}\leq\gamma<\psi^{k+2}\}} \,\theta_{\eta_{k+1}}^n \,\mathrm{d}t\\
		\leq & 3^p \int_0^{\frac{A}{3M}} pt^{p-1} \int_{\left\{ \frac{\psi^{k+2}+\eta_{k+1}+\gamma}{3}>\psi^{k+1}+t \right\}}\mathds{1}_{\{\psi^{k+1}\leq\gamma<\psi^{k+2}\}} \,\theta_{\eta_{k+1}}^n \,\mathrm{d}t && \text{By \eqref{eq:psietagammaineq}}\\
		\leq & C\int_0^{\frac{A}{3M}} pt^{p-1} \int_{\left\{ \frac{\psi^{k+2}+\eta_{k+1}+\gamma}{3}>\psi^{k+1}+t \right\}} \,\theta_{\frac{\psi^{k+2}+\eta_{k+1}+\gamma}{3}}^n \,\mathrm{d}t\\
		\leq &C\int_0^{\frac{A}{3M}} pt^{p-1} \int_{\left\{ \frac{\psi^{k+2}+\eta_{k+1}+\gamma}{3}>\psi^{k+1}+t \right\}} \,\theta_{\psi^{k+1}}^n \,\mathrm{d}t && \text{By \cref{lma:comp}}\\
		=& C \int_{\{\frac{A}{3M}\geq \frac{\psi^{k+2}+\eta_{k+1}+\gamma}{3}-\psi^{k+1}> 0\}}\left(\frac{\psi^{k+2}+\eta_{k+1}+\gamma}{3}-\psi^{k+1} \right)^p\,\theta_{\psi^{k+1}}^n \\
		\leq &CM^{-p} \int_{\{-\frac{A}{M}< \gamma-\psi^{k+1}\leq 2\frac{A}{M}\}}\,\theta_{\psi^{k+1}}^n && \text{By \eqref{eq:complicatedcontain}}.
	\end{aligned}
	\]
	So for $[A/\delta]+2\leq k\leq M-[2A/\delta]-1$,
	\[
	J_{2,k}\leq CM^{-p}\int_{\{\psi^{k-[A/\delta]-2}\leq \gamma\leq \psi^{k+[2A/\delta]+1}\}}\,\theta_{\psi^{k+1}}^n.
	\]
	Write $\theta_{\psi^{k+1}}$ as a combination of $\theta_{\psi^{0}}$ and $\theta_{\psi^{M}}$, we find
	\[
	\sum_{k=0}^{M-1}\left(J_{1,k}+J_{2,k}\right)\leq CM^{-p}.
	\]
	Putting all estimates together, we find $\ell^G-\ell^F\geq -C M^{-1/p}$.
	Hence, \eqref{eq:lFnearlF} holds.
	
	\textbf{Step 2}. For $k\geq 0$. Let $\varphi_1^k=(\varphi_0+k)\land \varphi_1$. Note that $\varphi_1^k$ increases to $\varphi_1$ almost everywhere. In fact, this is equivalent to say $[\varphi_0]\land\varphi_1=\varphi_1$. Obviously, $[\varphi_0]\land\varphi_1\leq\varphi_1$.
	For the other inequality, by domination principle(\cite[Proposition~3.11]{DDNL18mono}), it suffices to prove
	\[
	[\varphi_0]\land\varphi_1\geq\varphi_1,\quad  \theta_{[\varphi_0]\land\varphi_1}^n-a.e..
	\]
	But by \cite[Theorem~3.8]{DDNL18mono},
	\[
	\theta_{[\varphi_0]\land\varphi_1}^n\leq \mathds{1}_{\{[\varphi_0]\land\varphi_1=\varphi_1\}}\theta_{\varphi_1}^n.
	\]
	The inequality follows.

	By Step 1 and (i),
	\[
	d_p(\gamma\lor\varphi_0,\gamma\lor\varphi_1^k)\leq d_p(\varphi_0,\varphi_1^k)\leq d_p(\varphi_0,\varphi_1).
	\]
	Now by \cref{lma:tri1},
	\[
	d_p(\gamma\lor\varphi_0,\gamma\lor\varphi_1)\leq d_p(\gamma\lor\varphi_0,\gamma\lor\varphi_1^k)+d_p(\gamma\lor\varphi_1^k,\gamma\lor\varphi_1).
	\]
	So it suffices to prove that
	\[
	\lim_{k\to\infty}d_p(\gamma\lor\varphi_1^k,\gamma\lor\varphi_1)=0.
	\]
	By \cref{lma:eqdef},
	\[
	d_p(\gamma\lor\varphi_1^k,\gamma\lor\varphi_1)^p\leq \int_X  (\gamma\lor\varphi_1-\gamma\lor\varphi_1^k)^p\,\theta_{\gamma\lor\varphi_1^k}^n.
	\]
	The right-hand side tends to $0$ by \cref{prop:f}.
\end{proof}
\begin{lemma}\label{lma:calc1}
	Let $x,y,a,b\in [0,\infty)$. Let $p\in [1,\infty)$. Then
	\[
	\left(a^p+b^p+(x+y)^p\right)^{1/p}\leq (x^p+a^p)^{1/p}+(y^p+b^p)^{1/p}.
	\]
	In particular,
	\[
	(a+b)^{1/p}\leq a^{1/p}+b^{1/p}.
	\]
\end{lemma}
\begin{proof}
	We may assume that $x,y>0$. Let 
	\[
	F(a,b)\coloneqq \left(a^p+b^p+(x+y)^p\right)^{1/p}- (x^p+a^p)^{1/p}-(y^p+b^p)^{1/p}.
	\]
	Then
	\[
	\partial_a F(a,b)=a^{p-1}\left(a^p+b^p+(x+y)^p\right)^{1/p-1}-a^{p-1}(x^p+a^p)^{1/p-1}\leq 0.
	\]
	Similarly, $\partial_b F(a,b)\leq 0$
	So in order to prove that $F(a,b)\leq 0$, it suffices to prove this when $a=0$, $b=0$. But $F(0,0)=0$, so we are done.
\end{proof}

It is hard to check the triangle inequality from the definition of $d_p$ directly, so we provide an alternative definition.

\begin{definition}Let $\varphi_0,\varphi_1\in \Ep$, define
	\begin{equation}\label{eq:dp2}
		\tilde{d}_p(\varphi_0,\varphi_1)\coloneqq \inf_{\Psi\in \Gamma(\varphi_0,\varphi_1)} \ell(\Psi).
	\end{equation}
\end{definition}

\begin{proposition}\label{prop:dequald}
	For each $\varphi_0,\varphi_1\in \Ep$, we have $d_p(\varphi_0,\varphi_1)=\tilde{d}_p(\varphi_0,\varphi_1)$.
\end{proposition}
\begin{proof}
	By definition, $d_p(\varphi_0,\varphi_1)\geq \tilde{d}_p(\varphi_0,\varphi_1)$.
	
	For the other direction, fix some $\epsilon>0$, take $\Psi=(\psi_0,\ldots,\psi_N)\in \Gamma(\varphi_0,\varphi_1)$ such that
	\begin{equation}\label{eq:lprop}
		\ell(\Psi)< \tilde{d}_p(\varphi_0,\varphi_1)+\epsilon.
	\end{equation}
	We may take $\Psi$ with smallest $N$ so that \eqref{eq:lprop} is satisfied. In particular, $\Psi$ is reduced.
	
	\textbf{Step 1}. 
	We claim that we can always take $\Psi$ with the following additional property: there is $j\in [0,N]$, so that $\psi_k$ is decreasing for $k\leq j$ and increasing for $k\geq j$.
	
	It suffices to show that we may always assume that there are no type 1 turning points. Since then, it suffices to take $j$ to be the first turning point if there is any and take $j=0$ or $j=N$ otherwise.
	
	Assume that $j$ is the smallest type 1 turning point, let $i$ be the previous turning point if there is one, and let $i=0$ otherwise. Similarly, let $k$ be the next turning point if there is one, and $k=N$ otherwise. We claim that we may replace $\Psi$ with the reduction of
	\[
	\Psi'=(\psi_0,\ldots,\psi_{j-1},\psi_{j-1}\land \psi_{j+1},\psi_{j+1},\ldots,\psi_{N}).
	\]
	Now
	\[
	\begin{aligned}
		&\left(\sum_{a=i}^{j-2}f(\psi_a,\psi_{a+1})\right)^p+f(\psi_{j-1},\psi_{j-1}\land \psi_{j+1})^p+f(\psi_{j-1}\land \psi_{j+1},\psi_{j+1})^p+\left(\sum_{b=j+1}^{k-1}f(\psi_b,\psi_{b+1})\right)^p\\
		\leq & \left(\sum_{a=i}^{j-2}f(\psi_a,\psi_{a+1})\right)^p+f(\psi_{j-1},\psi_j)^p+f(\psi_j,\psi_{j+1})^p+\left(\sum_{b=j+1}^{k-1}f(\psi_b,\psi_{b+1})\right)^p \quad \text{By Condition~\ref{ax:a3},\ref{ax:a2}} \\
		\leq & \left(\sum_{a=i}^{j-1}f(\psi_a,\psi_{a+1})\right)^p+\left(\sum_{b=j}^{k-1}f(\psi_b,\psi_{b+1})\right)^p.
	\end{aligned}
	\]
    Here $f=F_p^{1/p}$.
	It follows that $\ell(\tilde{\Psi}')=\ell(\Psi')\leq \ell(\Psi)$.
	
	Observe that $\psi_{j-1}\land \psi_{j+1}$ is not equal to either $\psi_{j-1}$ or $\psi_{j+1}$, since otherwise, $\tilde{\Psi}'$ has length $N-1$, contrary to our choice of $\Psi$. So $j$ is a turning point of type 2 of $\Psi'$. In particular, if we modify the tuple $\Psi'$ further by decreasing some among $\psi_0,\ldots,\psi_{j-1}$, the index $j$ will never become a turning point of type 1.
	
	If some $j'\in [1,j-1]$ is a turning point of type 1 of $\Psi'$, we repeat the above procedure replacing $j$ by $j'$, and finally, we arrive at a new chain $\Psi''$ of length $N$ without any turning point of type 1 in the interval $[1,j]$, satisfying \eqref{eq:lprop}.
	If there is no turning points of type 1 in $[j+1,N-1]$, we are done, otherwise, repeat the same procedure.
	
	\textbf{Step 2}.
	Now we have
	\[
	\left(\ell(\Psi')^p+ \ell(\Psi'')^p\right)^{1/p}= \ell(\Psi)<\tilde{d}_p(\varphi_0,\varphi_1)+\epsilon,
	\]
	where $\Psi'=(\psi_j,\psi_{j-1},\ldots,\psi_0)$, $\Psi''=(\psi_j,\psi_{j+1},\ldots,\psi_{N})$.
	Hence,
	\[
	d_p(\psi_j,\varphi_0)^p+d_p(\psi_j,\varphi_1)^p\leq  \tilde{d}_p(\varphi_0,\varphi_1)^p.
	\]
	By \cref{prop:twoop}, 
	\[
	d_p(\varphi_0\land\varphi_1,\varphi_0)^p+d_p(\varphi_0\land\varphi_1,\varphi_1)^p\leq d_p(\psi_j,\varphi_0)^p+d_p(\psi_j,\varphi_1)^p.
	\]
	Hence, $d_p(\varphi_0,\varphi_1)\leq \tilde{d}_p(\varphi_0,\varphi_1)$.
\end{proof}

\begin{proposition}\label{prop:dpmetric}
	$d_p$ is a metric on $\Ep$.
\end{proposition}
\begin{proof}
	It is easy to see that $d_p$ is symmetric. It is finite by \cref{lma:tri1}. 
	
	\textbf{Step 1}. The triangle inequality. Let $\varphi,\psi,\gamma\in \Ep$. We need to prove 
    \[
        d_p(\varphi,\gamma)\leq d_p(\varphi,\psi)+d_p(\psi,\gamma).
    \]
	By \cref{prop:dequald}, it suffices to prove
	\begin{equation}\label{eq:temp1}
		\tilde{d}_p(\varphi,\gamma)\leq \tilde{d}_p(\varphi,\psi)+\tilde{d}_p(\psi,\gamma).
	\end{equation}
	
	Take $\epsilon>0$.
	Take $\Psi_1=(\psi_0,\ldots,\psi_N)\in \Gamma(\varphi,\psi)$, $\Psi_2=(\psi_N,\ldots,\psi_M)\in \Gamma(\psi,\gamma)$. We could assume that $N>0$, $M>N$.
	Let 
	\[
	\Psi=(\psi_0,\ldots,\psi_N,\psi_{N+1},\ldots,\psi_M)\in \Gamma(\varphi,\gamma).
	\]

	Let $i<N$ be last turning point of $\Psi_1$ if there is one. Let $i=0$ otherwise. Similarly, let $j>N$ be the first turning point of $\Psi_2$
	if there is one. Let $j=M$ otherwise. Let $A=\ell(\psi_0,\ldots,\psi_i)$, $B=\ell(\psi_j,\ldots,\psi_M)$.
	
	Then we claim that 
	\begin{equation}\label{eq:ellPsi}
		\ell(\Psi)\leq \ell(\Psi_1)+\ell(\Psi_2).
	\end{equation}
	Before proving this claim, let us observe that \eqref{eq:temp1} follows from this claim.
	
	In order to prove \eqref{eq:ellPsi}, we distinguish two cases.
	
	Case 1. $\psi_N$ is not a turning point of $\Psi$. Then
	\[
	\begin{aligned}
		\ell(\Psi)
		=& \left(A^p+\left(\sum_{k=i}^{N-1}f(\psi_k,\psi_{k+1})+\sum_{k=N}^{j-1}f(\psi_k,\psi_{k+1})\right)^p+B^p\right)^{1/p}\\
		\leq& \left(A^p+ \left(\sum_{k=i}^{N-1}f(\psi_k,\psi_{k+1})\right)^p\right)^{1/p}+\left(B^p+ \left(\sum_{k=N}^{j-1}f(\psi_k,\psi_{k+1})\right)^p\right)^{1/p}&& \text{By \cref{lma:calc1}}\\
		=& \ell(\Psi_1)+\ell(\Psi_2).
	\end{aligned}
	\]
 Here $f=F_p^{1/p}$.
	
	Case 2. $\Psi_N$ is a turning point of $\Psi$. Then
	\[
	\begin{aligned}
		\ell(\Psi)
		=& \left(A^p+\left(\sum_{k=i}^{N-1}f(\psi_k,\psi_{k+1})\right)^p+\left(\sum_{k=N}^{j-1}f(\psi_k,\psi_{k+1})\right)^p+B^p\right)^{1/p}\\
		\leq &\left(A^p+\left(\sum_{k=i}^{N-1}f(\psi_k,\psi_{k+1})\right)^p\right)^{1/p}+\left(\left(\sum_{k=N}^{j-1}f(\psi_k,\psi_{k+1})\right)^p+B^p\right)^{1/p}\\
		=& \ell(\Psi_1)+\ell(\Psi_2).
	\end{aligned}
	\]
	
	\textbf{Step 2}.
	We prove that $d_p$ is non-degenerate. Let $\varphi,\psi\in \Ep$, assume that $d_p(\varphi,\psi)=0$. We want to prove that $\varphi=\psi$.
	
	We may assume that $\varphi\leq \psi$.
	Let $\Psi=(\psi_0,\ldots,\psi_N)\in \Gamma_+(\varphi,\psi)$. By Condition~\ref{ax:a6},
	\[
	\ell(\Psi)
	\geq  C^{-1}E_1(\varphi,\psi)\geq C^{-1}F_1(\varphi,\psi).
	\]
	Thus, $F_1(\varphi,\psi)=0$.
	Hence, $\psi=\varphi$ by domination principle (\cite[Proposition~3.11]{DDNL18mono}).
\end{proof}

\begin{theorem}\label{thm:Eprooftop}
	$(\Ep,d_p,\land)$ is a $p$-strict rooftop metric space.
\end{theorem}
\begin{proof}
	By \cref{prop:dpmetric}, $d_p$ is a metric. 
	The fact that $\land$ is a rooftop structure follows from \cref{prop:twoop}.
	It follows from definition that $d_p$ is $p$-strict.
\end{proof}

\subsection{Properties of $d_1$}
\begin{proposition}\label{prop:d1energy}
	Let $\varphi,\psi\in \mathcal{E}^1(X,\theta;[\phi])$. Then
	\[
	d_1(\varphi,\psi)=E^{\phi}(\varphi)+E^{\phi}(\psi)-2E^{\phi}(\varphi\land\psi)\,.
	\]
\end{proposition}
\begin{proof}
	We may assume that $X$ is a compact Kähler manifold.
	
	By Pythagorean formula \eqref{eq:generalPyt}, we may assume that $\varphi\leq \psi$. 
 Take $\epsilon>0$.
 Let $\Psi=(\psi_0,\ldots,\psi_N)\in \Gamma_+(\varphi,\psi)$, so that
 \[
 \ell(\Psi)\leq d_1(\varphi,\psi)+\epsilon.
 \]
 By \cref{lma:gf}, we may assume that $\ell(\Psi)-\ell^G(\Psi)<\epsilon$.
 Then by \eqref{eq:coc},
	\[
	d_1(\varphi,\psi)+\epsilon\geq \ell(\Psi)
	\geq \sum_{j=0}^{N-1} E_1(\psi_j,\psi_{j+1})
	= \sum_{j=0}^{N-1} \left(E^{\phi}(\psi_{j+1})-E^{\phi}(\psi_j)\right)
	= E^{\phi}(\psi)-E^{\phi}(\varphi).
	\]
 On the other hand, 
 \[
 d_1(\varphi,\psi)-\epsilon\leq \ell^G(\Psi)\leq \sum_{j=0}^{N-1} E_1(\psi_j,\psi_{j+1})
	= E^{\phi}(\psi)-E^{\phi}(\varphi).
 \]
 As $\epsilon>0$ is arbitrary, we conclude that $d_1(\varphi,\psi)=E^{\phi}(\psi)-E^{\phi}(\varphi)$.
\end{proof}

\begin{lemma}
	Let $\varphi,\psi\in \mathcal{E}^1(X,\theta;[\phi])$. Then
	\[
	d_1\left(\varphi,\frac{\varphi+\psi}{2}\right)\leq \frac{3n+3}{2}d_1(\varphi,\psi)\,.
	\]
\end{lemma}
The argument is the same as \cite[Lemma~3.8]{DDNL18big}.
\begin{proposition}
	There is a constant $C>0$ such that for any $\varphi,\psi\in \mathcal{E}^1(X,\theta;[\phi])$,
	\[
	d_1(\varphi,\psi)\leq F_1(\varphi,\psi)\leq Cd_1(\varphi,\psi)\,.
	\]
\end{proposition}
\begin{proof}
	We may assume that $X$ is a compact Kähler manifold.
	
	We prove the left-hand inequality at first.
	\[
	d_1(\varphi,\psi)=\left(E^{\phi}(\varphi)-E^{\phi}(\varphi\land\psi)\right)+\left(E^{\phi}(\psi)-E^{\phi}(\varphi\land\psi)\right)\,.
	\]
	By symmetry, it suffices to deal with the first bracket.
	\[
	E^{\phi}(\varphi)-E^{\phi}(\varphi\land\psi)
	\leq  \int_X (\varphi-\varphi\land\psi)\,\theta_{\varphi\land\psi}^n
	\leq  \int_{\{\varphi\land\psi=\psi\}} (\varphi-\psi)\,\theta_{\psi}^n
	\leq  \int_X |\varphi-\psi|\,\theta_{\psi}^n\,.
	\]
	For the right-hand inequality, the argument is exactly the same as in \cite[Theorem~3.7]{DDNL18big}.
\end{proof}

\begin{proposition}\label{prop:sup}
	There is a constant $C>0$ such that for any $\varphi\in \mathcal{E}^1(X,\theta;[\phi])$,
	\[
	|\sup_X (\varphi-\phi)|\leq Cd_1(\varphi,\phi)+C\,.
	\]
\end{proposition}
\begin{proof}
	We may assume that $X$ is a compact Kähler manifold.
	
	In fact, we shall prove a stronger result
	\[
	-d_1(\varphi,\phi)\leq \sup_X (\varphi-\phi)\leq Cd_1(\varphi,\phi)+C\,.
	\]
	If $\sup_X(\varphi-\phi)\leq 0$, the right-hand inequality is trivial and
	\[
	-d_1(\varphi,\phi)=E^{\phi}(\varphi)\leq \sup_X (\varphi-\phi)\,.
	\]
	So we may assume that $\sup_X(\varphi-\phi)\geq 0$. In this case, the left-hand inequality is trivial. Recall that
	\[
	\theta_{\phi}^n\leq \mathds{1}_{\{\phi=V_{\theta}\}}\theta_{V_{\theta}}^n\leq \mathds{1}_{\{\phi=V_{\theta}=0\}}\theta^n\,.
	\]
	For a proof, see \cite[Theorem~3.8]{DDNL18mono}, \cite[Theorem~2.6]{DDNL18fullmass}.
	
	Take a Kähler form $\omega$ such that $\theta\leq \omega$
	Then by \cite[Lemma~2.2]{DDNL19log},
	\[
	\int_X \left|\varphi-\sup_X (\varphi-\phi)-\phi\right|\,\theta_{\phi}^n\leq C\,,
	\]
	where $C>0$ in independent of the choice of $\varphi$.

	Then
	\[
	\begin{aligned}
		d_1(\varphi,\phi)
		\geq & C^{-1}\int_X |\varphi-\phi|\,\theta_{\phi}^n\\
		\geq & C^{-1}\sup_X (\varphi-\phi)-C^{-1}\int_X \left|\varphi-\sup_X (\varphi-\phi)-\phi\right|\,\theta_{\phi}^n\\
		\geq & C^{-1}\sup_X (\varphi-\phi)-C\,.
	\end{aligned}
	\]
\end{proof}

\subsection{Properties of $d_p$}

\begin{proposition}\label{prop:dem} Let $\varphi^k$ be a sequence in $\Ep$, let $\varphi\in \Ep$, 
	\begin{enumerate}
		\item Assume that $\varphi^k$ is decreasing with pointwise limit $\varphi$. Let $\psi\in \Ep$, $\psi\leq \varphi$, then $\varphi\in \Ep$, $d_p(\psi,\varphi)=\lim_{k\to\infty}d_p(\psi,\varphi^k)$.
		\item Assume that $\varphi^k$ increases with $\varphi\coloneqq \sups \varphi^k\in \PSH(X,\theta)$. Let $\psi\in \Ep$, $\psi\geq \varphi$, then $\varphi\in \Ep$, $d_p(\psi,\varphi)=\lim_{k\to\infty}d_p(\psi,\varphi^k)$.
	\end{enumerate}
\end{proposition}
\begin{proof}
	(i) By \cref{prop:fundam}, $\varphi\in \Ep$.
	By \cref{prop:twoop}, $d_p(\psi,\varphi^k)$ is decreasing and is greater than $d_p(\psi,\varphi)$. So the limit exists and $d_p(\psi,\varphi)\leq\lim_{k\to\infty}d_p(\psi,\varphi^k)$.
	On the other hand, $d_p(\psi,\varphi^k)\leq d_p(\psi,\varphi)+d_p(\varphi,\varphi^k)$.
	So we need to prove $\lim_{k\to\infty} d_p(\varphi,\varphi^k)=0$.
	By \cref{lma:eqdef}, it suffices to show
	\[
	\left(\int_X (\varphi^k-\varphi)^p\,\theta_{\varphi}^n\right)^{1/p}\to 0,
	\]
	which follows from the dominated convergence theorem.
	
	(ii) As in (i), we only have to prove $\lim_{k\to\infty} d_p(\varphi,\varphi^k)=0$.
	Again, by \cref{lma:eqdef}, it suffices to prove $\lim_{k\to\infty}F_p(\varphi^k,\varphi)=0$.
	This follows from \cref{prop:f}.
\end{proof}

\begin{proposition}\label{prop:inc}
	Let $\varphi_j\in \Ep$ be a $d_p$-bounded increasing sequence. Then $\varphi_j$ converges to some $\varphi\in \Ep$ with respect to $d_p$.
\end{proposition}
\begin{proof}
	By \cref{prop:sup} and Condition~\ref{ax:a6}, $\varphi_j$ converges in $L^1$-topology to $\varphi\in \PSH(X,\theta)$. By Choquet's lemma and the fact that $\varphi_j$ is increasing, we conclude that $\varphi_j\to \varphi$ almost everywhere. Hence, $\varphi=\sups_{\!\!\! j} \varphi_j$.
	By \cref{prop:fundam}, $\varphi\in \Ep$. By \cref{prop:dem}, $d_p(\varphi_j,\varphi)\to 0$.
\end{proof}
Similarly, for decreasing sequences, we have
\begin{proposition}\label{prop:dec3}
    Let $\varphi_j\in \Ep$ be a $d_p$-bounded decreasing sequence. Assume that there is $\psi\in \Ep$ such that $\varphi_j\geq \psi$ for each $j$.
    Then $\varphi_j$ converges to some $\varphi\in \Ep$ with respect to $d_p$.
\end{proposition}

\begin{corollary}\label{cor:loccomp}
    $\Ep$ is a $p$-strict locally complete rooftop metric space.
\end{corollary}
\begin{proof}
    This follows from \cref{thm:Eprooftop}, \cref{prop:inc}, \cref{prop:dec3} and \cref{cor:loccompl}.
\end{proof}

Next we consider the problem of completeness. We do not expect $\Ep$ to be complete in general. However, in some useful cases, we show that $\Ep$ is indeed complete.

\begin{proposition}\label{prop:dec}
	Let $\varphi_j\in \mathcal{E}^1(X,\theta;[\phi])$ be a $d_1$-bounded decreasing sequence.  Then $\varphi_j$ converges to some $\varphi\in \mathcal{E}^1(X,\theta;[\phi])$ with respect to $d_1$.
\end{proposition}
\begin{proof}
	Let $\varphi=\infs_{j}\varphi_j$.
	By \cref{prop:sup}, $\varphi\in \PSH(X,\theta)$. Moreover, $\varphi_j\to\varphi$ in $L^1$-topology. 
	Since $\varphi_j$ is bounded in $\mathcal{E}^1(X,\theta;[\phi])$, we know that $E^{\phi}(\varphi_j)\geq -C$
	for some constant $C>0$. Hence, $\int_X |\varphi_j-V_{\theta}|\,\theta_{\varphi_j}^n\leq C$.
	So by \cref{prop:stab}, $\varphi\in \mathcal{E}^1(X,\theta;[\phi])$.
	Hence, by \cref{prop:dem}, $\varphi_j\to\varphi$ in $\mathcal{E}^1(X,\theta;[\phi])$.
\end{proof}
\begin{corollary}\label{cor:E1compl}
	$(\mathcal{E}^1(X,\theta;[\phi]),d_1)$ is complete.
\end{corollary}
\begin{proof}
	This follows from \cref{prop:rooftopcomp}, \cref{prop:inc} and \cref{prop:dec}.
\end{proof}
\cref{cor:E1compl} is also proved in \cite{DDNL18big} and \cite{Tru20}.

In order to proceed, we need 
\begin{condition}\label{conj:Gp_dom_by_dp}
Let $\varphi_0,\varphi_1\in \Ep$, $\varphi_0\leq \varphi_1$. Then
\[
G_p(\varphi_0,\varphi_1)^{1/p}\leq d_p(\varphi_0,\varphi_1).
\]
Equivalently, given $\psi_0\leq \psi_1\leq \cdots \leq \psi_N$ in $\Ep$, then
\[
\left(\int_X (\psi_N-\psi_0)^p\,\theta_{\psi_N}^n\right)^{1/p}\leq \sum_{j=0}^{N-1} \left(\int_X (\psi_{j+1}-\psi_j)^p\,\theta_{\psi_j}^n\right)^{1/p}.
\]
\end{condition}
This condition holds when $\phi=0$ and the cohomology class of $\theta$ is K\"ahler, as a consequence of \cite[Lemma~4.1]{Dar15}. As an immediate consequence, this condition also holds if $\theta$ is semi-positive or when the cohomology class of $\theta$ is nef if $\phi=0$.
This condition also holds if $p=1$ by \cref{prop:e}.

\begin{proposition}\label{prop:gpcl}
Assume that $\Ep$ satisfies \cref{conj:Gp_dom_by_dp}.
	Let $\varphi_0,\varphi_1\in \Ep$, $\varphi_0\leq \varphi_1$. Then
	\[
	F_p(\varphi_0,\varphi_1)^{1/p}\leq 2^{1+n/p} d_p(\varphi_0,\varphi_1). 
	\]
\end{proposition}
\begin{proof}	
	Let $\psi=(\varphi_0+\varphi_1)/2$. Then
	\[
	\begin{split}
		\left(\int_X (\varphi_1-\varphi_0)^p\,\theta_{\varphi_0}^n\right)^{1/p}\leq 2^{1+n/p}\left(\int_X (\psi-\varphi_0)^p\,\theta_{\psi}^n\right)^{1/p}\leq 2^{1+n/p} d_p(\varphi_0,\psi)\\
		\leq 2^{1+n/p} d_p(\varphi_0,\varphi_1),
	\end{split}
	\]
	where the last step follows from \cref{prop:twoop}.
\end{proof}
\begin{proposition}\label{prop:dec1}
Assume that $\Ep$ satisfies \cref{conj:Gp_dom_by_dp}.
	Let $\varphi_j\in \Ep$ be a $d_p$-bounded decreasing sequence.  Then $\varphi_j$ converges to some $\varphi\in \Ep$ with respect to $d_p$.
\end{proposition}
\begin{proof}
	We may assume that $\varphi_1\leq \phi$.
	
	By \cref{prop:dec}, we know that $\varphi\coloneqq \inf_j \varphi_j\in \mathcal{E}^{1}(X,\theta;[\phi])$ such that $\varphi_j\to \varphi$ in $\mathcal{E}^{1}(X,\theta;[\phi])$.
	
	Now by \cref{prop:gpcl} and the fact that $\varphi_j$ is $d_p$-bounded, we know that $\int_X (V_{\theta}-\varphi_j)^p\,\theta_{\varphi_j}^n\leq C$.
	So by \cref{prop:stab}, $\varphi\in \Ep$.
	The result follows from \cref{prop:dec3}.
\end{proof}

\begin{corollary}\label{cor:epc}

Assume that $\Ep$ satisfies \cref{conj:Gp_dom_by_dp}, then $(\Ep,d_p)$ is complete.
\end{corollary}
\begin{proof}
	This follows from \cref{prop:rooftopcomp}, \cref{prop:inc} and \cref{prop:dec1}. 
\end{proof}

Hence, we have proved the following:
\begin{theorem}\label{thm:main}
Assume that $\Ep$ satisfies \cref{conj:Gp_dom_by_dp}, then the space $(\Ep,d_p,\land)$ is a complete $p$-strict rooftop metric space.
\end{theorem}

\begin{proposition}
Assume that $\Ep$ satisfies \cref{conj:Gp_dom_by_dp}.
	There is a constant $C=C(p,n)>0$ such that for any $\varphi_0,\varphi_1\in \Ep$,
	\[
	d_p(\varphi_0,\varphi_1)\leq  F_p(\varphi_0,\varphi_1)^{1/p}\leq C d_p(\varphi_0,\varphi_1). 
	\]
\end{proposition}
\begin{proof}
	By \cref{prop:f} and Condition~\ref{ax:a2}, we may assume that $\varphi_0\leq \varphi_1$.
	
	The first inequality follows from \cref{lma:eqdef}.
	
	The second inequality follows from \cref{prop:gpcl}.
\end{proof}

\subsection{Relation with definitions in the literature}\label{subsec:lit}
In this section, we assume that $X$ is a Kähler manifold.

\subsubsection{Ample/big and nef classes without prescribed singularities}\label{subsubsec:relamp}
Assume that $\alpha$ is an ample class and $X$ is smooth. In this case, we may always take $\alpha$ to be a K\"ahler form $\omega$. Then the space $\Ep$ is usually written as $\mathcal{E}^p(X,\omega)$. The $d_p$-metric is defined in \cite{Dar15}. 

We recall the definition now. Let $\varphi_0,\varphi_1\in \mathcal{H}$. There is a unique weak geodesic $\varphi_t$ from $\varphi_0$ to $\varphi_1$. The geodesic has $C^{1,1}$-regularity(\cite{CTW18}). Then 
\begin{equation}\label{eq:olddp1}
	d_p(\varphi_0,\varphi_1)\coloneqq \left(\int_X |\dot{\varphi}_t|^p\,\omega_{\varphi_t}^n\right)^{1/p},
\end{equation}
for any $t\in [0,1]$. In particular, the right-hand side does not depend on the choice of $t$.

In general, let $\varphi_0,\varphi_1\in \mathcal{E}^p(X,\omega)$, take  decreasing sequences $(\varphi_j^k)_{k=1,2,\ldots}$ in $\mathcal{H}(X,\omega)$ with limits $\varphi_j$ for $j=0,1$. Then
\begin{equation}\label{eq:olddp2}
	d_p(\varphi_0,\varphi_1)\coloneqq \lim_{k\to\infty}d_p(\varphi_0^k,\varphi_1^k).
\end{equation}
The limit exists and is independent of the choice of $\varphi_j^k$.

\begin{proposition}\label{prop:coin1}
	The definition of $d_p$ in \eqref{eq:olddp1} and \eqref{eq:olddp2} coincides with $d_p$ defined in \cref{subsec:dp}.
\end{proposition}
\begin{proof}
	Let us write the metric defined in \cref{subsec:dp} as $D_p$ for the time being.
	
	Let $\varphi_0,\varphi_1\in \mathcal{E}^p(X,\omega)$. We want to prove
	\[
	d_p(\varphi_0,\varphi_1)=D_p(\varphi_0,\varphi_1).
	\]
	Since both sides satisfy the Pythagorean formula (\cref{cor:loccomp}, \cite[Theorem~3.26]{Dar19}), we may assume that $\varphi_0\leq \varphi_1$.
 Recall that order preserving simultaneous Demailly approximations of two potentials exist by the explicit construction in \cite{BK07}.
 Since both sides are continuous along Demailly approximations (\cref{prop:dem}, \cite[Proposition~4.9]{Dar15}), we may assume that $\varphi_0,\varphi_1\in \mathcal{H}$.

	Let $\varphi_t$ be the weak geodesic from $\varphi_0$ to $\varphi_1$.
	Now take $\Psi=(\varphi_{0/N},\ldots,\varphi_{N/N})\in \Gamma_+(\varphi_0,\varphi_1)$.
	Then
	\[
	\begin{aligned}
		\ell(\Psi)
		= & \sum_{j=0}^{N-1}\left(\int_X(\varphi_{(j+1)/N}-\varphi_{j/N})^{p}\,\omega_{\varphi_{j/N}}^n\right)^{1/p}\\
		\leq & \sum_{j=0}^{N-1}\left(\int_X(\dot{\varphi}_{j/N}N^{-1} +CN^{-2})^{p}\,\omega_{\varphi_{j/N}}^n\right)^{1/p} && \text{By } C^{1,1} \text{ regularity}\\
		\leq & \sum_{j=0}^{N-1}N^{-1}\left(\int_X(\dot{\varphi}_{j/N})^{p}\,\omega_{\varphi_{j/N}}^n\right)^{1/p}+CN^{-1}\\
		= & d_p(\varphi_0,\varphi_1)+CN^{-1},
	\end{aligned}
	\]
	Let $N\to\infty$, we find $D_p(\varphi_0,\varphi_1)\leq d_p(\varphi_0,\varphi_1)$.
	
	For the other inequality, let $\Psi=(\psi_0,\ldots,\psi_N)\in \Gamma_{+}(\varphi_0,\varphi_1)$. 
	We want to show that $d_p(\varphi_0,\varphi_1)\leq \ell(\Psi)$.
	By the triangle inequality of $d_p$. It suffices to prove that
	\[
	d_p(\psi_j,\psi_{j+1})\leq F_p(\psi_j,\psi_{j+1})^{1/p},\quad j=0,\ldots,N-1.
	\]
	This is \cite[Lemma~4.1]{Dar15}.
\end{proof}

The construction of $d_p$ has been extended to big and nef classes in \cite{DNL20}. A similar argument shows that our $d_p$ metric coincides with the definition in \cite{DNL20} as well.

\subsubsection{The case of $\mathcal{E}^1$}
The metric on $\mathcal{E}^1(X,\theta;[\phi])$ is defined in \cite{DDNL18big}, \cite{Tru20} We recall the definition. Let $\varphi,\psi\in \mathcal{E}^1(X,\theta;[\phi])$. Then
\begin{equation}\label{eq:d11}
	d_1(\varphi,\psi)\coloneqq E^{\phi}(\varphi)+E^{\phi}(\psi)-2E^{\phi}(\varphi\land\psi).
\end{equation}

\begin{proposition}
	The definition of $d_1$ in \eqref{eq:d11} coincides with the definition of $d_1$ in \cref{subsec:dp}.
\end{proposition}
This is nothing but \cref{prop:d1energy}.

\section{Spaces of geodesic rays}\label{sec:geod}
Let $X$ be a compact unibranch Kähler space of pure dimension $n$ and $\alpha$ be a big $(1,1)$-cohomology class. Let $\theta$ be a strongly closed smooth $(1,1)$-form in the class $\alpha$. Fix a resolution of singularities $\pi:Y\rightarrow X^{\Red}$. Recall that we always assume that $Y$ is Kähler.

\subsection{Definition of the metric}
Let $\mathcal{R}^1(X,\theta)$ be the space of geodesic rays in $\mathcal{E}^1(X,\theta)$ emanating from $V_{\theta}$.

We begin with a lemma:
\begin{lemma}
Let $(\ell_t)_{t\in [0,1]}$, $(\ell'_t)_{t\in [0,1]}$ be two geodesics in $\mathcal{E}^1(X,\theta)$. Then $d_1(\ell_t,\ell'_t)$ is a convex function in $t\in [0,1]$.
\end{lemma}
\begin{proof}
One immediately reduces to the smooth case, where it is proved in  \cite[Proposition~2.10]{DDNL19}.
\end{proof}

In particular, we can define
\begin{definition}
Let $\ell^1,\ell^2\in \mathcal{R}^1(X,\theta)$, set
\[
d_1(\ell^1,\ell^2)=\lim_{t\to\infty}\frac{1}{t}d_1(\ell^1_t,\ell^2_t)\,.
\]
\end{definition}
Note that $d_1(\ell^1,\ell^2)=d_1(\pi^*\ell^1,\pi^*\ell^2)$. In particular, $d_1$ is indeed a metric.
\begin{proposition}\label{prop:pullbackgeodray}
The pull-back induces an isometric isomorphism:
\[
\pi^*:(\mathcal{R}^1(X,\theta),d_1)\cn (\mathcal{R}^1(Y,\pi^*\theta),d_1)\,.
\]
\end{proposition}
\begin{proof}
This follows immediately from \cref{prop:corr}.
\end{proof}

\subsection{Construction of the rooftop structure}
Now we construct the rooftop structure on $\mathcal{R}^1(X,\theta)$.

\begin{lemma}\label{lma:rooftopr1}
Let $\ell^1,\ell^2\in \mathcal{R}^1(X,\theta)$, then there is a geodesic ray $\ell\in \mathcal{R}^1(X,\theta)$ such that $\ell\leq \ell^1$, $\ell\leq \ell^2$ that is maximal among all such rays. 
\end{lemma}
\begin{proof}
We may assume that $\ell^j_t\leq 0$ for all $t$.

\textbf{Step 1}. There is a geodesic ray $L\in \mathcal{R}^1(X,\theta)$, $L\leq \ell^1$, $L\leq \ell^2$.

For each $t>0$, let $L^t_s$ $(s\in [0,t])$ be the finite energy geodesic from $V_{\theta}$ to $\ell^1_t\land \ell^2_t$. Then for $0\leq t_1\leq t_2$,  $L^{t_2}_{t_1}\leq L^{t_1}_{t_1}$.
By symmetry, it suffices to prove $L^{t_2}_{t_1}\leq \ell^1_{t_1}$.
This follows from $\ell^1_{t_2}\land \ell^2_{t_2}\leq \ell^1_{t_2}$.
So for each $s\geq 0$, $L^t_{s}$ is decreasing in $t\geq s$. 
Now we claim that
\[
d_1(V_{\theta},\ell^1_s\land \ell^2_s)\leq Cs,\quad s\geq 0
\]
for a constant $C\geq 0$. In fact,
\[
d_1(V_{\theta},\ell^1_s\land \ell^2_s)\leq d_1(V_{\theta},\ell^1_s)+d_1(\ell^1_s,\ell^1_s\land \ell^2_s)\leq d_1(V_{\theta},\ell^1_s)+d_1(V_{\theta},\ell^2_s)\leq Cs\,.
\]

So for each fixed $s\geq 0$, $L^t_s$ $(t\geq s)$ is bounded in $\mathcal{E}^1(X,\theta)$. Let $L_s=\inf_{t\geq s} L^t_s$.
By \cite[Lemma~4.16]{Dar15} (applied to a resolution), $L_s\in \mathcal{E}^1(X,\theta)$ and $L_s$ is the $d_1$-limit of $L^t_s$ as $t\to \infty$. It is easy to see that $L\in \mathcal{R}^1(X,\theta)$. Then $L$ solves our problem.

\textbf{Step 2}. We show $L$ is maximal. Let $\ell\in \mathcal{R}^1(X,\theta)$. Assume that $\ell\leq \ell^1$, $\ell\leq \ell^2$. We claim that $\ell_s\leq L_s$ for any $s\geq 0$.
In fact, for each $t\geq s$, we have
\[
\ell_s\leq \ell^1_s\land \ell^2_s= L^s_s\leq L^t_s\,.
\]
Let $t\to \infty$, we find $\ell_s\leq L_s$.
\end{proof}
\begin{definition}
Let $\ell^1,\ell^2\in \mathcal{R}^1(X,\theta)$, we define
\begin{equation}\label{eq:roofray}
\ell^1\land \ell^2\coloneqq \sup\left\{\,\ell\in \mathcal{R}^1(X,\theta):\ell\leq \ell^1,\ell\leq \ell^2 \,\right\}\,.
\end{equation}
\end{definition}
\begin{theorem}\label{thm:R1comproof}
The space $(\mathcal{R}^1(X,\theta),d_1,\land)$ is a $1$-strict complete rooftop metric space. Moreover, $\pi^*:(\mathcal{R}^1(X,\theta),d_1)\rightarrow (\mathcal{R}^1(Y,\pi^*\theta),d_1)$ is an isomorphism of rooftop metric spaces.
\end{theorem}
\begin{proof}
It is clear that $\land$ is a rooftop structure. It follows from \cref{prop:pullbackgeodray} and \cite[Theorem~2.14]{DDNL19} that $(\mathcal{R}^1(X,\theta),d_1,\land)$ is a $1$-strict complete rooftop metric space. The second claim is obvious.
\end{proof}

\printbibliography

\bigskip
  \footnotesize

  Mingchen Xia, \textsc{Department of Mathematics, Institut de Mathématiques de Jussieu-Paris Rive Gauche}\par\nopagebreak
  \textit{Email address}, \texttt{mingchen@imj-prg.fr}\par\nopagebreak
  \textit{Homepage}, \url{https://mingchenxia.github.io/home/}.
  
\end{document}